\documentclass[11pt,abbrvbib]{article}
\usepackage{lscape}
\usepackage{jmlr2e}
\usepackage[utf8]{inputenc} 
\usepackage[T1]{fontenc}    
\usepackage{url}            
\usepackage{booktabs}       
\usepackage{amsfonts}       
\usepackage{nicefrac}       
\usepackage{microtype}      
\usepackage{microtype}
\usepackage{graphicx}
\usepackage{wrapfig}
\usepackage[small]{caption}
\usepackage{pgf}
\usepackage{subcaption}
\usepackage{booktabs} 
\usepackage{rotating}
\usepackage{wrapfig}
\def\ftime{f.time (sec)}
\def\Alg{Algorithm~\ref{alg:smoothing}}

\usepackage{times}
\usepackage{bm}
\usepackage{here}
\usepackage{amsmath}
\usepackage{amssymb}
\usepackage{algorithm,algorithmic}

\newcommand{\bA}{{\bm A}}
\newcommand{\A}{{\bm A}}
\newcommand{\B}{{\bm B}}
\newcommand{\C}{{\bm C}}
\renewcommand{\H}{{\bm H}}
\newcommand{\hwCw}{\hat{\w}^{\top}\C(\bar{\blambda})\hat{\w}}
\newcommand{\tred}{}

\newcommand{\bC}{\bm K}
\newcommand{\bff}{\bm f}
\renewcommand{\b}{b}
\newcommand{\diag}{\mathrm{diag}}

\newcommand{\argmin}{\operatornamewithlimits{argmin}}

\newcommand{\0}{\bm 0}
\newcommand{\temp}{\frac{2}{p-2}}
\newcommand{\mtemp}{\frac{2}{2-p}}
\newcommand{\redcolor}{}
\def\R{\mathbb{R}}
\def\algone{Alg.\ref{alg:smoothing}-A}
\def\algtwo{Alg.\ref{alg:smoothing}-B}
\def\dim{n+1}
\newcommand{\bxi}{{\bm \xi}}

\newcommand{\hzeta}{\hat{\zeta}}
\newcommand{\w}{\bm w} 
\newcommand{\blambda}{\bm \lambda}
\newcommand{\bzeta}{\bm \zeta}
\newcommand{\bW}{\bm W}
\newcommand{\bw}{\bm w}
\newcommand{\bphi}{\varphi}
\newcommand{\bbeta}{\bm \eta}
\newtheorem{LEMM}[theorem]{Lemma}
\newtheorem{assumption}[theorem]{Assumption}
\usepackage{ifthen}
\newcommand{\COMM}[2]{{
\ifthenelse{\equal{#1}{AT}}{\color{red}}{
\ifthenelse{\equal{#1}{TO}}{\color{blue}}{
\ifthenelse{\equal{#1}{AK}}{\color{green}}}}
[#1: #2]
}}
\title{
  On $\ell_p$-hyperparameter Learning via Bilevel Nonsmooth Optimization}
\author{
\name Takayuki Okuno \\
\email{takayuki.okuno.ks@riken.jp} \\
\addr
Center for Advanced Intelligence Project, RIKEN\\
Tokyo 103-0027, Japan \\
\AND 
\name Akiko Takeda \\
\email takeda@mist.i.u-tokyo.ac.jp \\
\addr 
Graduate School of Information Science and Technology, \\
The University of Tokyo\\
Tokyo 113-8656, Japan; \\
Center for Advanced Intelligence Project, RIKEN\\
Tokyo 103-0027, Japan \\
\AND 
\name Akihiro Kawana\\
\email
kawana.ak.pp@gmail.com
\\
\addr Department of Industrial Engineering and Economics,\\
Tokyo Institute of Technology\\
Tokyo 152-8550, Japan\\
(\tred{\scriptsize This research was conducted when he was a student at Tokyo Institute of Technology, and is completely irrevalent to the present company.} )
\AND 
\name Motokazu Watanabe\\
\email
mwatanabe@g.ecc.u-tokyo.ac.jp
\\
\addr
Department of Mathematical Informatics,\\
The University of Tokyo\\
Tokyo 113-8656, Japan;\\
\tred{Present Address: Tokyo Marine \& Nichido Fire Insurance Co., Ltd., Tokyo, Japan\\
  ({\scriptsize This research was conducted when he was a student at The University of Tokyo, and is completely irrevalent to the present company.} )}
}
\editor{}
\begin{document}
\maketitle
\begin{abstract}
  We propose a bilevel optimization strategy for selecting the best hyperparameter value for the nonsmooth $\ell_p$ regularizer with $0<p\le 1$. The concerned bilevel optimization problem has a nonsmooth, possibly nonconvex, $\ell_p$-regularized problem as the lower-level problem. Despite the recent popularity of nonconvex $\ell_p$-regularizer and  the usefulness of bilevel optimization for selecting hyperparameters, algorithms for such bilevel problems have not been studied because of the difficulty of $\ell_p$-regularizer.\\
  \indent Our contribution is the proposal of the first algorithm equipped with a theoretical guarantee for finding the best hyperparameter of $\ell_p$-regularized supervised learning problems. Specifically, we propose a smoothing-type algorithm for the above mentioned bilevel optimization problems and provide a theoretical convergence guarantee for the algorithm. Indeed, since optimality conditions are not known for such bilevel optimization problems so far,  new necessary optimality conditions, which are called the SB-KKT conditions, are derived and it is shown that a sequence generated by the proposed algorithm actually accumulates at a point satisfying the SB-KKT conditions under some mild assumptions.
The proposed algorithm is simple and scalable as our numerical comparison to Bayesian optimization and grid search indicates.
\end{abstract}
\section{Introduction}\label{introduction}
Hyperparameters are parameters that are set manually outside of a learning algorithm in the
context of machine learning. Hyperparameters often play important roles in exhibiting a high prediction performance. For example, a regularization parameter controls a trade-off between 
the regularization (i.e., model complexity)  and  the training set error (i.e., empirical error).
If the hyperparameters are tuned properly, the predictive performance of learning algorithms
will be increased. 

Hyperparameter optimization or learning is the task of finding
(near) optimal values of hyperparameters.
There are mainly a few methods currently in use for supervised learning. The most popular one would be grid search.
The method is to divide the space of possible hyperparameter values into regular intervals (a grid), train a learning model using training data for all values on the grid sequentially or preferably in parallel, and choose the best one with the highest prediction accuracy tested on validation data with e.g., using cross validation. 

There are other techniques for hyperparameter optimization; random search that evaluates learning models for randomly sampled hyperparameter values or more sophisticated method called Bayesian optimization \citep{Mockus1978}. To find a classifier/regressor with good prediction performance, it is reasonable to minimize the validation error in terms of hyperparameters. However we do not know
the explicit form of the validation error function, say $\bar{f}$, represented with hyperparameter, while we are able to compute the validation error of a classifier/regressor obtained with given hyperparameters $\blambda$, i.e., $\bar{f}(\blambda)$. For such a black-box (meaning unknown) objective function $\bar{f}$,
 Bayesian optimization algorithms use previous observational values $\bar{f}(\blambda)$ at some hyperparameter values $\bm \lambda$ to determine the next point $\bm \lambda^+$ to evaluate $\bar{f}(\blambda^+)$. This is based on the assumption that the function $\bar{f}$ is described by a Gaussian process as a prior.
 There are still essential questions unresolved;  how to select a kernel for the Gaussian process, 
 how to select the range of values to search in,  and lots of implementation details.
 As regards a comprehensive survey of hyperparameter optimization, we refer to \citep{feurer_hyperparameter_2018}.

Bilevel optimization is a more direct approach for finding a best set of hyperparameter values. A bilevel optimization problem
consists of two-level optimization problems; the upper-level problem minimizes the validation error in terms of hyperparameters and the
lower-level problem finds a best fit line for training data combined with a regularizer using given hyperparameter values.
Actually, the spirit of bilevel optimization underlies the methods introduced above. As mentioned below, some existing works pointed out the usefulness of the bilevel formulation for some classes of hyperparameters.
However, there is still a lot of room to pursue the bilevel approach further, in particular, for hyperparameter optimization of nonsmooth regularizers.
\subsection{Our Contribution}
The purpose of this paper is to provide a bilevel optimization approach for finding a best set of hyperparameter values for the nonsmooth $\ell_p$ ($p \le 1$) regularizer.
The nonsmooth bilevel optimization approaches examined here are entirely novel in the field of mathematical optimization too.
In recent years, research on sparse optimization using nonconvex nonsmooth regularizers 
has been actively conducted in  machine learning  \citep{gong2013general,hu2017group}, 
signal/image processing\,\citep{chen2010lower,hintermller2013nonconvex,wen2017nonconvex,marjanovic2013exact},
and continuous optimization \citep{ge2011note,lai2011unconstrained,bian2013worst,chen2014complexity,bian2015complexity,bian2017optimality}.
In particular, for the purpose of finding a sparse solution, the $\ell_p$-regularizer with $0<p<1$ is reported to be effective in wide applications such as matrix completion\,\citep{marjanovic2012l_q}, de-noising\,\citep{marjanovic2013exact}, compressing sensing\,\citep{zheng2016lp,wen2016robust,weng2016phase}, CT (computed tomography) reconstruction\,\citep{miao2016alternating}, machine learning\,\citep{xu2012l_} and so on.
Refer to the survey article\,\citep{wen2018survey} concerning nonconvex regularizers including the $\ell_p$-regularizer, and also see references therein.
In spite of plenty of researches supporting the efficiency of the $\ell_p$-regularizer, there exist fewer studies on bilevel optimization approaches to
hyperparameter learning or optimization of this regularizer.
A possible specific reason is that tractable optimality conditions for the arising bilevel problem have not been developed yet because of the $\ell_p$-regularizer's nonsmoothness or high nonconvexity when $p<1$.
Moreover, there are no practical algorithms ensured of convergence to a meaningful point for that nonsmooth bilevel problem.

Our contribution in this paper is the proposal of an algorithm with a theoretical convergence guarantee for solving an $\ell_p$-hyperparameter optimization problem, i.e., the problem of finding the best hyperparameter of an $\ell_p$-regularized supervised learning problem. 
Specifically, we first formulate it as a bilevel optimization problem with a nonsmooth and nonconvex lower-level problem having the $\ell_p$-regularizer. Since no optimality conditions have been explored adequately for such bilevel problems so far, we develop new optimality conditions, named scaled bilevel KKT (SB-KKT) conditions.
As a matter of fact, the SB-KKT conditions can be cast as an extension of the scaled first-order optimality conditions for some class of non-Lipschitz optimization problems originally given in
\citep{chen2010lower,chen2013optimality,bian2017optimality}.
We prove that these conditions are nothing but necessary optimality conditions for the one-level optimization problem acquired by replacing the lower-level problem with
its scaled first-order optimality conditions.
We moreover propose an iterative algorithm for solving the bilevel optimization problem with a nonsmooth and nonconvex lower-level problem.
One natural way for tackling such a problem would be formulating it as a one-level problem by replacing its lower-level-problem constraints with the first-order optimality condition formed
by the subdifferential (i.e., the set of subgradients) of the $\ell_p$-regularizer.
However, it is still nontrivial how we solve the resulting one-level problem with point-to-set mapping constraints.
To avoid this difficulty, we apply a smoothing technique for the $\ell_p$-regularizer, which enables us to make use of a gradient of the smoothed regularizer.
As a result, we obtain a one-level problem whose constraints are represented in terms of only smooth equations and inequalities.
In the presented smoothing algorithm, we generate a sequence of
KKT solutions of the smoothed problems while we control the degree of smoothing approximation.
     We will prove that a sequence generated by this algorithm accumulates at a point satisfying the SB-KKT conditions under some mild assumptions.
     Numerical experiments support the scalability of our algorithm compared to Bayesian optimization and grid search.
Finally, we discuss extension of the SB-KKT conditions and the proposed algorithm to other regularizers such as SCAD and MCP.
\subsection{Related Work on Bilevel Approach}
Most existing bilevel optimization models assume convexity and/or smoothness for all functions or at least once differentiability
for the lower-level objective functions.
If it is not once differentiable, we need to overcome the difficulty of selecting a subgradient to guarantee descent of the upper-level gradient
when solving such a problem.

\paragraph{Bilevel Formulations for Hyperparameter Opt.}
There are no existing works on bilevel hyperparameter optimization approach for our model and
existing works
are restricted to smooth and convex machine learning models.
A pioneer work in the line was by \citet{bennett+06,bennett+08}. They formulated
the selection technique of cross-validation for support vector regression as a bilevel optimization problem,
equivalently transformed it into a one-level nonconvex optimization problem whose constraints are the Karush-Kuhn-Tucker (KKT) optimality conditions of the lower-level problem and proposed two approaches to solve the nonconvex problem.
\citet{Moore2009,Moore2011} gave a bilevel optimization formulation for a nonsmooth and convex machine learning model, support vector regression (SVR), while
their proposed algorithms assume that the lower-level objective functions are at least once differentiable.
\cite{pmlr-v48-pedregosa16} gave a bilevel optimization formulation for more general supervised learning problems,
but the assumption of differentiability has been still imposed for  all functions.
More recently, \citet{franceschi2018bilevel} gave a unified bilevel perspective on hyperparameter optimization and meta learning, and presented a gradient-based algorithm using automatic differentiation techniques, which assumes the smoothness of the lower-level objective function.
\paragraph{Bilevel Optimization Algorithms}
As far as we investigated, the convergence analysis for bilevel problems with nonconvex nonsmooth regularizers has not been studied before. 
Many studies on bilevel optimization in optimization community
transform bilevel optimization problems into the one-level formulations
via the first-order optimality conditions
  for lower-level problems by assuming the differentiability of the functions, and focus on investigating theoretical properties for constraint qualifications and optimality conditions (see, for example,
\citep{ye1995optimality,dempe2006optimality,dempe2011generalized,Dempe2013,dempe2015bilevel}).
Recently, \cite{2016arXiv160207080O} proposed techniques for solving bilevel optimization problems with non-smooth ``convex'' lower level problems. 
They considered a gradient-based method for the optimization problem obtained by substituting a smoothly approximated solution mapping of the lower-level problem
  into the upper level problem.
  However, theoretical analysis concerning the limiting behavior of the derivatives of the approximated solution mappings was left to future work and the proposed method was written to be heuristic in the paper.
\cite{KunischPock2013} and \cite{Rosset2009} considered bilevel optimization problems
having the $\ell_p$-regularizer, which are similar to our problem, but the $p$ was mainly restricted to $1$ or $2$. 
Especially, the case of $p=0.5$ only appears in the numerical experiments in \cite{KunischPock2013} without any theoretical support, though
some convergence analysis is shown for the semismooth Newton algorithms for the case of $p=1$.

Another stream of bilevel algorithms is based on the reformulation
  as one-level problem by replacing the lower-level problem with a dynamical system, which arises in an iterative algorithm such as proximal gradient-type methods for solving the lower-level problem.
  The approach is employed for hyperparameter optimization in e.g., \citep{lorraine2020optimizing,franceschi2017forward,franceschi2018bilevel,maclaurin2015gradient,shaban2019truncated}. In their theoretical analysis, nonconvex and nonsmooth functions, which we will handle in this paper, are not supposed to be contained by the lower-level objective one.

\paragraph{Notations.}
In this paper, we often 
{denote a vector ${\bm z}\in \R^d$ by ${\bm z}=(z_1,z_2,\ldots,z_d)^{\top}$ and write} 
$\lim_{\ell\in L\to \infty}{\bm z}^{\ell}={\bm z}^{\ast}$ to represent that,
given a sequence $\{{\bm z}^{\ell}\}$, a subsequence $\{{\bm z}^{\ell}\}_{\ell\in L}$ with $L\subseteq\{1,2,\ldots,\}$
converges to ${\bm z}^{\ast}$.
The $\ell$-th vector ${\bm z}^{\ell} \in  \R^{d}$ is often represented as 
${\bm z}^{\ell}:=(z^\ell_1,z^\ell_2,\ldots,z_d^\ell)^{\top}$. 
We also denote 
the 
$d$-dimensional non-negative (positive) orthants by 
$\R_{+(++)}^d:=\{{\bm z}\in \R^d\mid z_i\ge(>)0\ (i=1,2,\ldots,d)\}$.
{For a set of vectors $\{\bm v_{i}\}_{i\in I}\subseteq \R^{m}$ with $I:=\{i_1,i_2,\ldots,i_p\}$,
we define $(\bm v_i)_{i\in I}:=(\bm v_{i_1},\bm v_{i_2},\ldots,\bm v_{i_p})\in \R^{m\times p}$.
}
We denote the sign function by 
${\rm sgn}:\R\to \{-1,0,+1\}$, i.e., 
${\rm sgn}(x):= 1~  (x>0)$, $0~ (x=0)$, and $-1~ (x<0)$ for any $x\in \R$.

{For a differentiable function $h:\R^n\to \R$, we denote the gradient function from $\R^n$ to $\R^n$ 
by $\nabla h$, i.e., $\nabla h({\bm x}):=(\frac{\partial h(\bm x)}{\partial x_1},\ldots,\frac{\partial h(\bm x)}{\partial x_n})^{\top}\in \R^n$ for $\bm x\in \R^n$, where $\frac{\partial h(\bm x)}{\partial x_i}$ stands for the partial differential of $h$ with respect to $x_i$ for $i=1,2,\ldots,n$. To express the gradient of $h$ with respect to a sub-vector $\tilde{\bm x}:=(x_i)_{i\in I}^{\top}$ of $\bm x$ with $I:=\{i_1,i_2,\ldots,i_p\}\subseteq \{1,2,\ldots,n\}$, we write $\nabla_{\tilde{\bm x}}h(\bm x):=
\left(\frac{\partial h(\bm x)}{\partial x_{i_1}},\frac{\partial h(\bm x)}{\partial x_{i_2}},\ldots,\frac{\partial h(\bm x)}{\partial x_{i_p}}\right)^{\top}\in \R^{|I|}$.
We often write $\nabla g(\bm x)|_{\bm x=\bar{\bm x}}$ $(\nabla_{\tilde{\bm x}} g(\bm x)|_{\bm x=\bar{\bm x}})$ or $\nabla h(\bar{\bm x})$ 
($\nabla_{\tilde{\bm x}} h(\bar{\bm x})$)
to represent the (partial) gradient value of $g$ at $\bm x=\bar{\bm x}$. Moreover, when $h$ is twice differentiable, 
we denote the Hessian of $h$ by $\nabla^2h:\R^n\to \R^{n\times n}$, i.e., $\nabla^2h(\bm x):=\left(\frac{\partial^2h(\bm x)}{\partial x_i\partial x_j}\right)_{1\le i,j\le n}\in \R^{n\times n}$. 

{\redcolor
  \paragraph{Organization of the Paper}
  The rest of this paper is organized as follows: In Section~\ref{sec:2}, we describe our problem setting precisely.
  In Section~\ref{sec:smooth}, we propose a smoothing algorithm for solving the targeted problem.
  In Section~\ref{sec:theory}, we present new necessary optimality conditions of the problem, already refereed to as the SB-KKT conditions.
  We also conduct the convergence analysis of the proposed smoothing algorithm.
  In Section~\ref{sec:numeric}, we examine the efficiency of the proposed algorithm by means of numerical experiments using real data sets.
  In Section~\ref{sec:discuss}, we discuss extension of the proposed algorithm to other classes of problems.
  Finally, in Section~\ref{sec:7}, we conclude this paper.
  In Appendix, we provide some proofs omitted in the main part together with other supplementary materials. 
}
\section{Formulation}\label{sec:2}
We consider the following {bilevel optimization problem with a nonsmooth, possibly nonconvex, lower-level problem}:
\begin{equation}
 \min_{\w^*_{\bm\lambda},\blambda}\ f(\w^*_{\bm\lambda})\hspace{0.5em}\mbox{s.t.}\hspace{0.5em}\displaystyle
  \w^*_{\bm\lambda} \in \argmin_{\bm w \in \R^n}\left(g(\bm w)+\sum_{i=1}^r \lambda_i R_i(\bm w)\right),\ {\bm \lambda \geq \bm 0}.
\label{eqn:bilevel}
\end{equation}
Suppose that $f:\R^n\to \R$ is once continuously differentiable,
$\blambda:=(\lambda_1,\lambda_2,\ldots,\lambda_r)^{\top}\in \R^r$,
$R_1(\bm w):=\|\bm w\|_p^p=\sum_{i=1}^n|w_i|^p$ ($0 < p \leq 1$), and
{the functions} $R_2,\cdots,R_r$, and $g$ are twice continuously differentiable functions.
We call the whole problem\,\eqref{eqn:bilevel} and $\min_{\w\in \R^n}g(\w)+\sum_{i=1}^r\lambda_iR_i(\w)$
the upper- and lower-level problem, respectively.
To make our notation simple, we often use
the function
\[G(\bm w,\bar{\bm \lambda}):=g(\bm w)+\sum_{i=2}^r \lambda_i R_i(\bm w),
\]
{with $\bar{\blambda}:=(\lambda_2,\ldots,\lambda_r)^{\top}\in \R^{r-1}$}
for expressing 
the lower-level problem as
\[
\min_{\bm w \in \R^n} G(\bm w,\bar{\bm \lambda})+\lambda_1 R_1(\bm w).
\]
Note that the function $R_1$ is nonconvex when $p<1$ and nonsmooth, though some differentiability is assumed
for other terms.
\tred{
  We also remark 
  that the proposed smoothing algorithm can be tailored to problems that contain multiple nonsmooth regularizers as long as suitable smoothing functions (see Section\,\ref{sec:smooth} for the definition) are found.
  Nonetheless, it is unclear whether the theoretical analysis can be established under such a setting. 
}
\subsection{Examples of Functions $g$, $\sum_{i=1}^r \lambda_i R_i$, and $f$}
When using the following loss function as the function $g$:
\begin{itemize}
\item $g(\bm w)=\sum_{i=1}^{m_{tr}} (\tilde{y}_i-\tilde{\bm x}_i^\top \bm w)^2$  for training samples $(\tilde{y}_i, \tilde{\bm x}_i) \in \R\times \R^n, i=1,\cdots,m_{tr}$
\item $g(\bm w)=\sum_{i=1}^{m_{tr}} \log(1+\exp(-\tilde{y}_i\tilde{\bm x}_i^\top \bm w ))$ 
for  training samples $(\tilde{y}_i, \tilde{\bm x}_i) \in \{+1,-1\}\times \R^n, i=1,\cdots,m_{tr}$,
\end{itemize}
the lower-level optimization problem in \eqref{eqn:bilevel} corresponds to  minimizing the
{$\ell_2$-loss function for regression}
and the {logistic-loss function for  binary classification}, respectively, combined with some regularization
including $\|\bm w\|_p^p$ for a given hyperparameter vector $\bm \lambda$.
This type of problem whose regularizer includes $\|\bm w\|_p^p$ is called a sparse optimization problem.
Various well-known sparse regularizers can be expressed by $\sum_{i=1}^r \lambda_i R_i(\bm w)$.
For example, 
\begin{itemize}
\item[\hspace*{3mm} $\star$]
  $\ell_1$ regularizer: $\lambda_1 \|\bm w\|_1$,
\item[\hspace*{3mm} $\star$]
  elastic net regularizer: $\lambda_1 \|\bm w\|_1+ \lambda_2 \|\bm w\|_2^2$, 
\item[\hspace*{3mm} $\star$]
  nonconvex regularizer: $\lambda_1 \|\bm w\|_q^q$ with $0 < q < 1$.
\end{itemize}

What we want to do is to find the best hyperparameter values of $\bm \lambda$ which lead to small validation error.
The upper-level problem can find such values for $\bm \lambda$.
By setting the same loss function with $g$ for $f$ but defined by validation samples  
$(\hat{y}_j, \hat{\bm x}_j)$, $j=1,\cdots,m_{\rm val}$,
the upper-level problem
finds the best hyperparameter values which
minimize the validation error, which is defined by 
$f(\bm w)=\sum_{i=1}^{m_{\rm val}} (\hat{y}_i-\hat{\bm x}_i^\top \bm w)^2$ for the $\ell_2$-loss or
  $f(\bm w)=\sum_{i=1}^{m_{\rm val}} \log(1+\exp(-\hat{y}_i\hat{\bm x}_i^\top \bm w ))$ for the logistic-loss.
\section{Smoothing Method for Nonconvex Nonsmooth Bilevel Program}\label{sec:smooth}
For problem \eqref{eqn:bilevel}, one may think of the one-level problem obtained by replacing the lower problem constraint with its first-order optimality condition \citep[10.1~Theorem]{rockafellar2009variational} represented in terms of (general) subgradient%
\footnote{{For precise definitions of a subgradient of a nonconvex function, see Appendix\,\ref{appendix:lemma} in this paper or Chapter~8 of \citep{rockafellar2009variational}.
}}, i.e., 
\begin{equation}
\min_{\w,\blambda}\ f(\w)\hspace{0.5em}\mbox{s.t.}\hspace{0.5em}
\bm 0\in \partial_{\w}(G(\w,\bar{\blambda})+\lambda_1R_1(\w)),\ \blambda \geq \bm 0.\label{eq:0615-2056}
\end{equation}
Notice that $G(\w,\bar{\blambda})+\lambda_1R_1(\w)$ is not convex with respect to $\w$ generally. Hence,
the feasible region of \eqref{eq:0615-2056} can be larger than
that of the original problem \eqref{eqn:bilevel} because not only the global optimal solutions of the lower-level problem but also its local optimal solutions are feasible
solutions for
\eqref{eq:0615-2056}. In that sense, problem \eqref{eq:0615-2056} is modified from the original one, but solving it may lead to better prediction performance because the best hyperparameter $\lambda$ is searched in the wider space
and above all, there is no way to solve the bilevel optimization problem \eqref{eqn:bilevel} as it is.
\subsection{Smoothing method}
In our approach for tackling problem\,\eqref{eqn:bilevel}, we will utilize the smoothing method, which is one of the most powerful methodologies developed for solving nonsmooth equations, nonsmooth optimization problems, and so on. 
Fundamentally, the smoothing method solves smoothed optimization problems 
or equations sequentially to produce a sequence converging to a point that satisfies some optimality conditions of the original nonsmooth problem. 
The smoothed problems solved therein are obtained by replacing the nonsmooth functions with so-called smoothing functions.

Let $\bphi_0:\R^n\to \R$ be a 
nonsmooth function. 
Then, we say that $\bphi:\R^n\times \R_+\to \R$ is a smoothing function of $\bphi_0$ when (i) $\bphi(\cdot,\cdot)$ is continuous and $\bphi(\cdot,\mu)$ is continuously differentiable for any $\mu>0$; 
(ii) $\lim_{\tilde{\w}\to \w,\mu\to 0+}\bphi(\tilde{\w},\mu)=\bphi_0(\w)$ for any $\w\in \R^n$.
In particular, we call $\mu\ge 0$ a smoothing parameter.
{For more details on smoothing methods, see the comprehensive survey article\,\citep{chen2012smoothing} and also relevant articles \citep{nesterov2005smooth,beck2012smoothing}.
  \subsection{Our approach}
We propose a smoothing based method for solving \eqref{eqn:bilevel}. 
In the method, we replace the nonsmooth, possibly nonconvex, term  $R_1(\bm w)=\|\bm w\|_p^p$ in \eqref{eqn:bilevel}
by the following smoothing function:
\[
\varphi_\mu(\bm w):=\sum_{i=1}^n (w_i^2+\,\mu^2)^{\frac{p}{2}}.
\]
We then have the following bilevel problem approximating the original one\,\eqref{eqn:bilevel}:
\begin{equation}
\min_{\w^*_{\bm\lambda},\blambda}\ f({\w^*_{\bm\lambda}})\hspace{0.5em}\mbox{s.t.}\hspace{0.5em}\displaystyle {\w^*_{\bm\lambda}} \in \argmin_{\bm w \in \R^n}\left(G(\bm w,\bar{\blambda})+\lambda_1\varphi_{\mu}(\w)\right),\ \bm \lambda \geq \bm 0\notag
\end{equation}
which {naturally} leads 
to the following one-level problem:
\begin{align}
\hspace{-2em}\begin{array}{rcc}
\displaystyle \min_{\bm w,\bm \lambda} && f(\bm  w) \\
\mbox{s.t.} && \displaystyle\nabla_{\w} G(\bm w,\bar{\bm \lambda})+ \lambda_1 \nabla \varphi_{\mu}(\bm w)  = \bm 0\\
&& \blambda\ge \bm 0.
\end{array}\label{eqn:smoothprob}
\end{align}
{Note that problem\,\eqref{eqn:smoothprob} is smooth since 
the function $\varphi_{\mu}$ is twice continuously differentiable
\footnote{Huber's function\,\citep{beck2012smoothing} is a popular smoothing function of $R_1(\cdot)$,
but is not twice continuously differentiable.}
when $\mu\neq 0$.
Hence, we can consider the Karush-Kuhn-Tucker (KKT) conditions for this problem.

Let us explain the proposed method in detail.
To this end, for a parameter $\hat{\varepsilon}>0$, we define an $\hat{\varepsilon}$-approximate KKT point 
for problem\,\eqref{eqn:smoothprob}.
We say that 
$(\w,\blambda,\bzeta,\bbeta)\in \R^n\times \R^r\times \R^n\times \R^r$ is 
an $\hat{\varepsilon}$-approximate KKT point
for \eqref{eqn:smoothprob}
if there exists a vector 
$(\bm\varepsilon_1,\varepsilon_2,{\bm \varepsilon}_3,{\bm\varepsilon}_4,\varepsilon_5)\in \R^{n}\times\R\times \R^{r-1}\times \R^n\times \R$ such that
\begin{eqnarray}
&\hspace{-3em}\nabla f(\bm  w) +\left(\nabla_{\w\w}^2 G(\w,\bar{\blambda}) + 
\lambda_1 \nabla^2\bphi_{\mu}(\w)\right){\bm\zeta}=\bm \varepsilon_1,
\label{eq:1215-1}
\\
&\nabla\bphi_{\mu}(\w)^{\top}\bzeta-\eta_1=\varepsilon_2,\label{eqn:1217-1}\\
&\nabla R_i(\w)^{\top}\bzeta-\eta_i=(\bm \varepsilon_3)_i\ \ (i=2,3,\ldots,r),\label{eqn:1217-2}\\
&\nabla_{\w} G(\w,\bar{\blambda})+\lambda_1\nabla\bphi_{\mu}(\w)=\bm \varepsilon_4,\label{eq:1215-2}\\
& \bm 0\le \blambda, ~ \bm 0 \le \bbeta, ~ \blambda^\top \bbeta=\varepsilon_5, \label{eq:1215-3}
\end{eqnarray}
and 
{$$
\|(\bm\varepsilon_1,\varepsilon_2,{\bm \varepsilon}_3,{\bm\varepsilon}_4,\varepsilon_5)\|\le \hat{\varepsilon},$$}
{where $\nabla_{\w\w}^2 G(\w,\bar{\blambda})$ is the Hessian of $G$ with respect to $\w$.}
{\begin{algorithm}[tb]
\caption{Smoothing Method for Nonsmooth Bilevel Program}
\label{alg:smoothing}
\begin{algorithmic}[1]
\REQUIRE
Choose $\mu_0\neq 0$, ${\beta_1,\beta_2} \in (0,1)$ and $\hat{\varepsilon}_0\ge 0$.
Set $k \leftarrow0$.
\REPEAT 
\STATE
{Find an $\hat{\varepsilon}_k$-approximate KKT point {$(\w^{k+1}, \blambda^{k+1},\bzeta^{k+1},\bbeta^{k+1})$} 
for problem\,\eqref{eqn:smoothprob} with $\mu=\mu_k$.}
\STATE
{Update the smoothing and error parameters by $\mu_{k+1} \leftarrow{\beta_1} \mu_k$
and 
$\hat{\varepsilon}_{k+1} \leftarrow {\beta_2} \hat{\varepsilon}_k$.
} 
\STATE $k \leftarrow k+1$.
\UNTIL {convergence of $(\w^k, \blambda^k,\bzeta^{k},\bbeta^{k})$.} 
\end{algorithmic}
\end{algorithm}}
Notice that an $\hat{\varepsilon}$-approximate KKT point is nothing but a KKT point%
\footnote{Note that 
\eqref{eqn:1217-1} and \eqref{eqn:1217-2} with $(\varepsilon_2,({\bm \varepsilon}_3)_2,\ldots,
({\bm \varepsilon}_3)_r)=\bm 0$
can be obtained from 
$$
\nabla_{\blambda}f(\w)+\nabla_{\blambda}
\left(
\left(
\nabla_{\w}
G(\w,\bar{\blambda})+\lambda_1\nabla\varphi_{\mu}(\w)\right)^{\top}\bzeta\right)-\bbeta = \0.
$$
}
for problem\,\eqref{eqn:smoothprob} if $\hat{\varepsilon}=0$.
Hence, }
$\bzeta\in \R^{n}$ and $\bbeta\in \R^r$
are {regarded as approximate} Lagrange multiplier vectors corresponding to 
the equality constraint 
$\nabla_{\w}G(\w,\bar{\blambda})+\lambda_1 \nabla\bphi_{\mu}(\w)=\bm 0$
and the inequality constraints $\blambda\ge \bm 0$, respectively.
{The proposed algorithm produces a sequence of $\hat{\varepsilon}$-approximate KKT points 
for problem\,\eqref{eqn:smoothprob} while decreasing the values of $\hat{\varepsilon}$ and $\mu$ to $0$.
Precisely, it is described as in Algorithm\,\ref{alg:smoothing}.}

Though, in Algorithm\,\ref{alg:smoothing}, we do not designate any means for computing an $\hat{\varepsilon}$-approximate KKT point for \eqref{eqn:smoothprob},
  sequential quadratic programming (SQP) methods\,\citep{nocedal2006numerical} are promising candidates.
  However, since such SQP methods are designed for solving general constrained problems, we may develop more efficient algorithms by exploiting structure of individual problems.
  This issue will be discussed later in Section~\ref{sec:numeric}.
    See also Appendix\,\ref{appendix:B}.
As for practical stopping criteria of Algorithm~1, we make use of the scaled bilevel SB-KKT conditions studied in the subsequent section.

\section{Theoretical Results}\label{sec:theory}
In this section, we will prove the global convergence of Algorithm\,\ref{alg:smoothing} by investigating an accumulation point of a sequence generated by that algorithm.
  For this purpose, in Section\,\ref{SB-KKT}, we first present new optimality conditions for the original bilevel problem\,\eqref{eqn:bilevel}, named {\it scaled bilevel KKT (SB-KKT)} conditions.
Moreover, in Section~\ref{sec:convtoSB}, we prove that any accumulation point of a sequence generated by Algorithm\,\ref{alg:smoothing} actually satisfies the SB-KKT conditions under some assumptions.


Throughout the section, we often use the following notations for $\w\in \R^n$:
\begin{align*}
I(\w):=\{i\in \{1,2,\ldots,n\}\mid w_i=0\},\ |\bw|^{p}:=\left(|w_1|^p,|w_2|^p,\ldots,|w_n|^p\right)^{\top}.
\end{align*}
\subsection{SB-KKT Conditions}\label{SB-KKT}
Now, we define the SB-KKT conditions for problem\,\eqref{eqn:bilevel}:
\begin{definition}
We say that the scaled bilevel 
Karush-Kuhn-Tucker (SB-KKT) conditions hold at $(\w^{\ast},\blambda^{\ast})\in \R^n\times \R^r$ 
for problem\,\eqref{eqn:bilevel}
when there {exists a pair of vectors} $(\bzeta^{\ast},\bbeta^{\ast})\in \R^n\times\R^r$ such that
\begin{eqnarray}
  &\bW_{\ast}^2\nabla f(\w^{\ast})+\H(\w^{\ast},\blambda^{\ast})\bzeta^{\ast}=\0,\label{eqa:1214-1}\\
&\bW_{\ast}\nabla_{\w}G({\w}^{\ast},\bar{\blambda}^{\ast})+p\lambda_1^{\ast}|\bw^{\ast}|^p=\0,\label{eqa:1214-2}\\
& p\sum_{i\notin I(\w^{\ast})}{\rm sgn}(w_i^{\ast})|w_i^{\ast}|^{p-1}\zeta_i^{\ast}=\eta_1^{\ast},\label{eqa:1214-3}\\
&\zeta_i^{\ast}=0\ (i\in I(\w^{\ast})),\label{eqa:1214-6}\\
&\nabla R_i(\w^{\ast})^{\top}\bzeta^{\ast}=\eta_i^{\ast}\ \ (i=2,3,\ldots,r),\label{eqa:1224-4}\\ 
&\0\le \blambda^{\ast}, ~\0 \le \bbeta^{\ast}, ~(\blambda^{\ast})^\top \bbeta^{\ast}=0,\label{eqa:1214-5}
\end{eqnarray}
where 
$\bW_{\ast}:={\rm diag}(\bw^{\ast})$.
Here, we write 
$$
\H(\w,\blambda):=\bW^2\nabla^2_{\bw}G(\w,\bar{\blambda})+\lambda_1p(p-1)\diag(|\bw|^p)
$$ 
{with $\bW:={\rm diag}(\bw)$} {for $\w\in \R^n$ and $\blambda\in \R^r$.}
In particular, we call 
{a point $(\bw^{\ast},\blambda^{\ast})\in \R^n\times \R^r$}
satisfying the above conditions\,\eqref{eqa:1214-1}--\eqref{eqa:1214-5} an SB-KKT point for problem\,\eqref{eqn:bilevel}.
\end{definition}

In fact, $(\0,\blambda^{\ast})$ is a trivial SB-KKT point for any $\blambda^{\ast}$.
  This can be checked by setting $(\bzeta^{\ast},\bbeta^{\ast})=(\0,\0)$ in the conditions\,\eqref{eqa:1214-1}--\eqref{eqa:1214-5}.
  Experimentally, started from a point apart from such a trivial point, Algorithm~\ref{alg:smoothing} finds non-trivial SB-KKT points in many cases.

{We next prove that the SB-KKT conditions are necessary optimality conditions 
for a certain one-level problem. For this purpose, we derive the {\it scaled first-order necessary condition}\,\citep{chen2010lower} for the lower-level problem in \eqref{eqn:bilevel}:
\begin{equation}
\min_{\bm w \in \R^n}G(\w,\bar{\blambda})+\lambda_1\|\w\|^p_p.\label{eq:low}
\end{equation}
{We say that 
the scaled first-order necessary condition of \eqref{eq:low} holds at $\w^{\ast}$ if}
\begin{equation}
  \bW_{\ast}\nabla_{\w}G(\w^{\ast},\bar{\blambda})+p\lambda_1|\bw^{\ast}|^p= \bm 0.\label{eq:scaled_first}
\end{equation}
Indeed, a local optimum $\w^{\ast}$ of \eqref{eq:low} satisfies the above condition.
This fact can be verified easily by following the proof of Theorem~2.1 in \citep{chen2010lower}.
The above scaled condition was originally presented in
\citep{chen2010lower,chen2013optimality,bian2017optimality} for some optimization problems admitting non-Lipschitz functions.

As in deriving \eqref{eq:0615-2056},
we obtain the following one-level problem by replacing the lower problem in \eqref{eqn:bilevel} with
the scaled first-order necessary condition \eqref{eq:scaled_first}:
\begin{equation}
\displaystyle\min_{\w,\blambda}\ f(\w)\hspace{0.5em}\mbox{s.t.}\hspace{0.5em}\bW \nabla_{\w}G(\w,\bar{\blambda})+p\lambda_1|\bw|^p= \bm 0,\ \blambda\ge \0.
\label{eqn:surrogate}
\end{equation}
As well as \eqref{eq:0615-2056}, the feasible region of \eqref{eqn:surrogate} includes not only 
the global optimal solution of the lower-level problem in the original problem \eqref{eqn:bilevel} but also its local solutions.
Notice that the above problem is still nonsmooth due to the existence of $|\bw|^p$. 

The following theorem states that the SB-KKT conditions are necessary optimality conditions for \eqref{eqn:surrogate}.}
Here, we just give an outline of the proof and defer its detail to Appendix~\ref{appendix:kkt}.\\
\begin{theorem}\label{thm:optimality}
Let $(\w^{\ast},\blambda^{\ast})\in \R^n\times \R^r$ be a local optimum of 
\eqref{eqn:surrogate}.
Then,
$(\w^{\ast},\blambda^{\ast})$ together with some vectors
$\bzeta^{\ast}\in \R^n$ and $\bbeta^{\ast}\in\R^r$ satisfies the 
SB-KKT conditions\,\eqref{eqa:1214-1}--\eqref{eqa:1214-5}
under an appropriate constraint qualification concerning the constraints 
$\frac{\partial{G(\w,\bar{\blambda})}}{\partial w_i}+p\,{\rm sgn}(w_i)\lambda_1|w_i|^{p-1}=0\ (i\notin I(\w^{\ast})),\ w_i=0\ (i\in I(\w^{\ast}))${, and $\blambda\ge \bm 0$}.
\end{theorem}
    {\noindent{\bf{Sketch of the proof}}: Notice that $(\w^{\ast},\blambda^{\ast})$ is a local optimum of the following problem:
\begin{align}\label{eqn:surrogate2}
\begin{array}{rcc}
\displaystyle\min_{\w,\blambda}&  & f(\w)  \\
\mbox{s.t.} &     &\displaystyle \frac{\partial{G(\w,\bar{\blambda})}}{\partial w_i}+p\,{\rm sgn}(w_i)\lambda_1|w_i|^{p-1}=0\ (i\notin I(\w^{\ast}))\\
            &      &w_i=0\ (i\in I(\w^{\ast}))\\  
           &       &\blambda\ge \0.
\end{array}
\end{align}
This is because  
$(\w^{\ast},\blambda^{\ast})$ is also feasible to \eqref{eqn:surrogate2} and the feasible region of 
\eqref{eqn:surrogate} is larger than that of \eqref{eqn:surrogate2}.
Hence, the KKT conditions hold at $(\w^{\ast},\blambda^{\ast})$ for \eqref{eqn:surrogate2} in the presence of a constraint qualification.
Finally, these KKT conditions can be equivalently transformed into the desired SB-KKT conditions.}
\hfill$\blacksquare$

In the next section, we will study convergence analysis of Algorithm~\ref{alg:smoothing} to an SB-KKT point.
Before proceeding to the convergence analysis, let us see the relationship between the two one-level problems
\,\eqref{eq:0615-2056} and \eqref{eqn:surrogate}. 
The following lemma concerns the feasible regions of \eqref{eq:0615-2056} and \eqref{eqn:surrogate}.
\begin{LEMM}\label{lem:0708-2352}
For $\w\in \R^n$ and $\blambda\in \R^r_+$, 
if 
$\bm 0\in \partial_{\w}(G(\w,\bar{\blambda})+\lambda_1R_1(\w))$,
then  $\bW \nabla_{\w}G(\w,\bar{\blambda})+p\lambda_1|\bw|^p= \bm 0$. 
In particular, when $p<1$, the converse is also true.
\end{LEMM}
\begin{proof}
See Appendix~\ref{appendix:lemma}. 
\end{proof}
In view of the above lemma, we find that 
the feasible region of \eqref{eqn:surrogate} is larger than that of \eqref{eq:0615-2056} in general. 
However, for the case of $p<1$, we also see that these two regions are identical.  These relationships are summarized as in the following diagram.
\begin{equation*}
\mbox{Feasible region of \eqref{eqn:bilevel}} \subseteq \mbox{Feasible region of }\eqref{eq:0615-2056} 
\left\{\begin{array}{c}
\subseteq_{p=1}\\
=_{p<1}
\end{array}\right\}
 \mbox{Feasible region of }\eqref{eqn:surrogate}.
\end{equation*}
From this observation and Theorem\,\ref{thm:optimality}, we can derive the following theorem immediately: 
\begin{theorem}
Let $p<1$ and $(\w^{\ast},\blambda^{\ast})\in \R^n\times \R^r$ be a local optimum of \eqref{eq:0615-2056}.
Then,
$(\w^{\ast},\blambda^{\ast})$ together with some vectors
$\bzeta^{\ast}\in \R^n$ and $\bbeta^{\ast}\in\R^r$ satisfies the 
SB-KKT conditions\,\eqref{eqa:1214-1}--\eqref{eqa:1214-5}
under an appropriate constraint qualification concerning the constraints 
$\frac{\partial{G(\w,\bar{\blambda})}}{\partial w_i}+p\,{\rm sgn}(w_i)\lambda_1|w_i|^{p-1}=0\ (i\notin I(\w^{\ast})),\ w_i=0\ (i\in I(\w^{\ast}))${, and $\blambda\ge \bm 0$}.
\end{theorem}
 
\subsection{Convergence of Algorithm~\ref{alg:smoothing} to an SB-KKT Point}\label{sec:convtoSB}
Hereafter, 
for convenience of explanation,
we suppose that 
an $\hat{\varepsilon}_{k-1}$-approximate KKT point 
$(\w^k, \blambda^k,\bzeta^{k},\bbeta^{k})$
is a solution satisfying conditions\,\eqref{eq:1215-1}-\eqref{eq:1215-3}
with 
$$
(\bm\varepsilon_1,\varepsilon_2,{\bm \varepsilon}_3,{\bm\varepsilon}_4,\varepsilon_5)
=(\bm\varepsilon^{k-1}_1,\varepsilon_2^{k-1},{\bm \varepsilon}_3^{k-1},{\bm\varepsilon}_4^{k-1},\varepsilon_5^{k-1}),
$$
where $\{(\bm\varepsilon^{k}_1,\varepsilon_2^{k},{\bm \varepsilon}_3^{k},{\bm\varepsilon}_4^{k},\varepsilon_5^{k})\}$ is a sequence that converges to zero as $k\to \infty$.

Moreover, we suppose that the algorithm is well-defined in the sense 
that {an $\hat{\varepsilon}_k$-approximate} KKT point of \eqref{eqn:smoothprob} is found in Step~2 at every iteration, and
it generates {an} infinite number of iteration points.
In addition, we make the following assumptions:\vspace{1em}

\noindent{\bf Assumption~A:}\ Let {$\{(\w^k, \blambda^k,\bzeta^k,\bbeta^k)\}\subseteq \R^{n}\times \R^r
\times \R^{n}\times \R^r$
be a} sequence produced by the proposed algorithm. 
Then, the following properties hold:
\begin{enumerate}
\item[{\bf A1}:] ${\displaystyle \liminf_{k\to\infty}\lambda^k_1>0}$.
\item[{\bf A2}:] 
The sequence $\{(\w^k, \blambda^k,\bzeta^k,\bbeta^k)\}$ is bounded.
\item[{\bf A3}:] 
Let $p=1$ and $(\w^{\ast}, \blambda^{\ast})$
be an arbitrary accumulation point of the sequence $\{(\w^k, \blambda^k)\}$.
It then holds that $\lambda^{\ast}_1\neq \left|\frac{\partial G(\w^{\ast},\bar{\blambda}^{\ast})}{\partial w_i}\right|$ for any $i\in I(\w^{\ast})$.
\end{enumerate}
Assumption~A1 means that the $\ell_p$-regularization term, i.e., the function $R_1$ works effectively.
We will discuss Assumption~A2 at the end of this section.
  Specifically, we will prove that
under certain conditions, the Lagrange multiplier part $\{\bzeta^k,\bbeta^k)\}$ is actually bounded. 
Assumption~A3 is a technical assumption for the case of $p=1$.
It indicates that, for all $i\in I(\w^{\ast})$, zero is not situated on the boundary of the subdifferential of $G(\w,\blambda)+\lambda\|\w\|_1$ w.r.t. $w_i$.
Interestingly, for the case of $p<1$, we can establish the convergence of Algorithm~\ref{alg:smoothing} in the absence of A3. 

{Under the presence of these assumptions, our goal is to prove the following convergence theorem, which motivates us to make a stopping criterion of the algorithm based on the SB-KKT conditions in the numerical experiments we will conduct later.
\begin{theorem}\label{last_thm}
Suppose that Assumptions~A1--A3 hold.
Then, any accumulation point of $\{(\w^k,\blambda^k,{\bzeta^k},\bbeta^k)\}$
generated by Algorithm~\ref{alg:smoothing} satisfies the 
SB-KKT conditions\,\eqref{eqa:1214-1}--\eqref{eqa:1214-5} for problem \eqref{eqn:bilevel}.
\end{theorem}
We will prove this theorem by showing that passing the approximate KKT conditions~\eqref{eq:1215-1}-\eqref{eq:1215-3} to the limit yields the SB-KKT conditions.
For this purpose, in particular, we have to examine how $\nabla\bphi_{\mu_k}(\w^k)$ and
$\nabla^2\bphi_{\mu_k}(\w^k)$ behave in the limit.
{We remark that each component of $\nabla\bphi_{\mu}(\w)$ and each diagonal one of $\nabla^2\bphi_{\mu}(\w)$ are expressed as 
\begin{align}
&(\nabla\bphi_{\mu}(\w))_i=pw_i(w_i^2+\mu^2)^{\frac{p}{2}-1},\label{al:1214-1}\\
&(\nabla^2\bphi_{\mu}(\w))_{ii}=p(w_i^2+\mu^2)^{\frac{p}{2}-1}+p(p-2)w_i^2(w_i^2+\mu^2)^{\frac{p}{2}-2}\label{al:1214-2} 
\end{align}
for $i=1,2,\ldots,n$, $\mu>0$, and $\w\in \R^n$.
Note that all the off-diagonal components of $\nabla^2\bphi_{\mu}(\w)$ are zeros.
We present the following proposition, whose proof is given in Appendix~\ref{appendix:Lemma2}.
\begin{proposition}\label{lem:1217-1}
Let $\w^{\ast}$ be the point defined in {\bf A3}.
Then, we have
\begin{align}
&\lim_{k\to\infty}\bW_k\nabla\bphi_{\mu_{k-1}}(\w^k)=p|\w^{\ast}|^p,\label{al:1214-1921}\\
&\lim_{k\to\infty}\bW^2_k\nabla^2\bphi_{\mu_{k-1}}(\w^k)=p(p-1)\diag(|\w^{\ast}|^p),\label{al:1214-1922}
\end{align}
where 
${\bm W}_k:={\diag}(\bm w_k)$
for each $k$.
\end{proposition}}
Next, we prove that, for $i\in I(\w^{\ast})$, the $i$-th diagonal component of $(\nabla^2\bphi_{\mu_{k-1}}(\w^k))_{ii}$ diverges.
The key of its proof is the approach speed of $w_i\ (i\in I(\w^{\ast}))$ towards zeros compared with that of the smoothing parameter $\mu_{k-1}$.
Actually, according to the next lemma, $\mu_{k-1}$ gradually approaches $0$ with the speed not faster than $\max_{i\in I(\w^{\ast})}|w_i^k|^{\frac{1}{2-p}}$.
The proof will be given in Appendix~\ref{appendix:speed}.
Remarkably, when $p<1$, it holds true in the absence of Assumption\,{\bf A3}.
}
\begin{LEMM}\label{prop:0312}
Suppose that Assumptions~A1--A3 hold.
Let 
$(\w^{\ast},\blambda^{\ast})$ be an arbitrary accumulation point of 
$\{(\w^k,\blambda^k)\}$ and $\{(\w^k,\blambda^k)\}_{k\in K}(\subseteq \{(\w^k,\blambda^k)\})$ be an arbitrary subsequence converging to $(\w^{\ast},\blambda^{\ast})$.
Then, there exists some $\gamma > 0$ such that
\begin{equation}
\mu_{k-1}^2\ge \gamma |w_i^k|^{\frac{2}{2-p}} \ \ (i\in I(\w^{\ast}))\label{eq:A3}
\end{equation}
for all $k\in K$ sufficiently large.
\end{LEMM}
From the above lemma, we can derive the following proposition.
\begin{proposition}\label{lem:0312}
Suppose that Assumptions~A1--A3 hold. 
Let 
$\w^{\ast}$ be an arbitrary accumulation point of 
the sequence 
$\{\w^k\}$ and $\{\w^k\}_{k\in K}(\subseteq \{\w^k\})$ be an arbitrary subsequence converging to $\w^{\ast}$.
Then, for any $i\in I(\w^{\ast})$, 
\begin{equation}
\lim_{k\in K\to\infty}\left|(\nabla^2\bphi_{\mu_{k-1}}(\w^k))_{ii}\right|=\infty.\label{eq:0312:1012}
\end{equation} 
\end{proposition}
\begin{proof}
Choose $i\in I(\w^{\ast})$ arbitrarily.
Note that 
\begin{equation}
\lim_{k\in K\to \infty}|w^k_i|=w^{\ast}_i=0.\label{eq:0312-1540}
\end{equation}
By Lemma\,\ref{prop:0312}, 
there is some $\gamma>0$ such that 
\begin{equation}
\mu_{k-1}^2\ge \gamma|w^k_i|^{\frac{2}{2-p}}\label{eq:0312-1526}
\end{equation}
for all $k\in K$ sufficiently large. 
In view of this fact, we have, for all $k\in K$ large enough,
\begin{align}
\mu_{k-1}^2+(p-1)(w^k_i)^2&\ge  \gamma|w^k_i|^{\frac{2}{2-p}}+(p-1)|w^k_i|^2\notag \\ 
                        &= |w^k_i|^{\frac{2}{2-p}}(\gamma + (p-1)|w^k_i|^{2-\frac{2}{2-p}})\notag \\
                        &\ge 0, \label{al:0312-1504}
\end{align}
where the second inequality can be verified by noting that $\gamma + (p-1)|w^k_i|^{2-\frac{2}{2-p}}>0$ holds for all $k\in K$ sufficiently large because $\gamma>0$ and $(p-1)\lim_{k\in K\to \infty}|w^k_i|^{2-\frac{2}{2-p}}=
(p-1)|w_i^{\ast}|^{2-\frac{2}{2-p}}=0$. 
Furthermore, notice that
\begin{equation}
|w^k_i|^{\frac{2}{2-p}}\ge |w^k_i|^2\label{eq:0402-2}
\end{equation}
holds for all $k\in K$ sufficiently large because $1<\frac{2}{2-p}\le 2$ and 
$|w^k_i|<1$ for all $k\in K$ large enough by \eqref{eq:0312-1540}.
Relation\,\eqref{eq:0312-1526} then implies
\begin{equation}
\mu_{k-1}^2\ge \gamma|w^k_i|^{2}.\label{eq:0313-2}
\end{equation}
From expression\,\eqref{al:1214-2}, it follows that 
\begin{align}
\left|(\nabla^2\bphi_{\mu_{k-1}}(\w^k))_{ii}\right|
&=p\left|((w^k_i)^2+\mu_{k-1}^2)^{\frac{p}{2}-2}\left((w^k_i)^2+\mu_{k-1}^2+(p-2)(w^k_i)^2\right)\right|
\notag\\
&=p \left|((w^k_i)^2+\mu_{k-1}^2)^{\frac{p}{2}-2}(\mu_{k-1}^2+(p-1)(w^k_i)^2)\right|\notag\\
                                      &= p ((w^k_i)^2+\mu_{k-1}^2)^{\frac{p}{2}-2}\left(\mu_{k-1}^2+(p-1)(w^k_i)^2\right)\notag \\
                                      &\ge  p \left(1+\frac{1}{\gamma}\right)^{\frac{p}{2}-2}\mu_{k-1}^{{p}-4}\left(\mu_{k-1}^2+(p-1)(w^k_i)^2\right)\notag\\
                                      &=  p \left(1+\frac{1}{\gamma}\right)^{\frac{p}{2}-2}\mu_{k-1}^{p-2}\left(1+(p-1)\frac{(w^k_i)^2}{\mu_{k-1}^2}\right),\label{al:0312-1444}
\end{align}
where the third equality follows from \eqref{al:0312-1504} and the first inequality comes from \eqref{eq:0313-2} and $\frac{p}{2}-2<0$.
Moreover, by \eqref{eq:0312-1526}, we see ${\mu_{k-1}^{2(2-p)}}/{\gamma^{2-p}}\ge (w^k_i)^2$
and thus have
\begin{equation}
\frac{\mu_{k-1}^{2-2p}}{\gamma^{2-p}}\ge \frac{(w^k_i)^2}{\mu_{k-1}^2}.\notag
\end{equation}
By this inequality, it holds that 
\begin{equation}
\lim_{k\in K\to \infty}\left|\frac{(w^k_i)^2}{\mu_{k-1}^2}\right|\begin{cases}=0\hspace{1em}&(p<1)\\
\le \frac{1}{\gamma}\hspace{1em}&(p=1).
\end{cases}
\notag
\end{equation}
Thus, we obtain  
\begin{equation}
\lim_{k\in K\to \infty}(p-1)\left|\frac{(w^k_i)^2}{\mu_{k-1}^2}\right|=0,\notag
\end{equation}
which together with \eqref{al:0312-1444} {and $\lim_{k\to \infty}\mu_{k-1}^{p-2}=\infty$} implies
\begin{equation}
\lim_{k\in K\to \infty}\left|(\nabla^2\bphi_{\mu_{k-1}}(\w^k))_{ii}\right|=\infty.\notag
\end{equation}
Since $i\in I(\w^{\ast})$ was arbitrarily chosen, the proof is complete.
\end{proof}

Now, we are ready to prove Theorem~\ref{last_thm} using Propositions~\ref{lem:1217-1} and \ref{lem:0312}.
\subsection*{Proof of Theorem~\ref{last_thm}}
Consider the $\hat{\varepsilon}_{k-1}$-approximate KKT conditions \eqref{eq:1215-1}, \eqref{eqn:1217-2}, and \eqref{eq:1215-2} with $(\w,\blambda,{\bzeta},\bbeta)=(\w^k,\blambda^k,{\bzeta^k},\bbeta^k)$ and
$({\bm \varepsilon}_1,
{\bm \varepsilon}_3,
{\varepsilon}_5)
=({\bm \varepsilon}^{k-1}_1,
{\bm \varepsilon}_3^{k-1},
{\varepsilon}_5^{k-1})$.
Let $(\w^{\ast},\blambda^{\ast},\bzeta^{\ast},\bbeta^{\ast})$ be an arbitrary accumulation point of $\{(\w^k,\blambda^k,\bzeta^k,\bbeta^k)\}$.
By taking a subsequence if necessary,
without loss of generality, we can suppose that
\begin{equation}
\lim_{k\to \infty}
(\w^k,\blambda^k,\bzeta^k,\bbeta^k)=(\w^{\ast},\blambda^{\ast},\bzeta^{\ast},\bbeta^{\ast}).\label{eq:conv}
\end{equation}
To show the desired result, it suffices to prove that $(\w^{\ast},\blambda^{\ast},\bzeta^{\ast},\bbeta^{\ast})$ satisfies \eqref{eqa:1214-1}--\eqref{eqa:1214-5}.
Now, we give the proof by three blocks as follows:\vspace{0.5em}\\
\noindent{\it Proof of conditions~\eqref{eqa:1214-1}, \eqref{eqa:1214-2}, \eqref{eqa:1224-4} and \eqref{eqa:1214-5}}:
As for \eqref{eqa:1214-1} and \eqref{eqa:1214-2}, multiplying
\eqref{eq:1215-1} and \eqref{eq:1215-2} {with $(\w,\blambda,{\bzeta},\bbeta)=(\w^k,\blambda^k,{\bzeta^k},\bbeta^k)$}
by $\bW_k^2$ and $\bW_k$ on the left, respectively, we obtain
\begin{align*}
&\bW_k^2\nabla f(\w^k) + \bW_k^2\left(\nabla_{\w\w}^2 G(\w^k,\bar{\blambda}^k)+\lambda_1^k \nabla^2\bphi_{\mu_{k-1}}(\w^k)\right){\bm\zeta^k}=\bW_k^2{\bm \varepsilon}^{k-1}_1,\\
&\bW_k \nabla_{\w} G(\w^k,\bar{\blambda}^k)+\lambda_1^k\bW_k \nabla\bphi_{\mu_{k-1}}(\w^k)={
\bW_k{\bm \varepsilon}^{k-1}_4}.
\end{align*}
Note that 
the functions $\nabla f$, $\nabla_{\w\w}^2 G$, and $\nabla_{\w} G$ are continuous and let $k\to \infty$ in the above equations. 
Then, using \eqref{al:1214-1921} and \eqref{al:1214-1922} in Proposition\,\ref{lem:1217-1} together with 
$\mu_{k-1}\to 0,\ 
{({\bm \varepsilon}^{k-1}_1,{\bm \varepsilon}^{k-1}_4)\to (\0,\0)
}$ as $k\to \infty$, 
we get \eqref{eqa:1214-1} and \eqref{eqa:1214-2}, that is to say,
\begin{align*}
  &\bW_{\ast}^2\nabla f(\w^{\ast}) +\H(\w^{\ast},\blambda^{\ast})\bzeta^{\ast}=\0, \\
&\bW_{\ast}\nabla_{\w}G(\w^{\ast},\bar{\blambda}^{\ast})+p\lambda_1^{\ast}|\bw^{\ast}|^p=\0.
\end{align*}
Conditions\,\eqref{eqa:1224-4} and \eqref{eqa:1214-5} are obtained by driving $k$ to $\infty$ in \eqref{eqn:1217-2} and \eqref{eq:1215-3} {with $(\w,\blambda,{\bzeta},\bbeta)=(\w^k,\blambda^k,{\bzeta^k},\bbeta^k)$}.\vspace{0.5em}

\noindent{\it Proof of condition~\eqref{eqa:1214-6}}:
Choose $i\in I(\w^{\ast})$ arbitrarily.
Note that the continuity of the functions $\nabla^2_{\w\w}G$ and $\nabla f$.
Then, from \eqref{eq:conv} and condition\,\eqref{eq:1215-1} with $(\w,\blambda,\bzeta)=(\w^k,\blambda^k,\bzeta^k)$, $\{\lambda^k_1 \left(\nabla^2\bphi_{\mu_{k-1}}(\w^k)\right)_{ii}{\zeta^k_i}\}$ is bounded.
On the other hand, recall that  $\{(\nabla^2\bphi_{\mu_{k-1}}(\w^k))_{ii}\}$ is unbounded from Proposition\,\ref{lem:0312}
and $\lim_{k\to\infty}\lambda^k_1=\lambda^{\ast}_1>0$ from Assumption~A1. 
Thus, we get $\lim_{k\to\infty}\zeta^k_i=0$. 
Since the index $i$ was chosen from $I(\w^{\ast})$ arbitrarily, we conclude condition~\eqref{eqa:1214-6}.\vspace{0.5em}

\noindent{\it Proof of condition~\eqref{eqa:1214-3}:} We begin with proving 
\begin{equation}
\lim_{k\to\infty}\sum_{i\in I(\bw^{\ast})}w^k_i((w^k_i)^2+\mu_{k-1}^2)^{\frac{p}{2}-1}\zeta^k_i=0. \label{eq:1217-2340}
\end{equation}
Choose $i\in I(\w^{\ast})$ arbitrarily again.
Note that by Lemma\,\ref{prop:0312}, there exists some $\gamma>0$ such that 
\begin{equation}
\mu_{k-1}^2\ge \gamma|w^k_i|^{\frac{2}{2-p}}\label{eq:0312-1818}
\end{equation}
for all $k$ sufficiently large. In what follows, we consider sufficiently large $k$ so that the 
inequality \eqref{eq:0312-1818} holds.
Then, by $0< p\le 1$, we get 
\begin{equation}
\frac{\mu_{k-1}^{2-p}}{\gamma^{\frac{2-p}{2}}}\ge |w^k_i|.\label{eq:0312-1818-2}
\end{equation}
We then have  
\begin{align}
\left|w^k_i((w^k_i)^2+\mu_{k-1}^2)^{\frac{p}{2}-1}\zeta^k_i\right|&\le \left|w^k_i\mu_{k-1}^{2(\frac{p}{2}-1)}\zeta^k_i\right|\notag \\
&\le \frac{\mu_{k-1}^{2-p}}{\gamma^{\frac{2-p}{2}}}\mu_{k-1}^{2(\frac{p}{2}-1)}\left|\zeta^k_i\right|\notag\\
&=\gamma^{\frac{p}{2}-1}\left|\zeta^k_i\right|.\label{al:1216:2340}
\end{align}
Relation\,\eqref{al:1216:2340} and condition~\eqref{eqa:1214-6}, which was proved above, imply
\begin{equation*}
\lim_{k\to\infty}\left|w^k_i((w^k_i)^2+\mu_{k-1}^2)^{\frac{p}{2}-1}\zeta^k_i\right|=0\label{eq:0312-1828}
\end{equation*}
and hence summing up this equation over $I(\w^{\ast})$ gives the desired expression\,\eqref{eq:1217-2340}}.

  Next, by using \eqref{eq:conv}, $\mu_{k-1}\to 0\ (k\to\infty)$, $w_i^{\ast}\neq 0\ (i\notin I(\w^{\ast}))$, and
$w^{\ast}_i={\rm sgn}(w^{\ast}_i)|w^{\ast}_i|$, we obtain 
\begin{align}
&\lim_{k\to\infty}\sum_{i\notin I(\bw^{\ast})}w^k_i((w^k_i)^2+\mu_{k-1}^2)^{\frac{p}{2}-1}\zeta^k_i\notag \\
&=\sum_{i\notin I(\bw^{\ast})}{\rm sgn}(w^{\ast}_i)|w^{\ast}_i|^{p-1}\zeta^{\ast}_i.\label{eq:1216-2357}
\end{align}
Combining \eqref{eq:1217-2340} and \eqref{eq:1216-2357} with \eqref{al:1214-1} yields 
\begin{align*}
&\lim_{k\to\infty}\nabla\bphi_{\mu_{k-1}}(\w^k)^{\top}\bzeta^k\\
&=
\lim_{k\to\infty}p\sum_{i=1}^nw^k_i((w^k_i)^2+\mu_{k-1}^2)^{\frac{p}{2}-1}\zeta^k_i\\
&=p\lim_{k\to\infty}\left(\sum_{i\in I(\w^{\ast})}w^k_i((w^k_i)^2+\mu_{k-1}^2)^{\frac{p}{2}-1}\zeta^k_i
+\sum_{i\notin I(\w^{\ast})}w^k_i((w^k_i)^2+\mu_{k-1}^2)^{\frac{p}{2}-1}\zeta^k_i
\right)
\\
&=p \sum_{i\notin I(\bw^{\ast})}{\rm sgn}(w^{\ast}_i)|w^{\ast}_i|^{p-1}\zeta^{\ast}_i,
\end{align*}
which together with 
driving $k$ to $\infty$ in condition\,\eqref{eqn:1217-1} with $(\w,\blambda,\bzeta,\bbeta)=(\w^k,\blambda^k,\bzeta^k,\bbeta^k)$
implies
\begin{equation*}
p \sum_{i\notin I(\bw^{\ast})}{\rm sgn}(w^{\ast}_i)|w^{\ast}_i|^{p-1}\zeta^{\ast}_i=\eta_1^{\ast},
\end{equation*}
where we use $\eta^{\ast}_1=\lim_{k\to \infty}\eta^k_1$.
This is nothing but condition\,\eqref{eqa:1214-3}.

Consequently, the proof of Theorem~\ref{last_thm} is complete.\\
\hfill$\blacksquare$

{\subsection*{Boundedness of the Lagrange-multiplier Sequence}
  In Assumption~{\bf A2},
  we suppose that the Lagrange multiplier sequence $\{(\bzeta^k,\bbeta^k)\}$ is bounded.
One may ask when this condition holds. 
In many optimization algorithms, boundedness properties of relevant Lagrange multiplier sequences are shown to be true under suitable constraint qualifications.
In fact, we can verify the boundedness of $\{(\bzeta^k,\bbeta^k)\}$ under the presence of linearly constraint-like qualifications as follows:
\begin{description}
\item[{\bf A4}:] 
Let $(\w^{\ast},\blambda^{\ast})\in \R^n\times \R^r$ be an arbitrary accumulation point of the sequence $\{(\w^k,\blambda^k)\}$.
{Let 
$$I(\blambda^{\ast}):=\{i\in\{1,2,\ldots,r\}\mid \lambda_i^{\ast}=0\}.$$}
Then, the linearly independent constraint qualification (LICQ) holds at $(\w,\blambda)=(\w^{\ast},\blambda^{\ast})$ for the constraints 
${\Phi_{i}(\w,\blambda):=}\frac{\partial{G(\w,\bar{\blambda})}}{\partial w_i}+p\,{\rm sgn}(w_i)\lambda_1|w_i|^{p-1}=0\ (i\notin I(\w^{\ast})),\ w_i=0\ (i\in I(\w^{\ast}))$, and $\blambda\ge \bm 0$,
{that is to say, the gradient vectors for the active constraints 
$$\left\{
\left\{\nabla \Phi_{i}(\w^{\ast},\blambda^{\ast})\right\}_{i\notin I(\w^{\ast})},
\left\{\nabla_{(\w,\blambda)} w_i|_{\w = \w^{\ast}}\right\}_{i\in I(\w^{\ast})},
\left\{\nabla_{(\w,\blambda)} \lambda_i|_{\blambda=\blambda^{\ast}}\right\}_{i\in I(\blambda^{\ast})}
\right\}\subseteq \R^{n+r}
$$
are linearly independent.
}
\end{description}

\begin{proposition}\label{prop:0607}
  Suppose that Assumptions~A1, A3, and A4 hold.
  Additionally, suppose that the sequence $\{(\w^k,\blambda^k)\}$ is bounded.
Let $\{(\bzeta^k,\bbeta^k)\}\subseteq \R^{n}\times \R^r$ be a sequence of the accompanying Lagrange multiplier vectors which satisfy the KKT conditions\,\eqref{eq:1215-1}--\eqref{eq:1215-3}.
Then, $\{(\bzeta^k,\bbeta^k)\}$ is bounded.
\end{proposition}
\begin{proof}
  We derive contradiction by supposing that $\{(\bzeta^k,\bbeta^k)\}$ is unbounded.
For details, see Appendix~\ref{appendix:boundedness}.
\end{proof}}

\section{Numerical Experiments}\label{sec:numeric}
In this section, we investigate the performance of Algorithm~\ref{alg:smoothing} through comparison to other hyperparameter learning methods such as Bayesian optimization\,\citep{Mockus1978} and gridsearch.
All the experiments are conducted on a personal computer with
  Intel Core i7-8559U CPU @ 2.70GHz, 16.00 GB memory.  
Algorithm~\ref{alg:smoothing} and the other competitor algorithms are implemented with MATLAB R2020a.

Two kinds of bilevel problems relevant to linear regression with real data are solved.
The first problem handles a single hyperparameter related to the $\ell_p$-regularizer, while the second one does multiple hyperparameters.
\subsection{Linear regression bilevel problem with a single $\ell_p$ hyperparameter}\label{sec:single}
In this section, we solve the following bilevel problem regarding squared linear regression problem with a single $\ell_p$ hyperparameter.
\begin{align}\label{al:0913}
\begin{array}{rll}
\displaystyle{\min_{\bw,\blambda}}&\ &\|\bA_{\rm val}{\w}-\bm b_{\rm val}\|_2^2\\ 
\mbox{s.t. }&  &\bw \in \displaystyle{\mathop{\rm argmin}_{\hat{\bw}}}\ \left(\|\bA_{\rm tr}\hat{\w}-\bm b_{\rm tr}\|_2^2 + {\exp(\lambda_1)\|\hat{\bw}\|_p^p}
\right),\\
\end{array}
\end{align}
where $p\in (0,1]$ and $(\bA_{\rm txt},\bm b_{\rm txt})\in \R^{m_{\rm txt}\times n}\times \R^{m_{\rm txt}}$ for ${\rm txt}\in \{{\rm val,tr}\}$.
Notice that the form of the above problem slightly differs from that of \eqref{eqn:bilevel} in that $\lambda_i$ is replaced with $\exp(\lambda_i)$ for each~$i$.
  This is because positive hyperparameters, in particular $\lambda_1$, are actually desirable as outputs.
  With this manipulation, the nonnegative constraint $\exp(\lambda_i)\ge \0$, which corresponds to $\lambda_i\ge \0$ in \eqref{eqn:bilevel}, is clearly fulfilled and thus removed.

   For the sake of examining the accuracy of solutions obtained by solving the above problem, we use
  $\|\bA_{\rm te}\w-\bm b_{\rm te}\|_2^2$ as a test error function, where $\bA_{\rm te}\in \R^{m_{\rm te}\times n}$ and $\bm b_{\rm te}\in \R^{m_{\rm te}}$.
The data matrices and vectors $\bA_{\{{\rm val,tr,te}\}}, \bm b_{\{{\rm val,tr,te}\}}$ are taken from UCI machine learning repository\,\cite{Lichman:2013}:
{\bf Facebook} Comment Volume ($\bar{m}=40949$, $n=53$), 
{\bf Insurance} Company Benchmark ($\bar{m}=9000$, $n= 85$),
{\bf Student} Performance for a math exam ($\bar{m}={395}$, $n= 272$)\footnote{The original dataset has $n= 32$, but the feature size is increased by adding new features: interaction effects generated by pairwise products among some features for each sample.},
{{\bf BodyFat} ($\bar{m}=336$, $n= 14$), 
and
{\bf CpuSmall} ($\bar{m}=8192$, $n= 12$)} are from UCI machine learning repository \cite{Lichman:2013}.
The $\bar{m}$ samples are divided into 3 groups
(training, validation and test samples) with the same sample size {$\lceil\bar{m}/3 \rceil$}.
Hence, $m_{\{{\rm val,tr,te}\}}=\lceil\bar{m}/3 \rceil$.

\subsubsection{Experimental conditions}
{\paragraph{Method for solving the smoothed subproblem\,\eqref{eqn:smoothprob}:}
Algorithm~\ref{alg:smoothing} requires an $\hat{\varepsilon}$-approximate KKT point of \eqref{eqn:smoothprob} in Step~2 at every iteration.
To compute such a point, we present an algorithm using implicit functions.
Several past works also employed similar approaches based on implicit functions for hyperparameter optimization. For example, see \citep{maclaurin2015gradient,pmlr-v48-pedregosa16,franceschi2018bilevel}.

We only explain the algorithmic framework of the proposed implicit function based method, leaving the precise description to Algorithm~\ref{alg_subproblem0} in Appendix~\ref{appendix:B1}.
At every iteration, Algorithm~\ref{alg_subproblem0} locally represents problem\,\eqref{eqn:smoothprob} as problem having the hyperparameter $\blambda$ as variables by means of implicit functions defined over the $\blambda$-space.  
The implicit function, say $\w(\cdot)$, is defined on some open set $U$ and expresses a solution set for the smoothed lower-level problem 
$\min_{\bm w \in \R^n} G(\bm w,\bar{\blambda})+\lambda_1\varphi_{\mu}(\w)$. Namely,
we have $\displaystyle\nabla_{\w} G(\bm w(\blambda),\bar{\bm \lambda})+ \lambda_1 \nabla \varphi_{\mu}(\bm w(\blambda))  = \bm 0$ for all $\blambda\in U$.
We then solve the reformulated problem\,\eqref{eqn:smoothprob} that is described in terms of $\blambda$ by the quasi-Newton method\,\citep{nocedal2006numerical}.
Note that, though it is difficult in general to know the concrete form of the implicit function $\w(\cdot)$,
we can compute the gradient $\nabla\w$ in virtue of the implicit function theorem, which enables us to perform gradient based methods like the quasi-Newton method for solving problems that are described in terms of $\w(\cdot)$.

Actually, the efficiency of this approach relies on how rapidly and accurately $\w(\blambda)$ is computed for a given $\blambda$. To this end, we employ a certain modified Newton method, which was originally proposed by \citet{lai2011unconstrained}. See Appendix~\ref{appendix:B2} for details.

Moreover, as a starting point of the algorithm,
we use a solution of the smoothed problem\,\eqref{eqn:smoothprob} at the previous iteration of Algorithm~\ref{alg:smoothing}, aiming for the so-called hot-start effect.}
\paragraph{Other algorithms for comparison:}
For the sake of comparison, we also implement Bayesian optimization and the gridsearch method. 
We use \texttt{bayesopt} in MATLAB with ``MaxObjectiveEvaluations=30'' for Bayesian optimization.
{In gridsearch, we search for the best value of $\|\bA_{\rm val}{\w}-\bm b_{\rm val}\|_2^2$ among 30 grids $\lambda =10^{-4},10^{-4+\frac{8}{29}},\cdots,10^{4-\frac{8}{29}},10^{4}$ for problem\,\eqref{al:0913}.
At each iteration of \texttt{bayesopt} and gridsearch, we
make use of Matlab built-in solver~\texttt{fmincon} so as to solve the lower-level problem of \eqref{al:0913} with a given $\bm \lambda$. 

\paragraph{Parameter setting and termination criteria:}
The smoothing parameter in Algorithm~\ref{alg:smoothing}
is initialized as $\mu_0=1$ and updated by $\mu_{k+1}=\min(0.9\mu_k,10\mu_k^{1.3})$.
The smoothed subproblem~\eqref{eqn:smoothprob} is solved as exactly as possible by fixing $(\hat{\varepsilon}_{0},\beta_0)$ to $(10^{-6},1)$.
As for the termination criteria of Algorithm~\ref{alg:smoothing}, writing a resulting solution as $\w^{\ast}$,
we stop it if the SB-KKT conditions \eqref{eqa:1214-1},
\eqref{eqa:1214-2} and \eqref{eqa:1214-3}
are within the error of $\epsilon:=10^{-3}$.
We also check whether the other SB-KKT conditions\,\eqref{eqa:1214-6}-\eqref{eqa:1214-5} are satisfied.
The default setting of \texttt{bayesopt} is employed. 
Time limits of all the algorithms are set to 600 seconds. }

\subsubsection{Numerical results for problem~\eqref{al:0913} with fixed data size}
{We first show the results of applying the three algorithms to problem~\eqref{al:0913} with $p=1,0.8,0.5$. 
The algorithms are run for 5 times from different starting points $(\blambda^0,\w^0)$ generated in the manner that $\blambda^0$ is set to $\0$ and $\w^0$ is chosen randomly from $[-5,5]^n$.
All the results are summarized in Table~\ref{tbl:comparison_full}, where the value $\mbox{Err}_{\rm te}$ indicates the averaged values of
$\|\bA_{\rm te}\w-\bm b_{\rm te}\|^2$ over 5 runs. The value $\mbox{Err}_{\rm val}$ stands for the averaged value of $\|\bA_{\rm val}\w-\bm b_{\rm val}\|^2$.
The value ``sparsity'' means the ratio of zero elements in the obtained solution $\bm w \in \R^n$, i.e., ${\rm sparsity}=\left|\{i\mid w_i=0\}\right|/n$ and hence, the solution with sparsity$\approx 1$  is very sparse.
We denote by {\rm time(sec)} the spent time from the start to the termination.
In the experiments, for each $i$, we regarded $w_i$ as zero if $|w_i|\le 10^{-4}\max_{1\le i\le n}|w_i|$.
The best (smallest) values of $\mbox{Err}_{\{\rm te,val\}}$ and time(sec), among the three algorithms are displayed in bold.

From the table,  
there are significant differences in time(sec) of the three algorithms, while ${\rm Err}_{\rm te}$ and ${\rm Err}_{\rm val}$ seem comparative. In particular, Algorithm~\ref{alg:smoothing} tends to be the fastest.
Indeed, it attains the best values in time(sec) for 10 out of 15 problem-instances, seven of which moreover achieve the best values in ${\rm Err}_{\rm val}$.
For example, for {\bf Facebook} with $p=1$, it computes a solution with ${\rm Err_{\rm val}}=6.474$ by about 17 seconds, while \texttt{bayesopt} and gridsearch do solutions with ${\rm Err_{\rm val}}\ge 6.476$ after spending more than 40 seconds.
Thus, Algorithm~\ref{alg:smoothing} is likely to be the most effective among the three on seeking $(\w,\blambda)$ with good ${\rm Err}_{\rm val}$.
{\redcolor
  As pointed out by a reviewer, \texttt{bayesopt} actually found the final solutions or close solutions earlier than the recorded time on the table.
%
Nonetheless, in many instances, Algorithm~\ref{alg:smoothing} reached the final solution more quickly than \texttt{bayesopt} found such a solution. 
  We refer readers to 
 Table\,\ref{tbl1:appendix} 
 in Appendix~\ref{sec:appC}, which shows the first time of \texttt{bayesopt} for finding a solution which attains the final best observed objective value, i.e., validation value. Also see Figure\,\ref{Fig;0618-1} that depicts the time-series of the best observed objective value of \texttt{bayesopt} for the problem organized with the data sets of {\bf Student} and {\bf CpuSmall}.}

From the values of sparsity, the problems with smaller $p$ tends to output sparser solutions. For example, 
Algorithm~\ref{alg:smoothing} outputs a solution with ${\rm sparsity}\ge 0.9$ for {\bf Facebook} with $p=0.5$, while ${\rm sparsity}\le 0.3$ for $p=0.8,1$. 
Nevertheless, sparsity of Algorithm~\ref{alg:smoothing} is 0.00
for {\bf BodyFat} with $p=0.8$, while 
${\rm sparsity}=0.7, 0.07$ for $p=0.5,1$, respectively. 
In this case, Algorithm~\ref{alg:smoothing} might fall into a local optimum with small $E_{\rm val}$ at the expense of sparsity.}

\subsubsection{Performance with Varied Data Size}
{Changing the data sizes of {\bf Student} and {\bf Facebook} datasets, we make comparison of the performances of Algorithm~\ref{alg:smoothing}, \texttt{bayesopt}, and gridsearch.

  We first examine how ${\rm Err}_{\rm te}$, ${\rm Err}_{\rm val}$, and time(sec) of the three algorithms change against the sample size, $\hat{m}$, of {\bf Facebook}.
  The sample size $\hat{m}$ is increased from $\frac{1}{2}\bar{m}$ to $\bar{m}$ by $\frac{1}{10}\bar{m}$ with $\bar{m}\approx 40000$.
  We apply the algorithms using $\ell_{0.8}$ regularizer to these problems with a varied sample size. The obtained results are depicted in Figure~\ref{fig:lasst}.
  Figures~\ref{fig:last1} and \ref{fig:last2} indicates that the computed test and validation values ${\rm Err}_{\rm te}$ and ${\rm Err}_{\rm val}$ behave analogously as $\hat{m}$ increases. There are no crucial differences among those values. However, from Figure~\ref{fig:last3},
  computation time, time(sec), of \texttt{bayesopt} grows more rapidly than the others. This may be because 
  \texttt{bayesopt} has to search a wider region as the sample size grows.
  In contrast, the values of time(sec) of Algorithm~\ref{alg:smoothing} and gridsearch grow moderately. In particular, Algorithm~\ref{alg:smoothing} is the fastest for most cases.

 \begin{figure}[htbp]
 \begin{minipage}{.5\hsize}
 \centering
 \includegraphics[scale = 0.5]{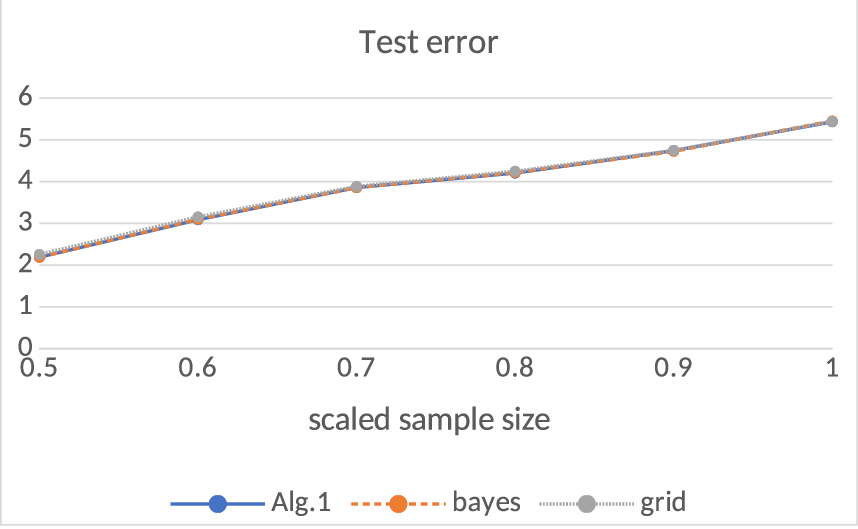}
 \subcaption{${\rm Err}_{\rm te}$ vs scaled sample size $\hat{m}/\bar{m}$
 }
 \label{fig:last1}
 \end{minipage}
 \begin{minipage}{.5\hsize}
 \centering
 \includegraphics[scale = 0.5]{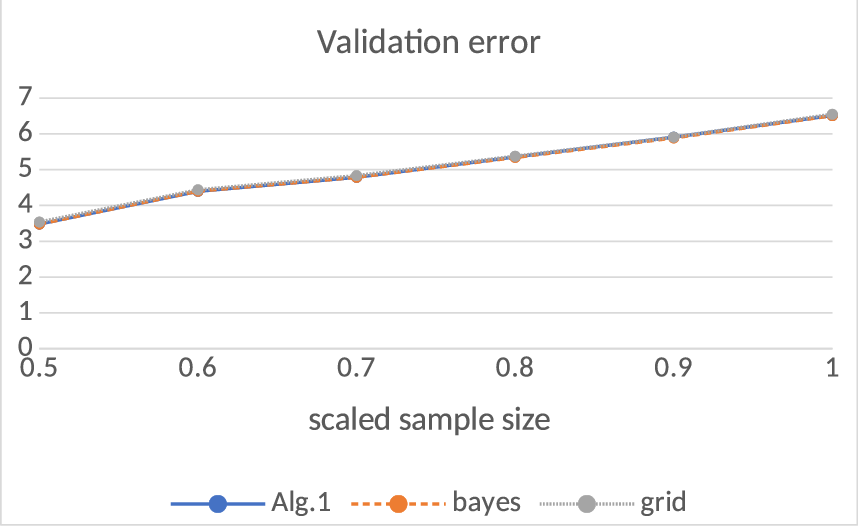}
 \subcaption{${\rm Err}_{\rm val}$ vs scaled sample size $\hat{m}/\bar{m}$
 }
 \label{fig:last2}
 \end{minipage}  \vspace{1em}\\
 \centering
\begin{minipage}{.5\hsize}
 \centering
 \includegraphics[scale = 0.5]{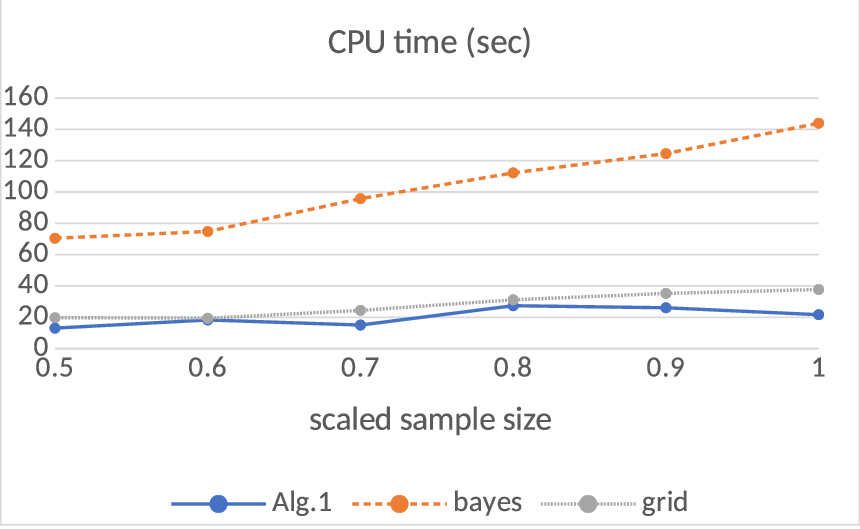}
 \subcaption{time(sec) vs scaled sample size $\hat{m}/\bar{m}$
 }
 \label{fig:last3}
 \end{minipage}
 
   \caption{Performance of Algorithm~\ref{alg:smoothing}, \texttt{bayesopt}, and gridsearch using $\ell_{0.8}$ regularizer for
       {\bf Facebook} with fixed feature size $n=53$ and varying sample size $\hat{m}$
       ($\bar{m}=40949$); Alg.1: Algorithm~\ref{alg:smoothing}, bayes: \texttt{bayesopt}, grid: gridsearch}
\label{fig:lasst}
 \end{figure}

   Next, we investigate the performances of the algorithms by varying the feature size, $\hat{n}$, of {\bf Student}. As in the above experiment, we use $\ell_{0.8}$ regularizer.
   The feature size $\hat{n}$ is increased from $\frac{1}{2}n$ to $n$ by $\frac{1}{10}n$ with $n=272$.
   The obtained results are shown in Figure~\ref{fig:lasst-2}.
   Algorithm~\ref{alg:smoothing} successfully attains better values in all time(sec), $\mbox{Err}_{\rm te}$, and $\mbox{Err}_{\rm val}$ than \texttt{bayesopt} and gridsearch
   for $n\ge 0.7\hat{n}$. From Figures~\ref{fig:last1-2} and \ref{fig:last2-2}, \texttt{bayesopt} seems stuck in a local optimum with larger
$\mbox{Err}_{\rm val}$ and $\mbox{Err}_{\rm te}$ for $n\ge 0.8\hat{n}$.
   From Figure~\ref{fig:last3-2}, time(sec) for \texttt{bayesopt} and gridsearch grow more rapidly than ours as $\hat{n}$ increases.

The above two experiments suggest that, compared with gridsearch and Bayesian optimization, Algorithm~\ref{alg:smoothing} is unlikely to be affected by growth of the data size.

 \begin{figure}[htbp]
 \begin{minipage}{.5\hsize}
 \centering
 \includegraphics[scale = 0.5]{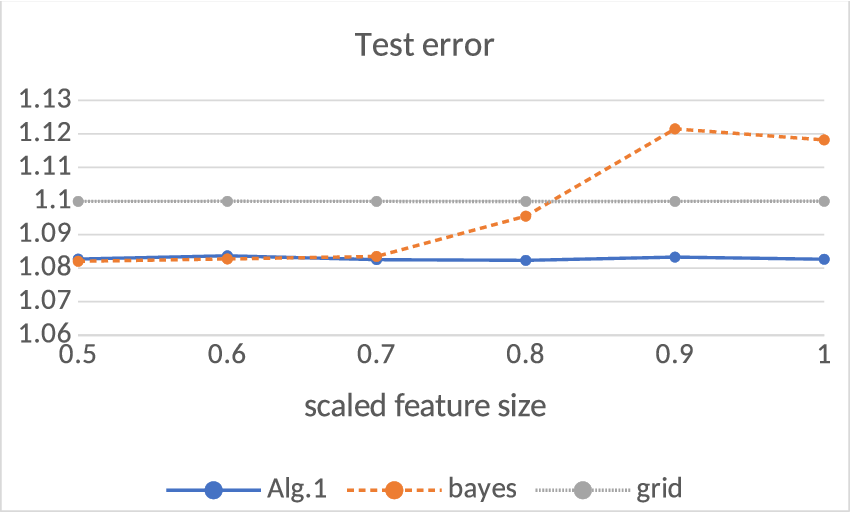}
 \subcaption{${\rm Err}_{\rm te}$ vs scaled feature size $\hat{n}/n$
 }
 \label{fig:last1-2}
 \end{minipage}
 \begin{minipage}{.5\hsize}
 \centering
 \includegraphics[scale = 0.5]{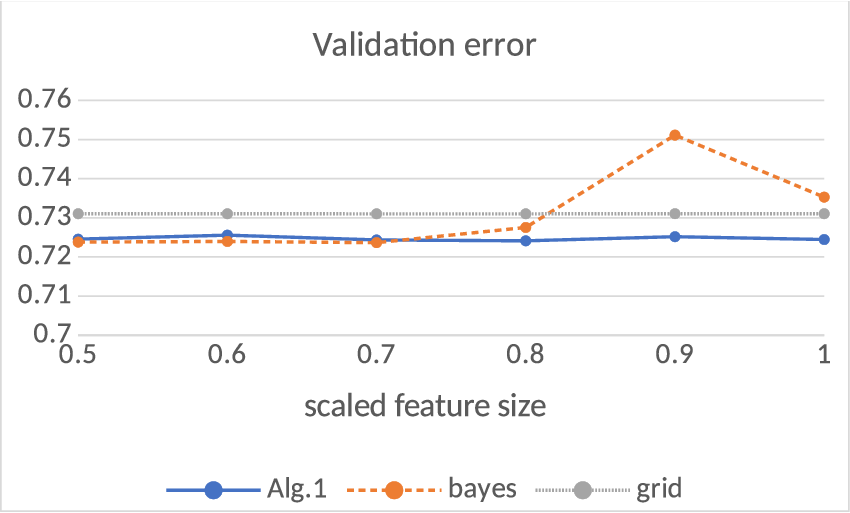}
 \subcaption{${\rm Err}_{\rm val}$ vs scaled feature size $\hat{n}/n$
 }
 \label{fig:last2-2}
 \end{minipage}  \vspace{1em}\\
 \centering
\begin{minipage}{.5\hsize}
 \centering
 \includegraphics[scale = 0.5]{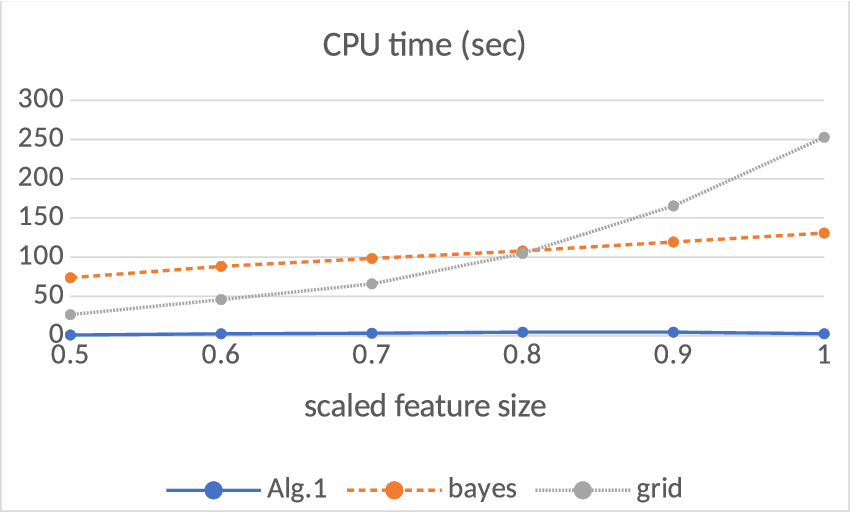}
 \subcaption{time(sec) vs scaled feature size $\hat{n}/n$
 }
 \label{fig:last3-2}
 \end{minipage}
 
   \caption{Performance of Algorithm~\ref{alg:smoothing}, \texttt{bayesopt}, and gridsearch using $\ell_{0.8}$ regularizer for
{\bf Student} with varying feature size $\hat{n}$ ($n=272$) and fixed sample size $\bar{m}=395$; Alg.1: Algorithm~\ref{alg:smoothing}, bayes: \texttt{bayesopt}, grid: gridsearch} 
 \label{fig:lasst-2}
 \end{figure}   
}

\begin{landscape}
  \begin{table}[h]
\centering
\caption{
  Comparison of Algorithm~\ref{alg:smoothing}, \texttt{bayesopt} in MATLAB and gridsearch in terms of squared errors (validation error ${\rm Err}_{\rm val}$ and test error ${\rm Err}_{\rm te}$), CPU times (time(sec)), and sparsities (sparsity).
Here, ${\rm sparsity}:=\left|\{i\mid w_i=0\}\right|/n$ and hence 
a solution is sparser as ``{\rm sparsity}'' is closer to 1. 
The best values in ${\rm Err}_{\rm val}$, ${\rm Err}_{\rm te}$, and time(sec) are displayed in bold.}
\label{tbl:comparison_full}
	{
\scalebox{1}{
\begin{tabular}{lrrrrrrrrrrrrrrrrr}
\toprule
\multicolumn{2}{c}{Data}&
\multicolumn{4}{c}{{\Alg}}&
\multicolumn{4}{c}{\texttt{bayesopt} in MATLAB}&
\multicolumn{4}{c}{grid search}\\
\cmidrule(lr){1-2}
\cmidrule(lr){3-6}
\cmidrule(lr){7-10}
\cmidrule(lr){11-14}
\cmidrule(lr){15-18}
\multicolumn{1}{c}{name}&
\multicolumn{1}{c}{$p$}&
\multicolumn{1}{c}{$\mbox{Err}_{\rm te}$}&
\multicolumn{1}{c}{$\mbox{Err}_{\rm val}$}&
\multicolumn{1}{c}{time (sec)}&
\multicolumn{1}{c}{sparsity}&
\multicolumn{1}{c}{$\mbox{Err}_{\rm te}$}&
\multicolumn{1}{c}{$\mbox{Err}_{\rm val}$}&
\multicolumn{1}{c}{time (sec)} &
\multicolumn{1}{c}{sparsity}&
\multicolumn{1}{c}{$\mbox{Err}_{\rm te}$}&
\multicolumn{1}{c}{$\mbox{Err}_{\rm val}$}&
\multicolumn{1}{c}{time (sec)}&
\multicolumn{1}{c}{sparsity}\\
\midrule
{\bf Facebook} &1 	&	5.399 	&	{\bf 6.474} 	&	{\bf 17.399} 	&	0.057 	
&	5.401 	&	6.476 	&	146.144 	&	0.249 	&	{\bf 5.394} 	&	6.483 	&	43.738 	&	0.264 \\		
					& 		0.8 	&	5.431 	&	6.512 	&	{\bf 22.242} 	&	0.245 	
					&	{\bf 5.411} 	&	{\bf 6.499} 	&	137.764 	&	0.143 	
					&	5.439 	&	6.545 	&	37.021 	&	0.113 \\
					& 		0.5 	&	5.455 	&	6.550 	&	16.820	&	0.925
					 	&	{\bf 5.452}	&	{\bf 6.535} 	&	93.967 	&	0.596
					 	&	5.704 	&	6.747 	&	{\bf 16.514} 	&	0.925 \\ \hline
{\bf Insurance}      			& 		1	&	 {\bf 87.837} 	&	{\bf 95.764} 	&	33.077 	&	0.035 	
	&	87.842 	&	{\bf 95.764} 	&	55.304 	&	0.435 	
	&	87.843 	&	95.765 	&	{\bf 19.779} 	&	0.435  \\
					& 		0.8	&	87.891 	&	95.676 	&	32.465 	&	0.188 	
					&	87.882 	&	{\bf 95.634} 	&	75.384 	&	0.287 	
					&	{\bf 87.872} 	&	95.806 	&	{\bf 19.388} 	&	0.329  \\
					& 		0.5 	&	88.625 	&	95.562 	&	44.904 	&	0.859 	
					&	{\bf 88.101} 	&	{\bf 95.526} 	&	45.359 	&	0.675 	
					&	88.211 	&	96.592 	&	{\bf 5.453} 	&	0.871  \\ \hline

{\bf Student} 				& 		1	&	{\bf 1.127} 	&	{\bf 0.778} 	&	{\bf 10.586} 	&	0.625
 	 	&	1.147 	&	0.785 	&	217.415 	&	0.032 	
 	 	&	1.147 	&	0.785 	&	81.766 	&	0.066 \\

					& 	0.8 	&	1.083 	&	{\bf 0.724} 	&	{\bf 2.348} 	&	0.996 	
						&	{\bf 1.082} 	&	{\bf 0.724} 	&	409.382 	&	0.004 	
						&	1.100 	&	0.731 	&	241.298 	&	0.004 \\

					& 	0.5	&	{\bf 1.082} 	&	{\bf 0.724} 	&	{\bf 3.618} 	&	0.996 	
					&	1.125 	&	0.772 	&	152.826 	&	0.829 	&	2.848 	&	1.120 	&	9.520 	&	0.974 \\ \hline
{\bf BodyFat} 				&		1	&	0.277 	&	{\bf 0.209} 	&	{\bf 0.068} 	&	0.071 	
	&	{\bf 0.276} 	&	0.210 	&	10.253 	&	0.714 	&	0.279 	&	0.216 	&	0.820 	&	0.714 \\
                   			& 		0.8	&	0.288 	&	{\bf 0.179} 	&	{\bf 0.203} 	&	0.000
                   			 		&	{\bf 0.280} 	&	0.184 	&	10.567 	&	0.571 	&	0.287 	&	0.182 	&	0.808 	&	0.429 \\
					& 		0.5 	&	0.582 	&	0.267 	&	{\bf 0.395} 	&	0.714 
						&	0.353 	&	0.229 	&	7.279 	&	0.286
						&	{\bf 0.316} 	&	{\bf 0.256} 	&	0.931 	&	0.286 \\ \hline

{\bf CpuSmall} 				&1 	&	132326	&	{\bf 130981} 	&	11.299 	&	0.083 	
&	{\bf 131834}	&	131123 	&	21.334 	&	0.917 	
&	132261	&	130983	&	{\bf 1.164} 	&	0.917 \\
					& 0.8 	&	132339 	&	{\bf 130982}	&	{\bf 0.741} 	&	0.250 
						&	{\bf 131770} 	&	131187 	&	19.909 	&	0.917 	
						&	132205 	&	130991 	&	1.240 	&	0.750 \\
					& 0.5	&	132093	&	{\bf 131058} 	&	{\bf 0.672} 	&	0.250 	
					&	{\bf 131754}	&	131234	&	17.619 	&	0.917 
				  &	132127 	&	131059 	&	1.514 	&	0.750 \\ 
\bottomrule
\end{tabular}}}
\end{table}
\end{landscape}

  \subsubsection{Performance as the smoothing parameter $\mu$ decreases}\label{experiment:mu}
{We examine impact of the smoothing parameter $\mu$ on test error of solutions of the smoothed subproblems\,\eqref{eqn:smoothprob}. 
Figures~\ref{fig:last3-3}-\ref{fig:last5-3} depict the growth behavior of the test error in the final stage of Algorithm~\ref{alg:smoothing} for the problems of {\bf Facebook}, {\bf BodyFat}, and {\bf Insurance}.

From the figures, the test errors for the three problems do not vary significantly.
Taking into account that the smoothed subproblem\,\eqref{eqn:smoothprob} is more difficult as $\mu$ becomes smaller,
it may be good strategy to stop the algorithm earlier than convergence to a SB-KKT point.	
This is also indicated by the fact that, around $\mu=0.01$, the algorithm attains solutions with better test errors than the solutions upon termination for {\bf Facebook} and {\bf BodyFat}.}

\begin{figure}[htbp]
 \begin{minipage}{.5\hsize}
 \centering
 \includegraphics[scale = 0.48]{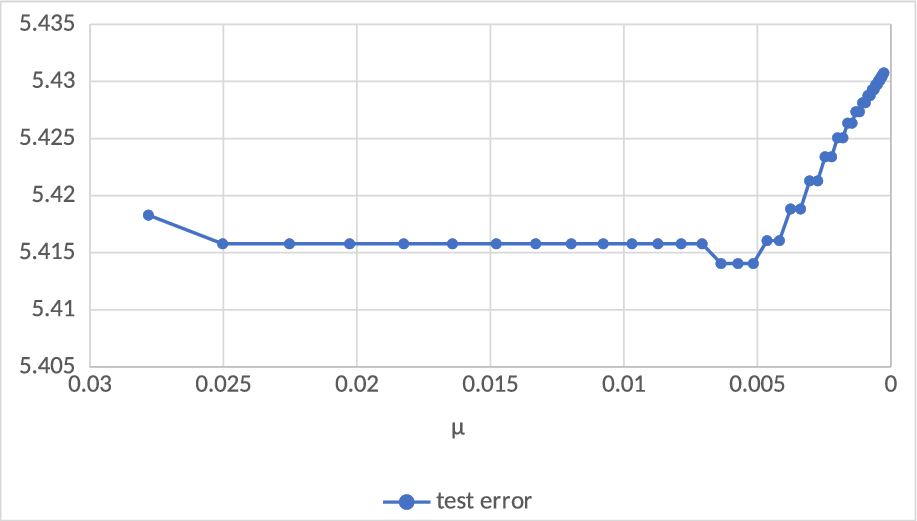}
 \subcaption{{\bf Facebook}
 }
 \label{fig:last3-3}
 \end{minipage}
 \begin{minipage}{.5\hsize}
 \centering
 \includegraphics[scale = 0.48]{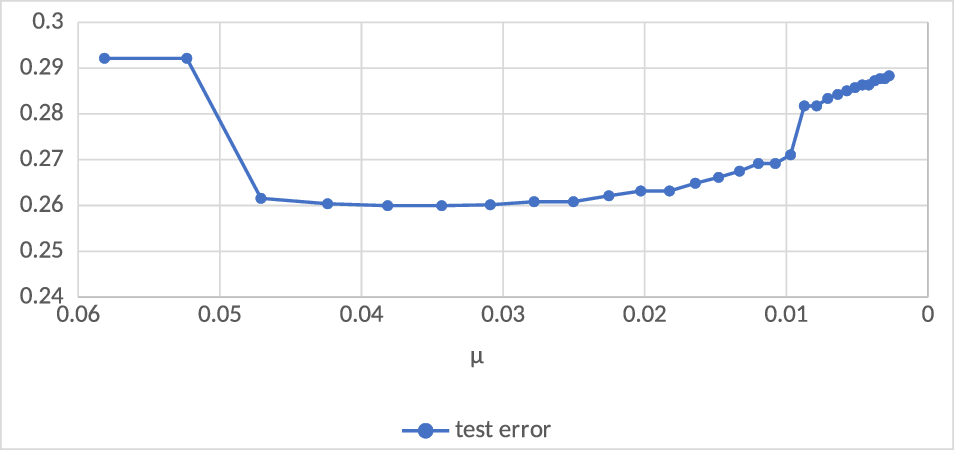}
 \subcaption{{\bf BodyFat}
 }
 \label{fig:last4-3}
 \end{minipage}
 \vspace{0.5em}\\
 \centering
  \begin{minipage}{.5\hsize}
 \centering
 \includegraphics[scale = 0.48]{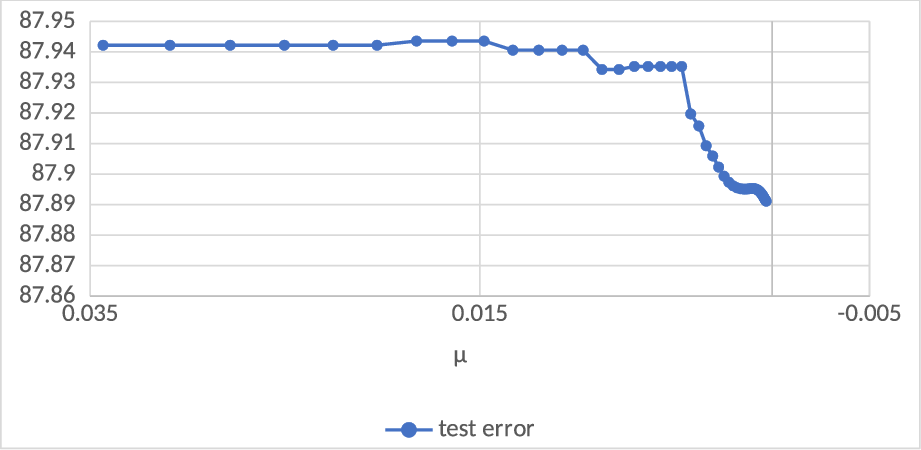}
 \subcaption{{\bf Insurance}}
 \label{fig:last5-3}
  \end{minipage}
     \caption{Change of test error for $\ell_{0.8}$ regularizer as 
  $\mu$ decreases; The horizontal axis: a smoothing parameter $\mu$; The vertical axes: test error}
 \label{fig:lasst-3}
 \end{figure} 

{\subsection{Linear regression problem with multiple hyperparameters}
Next, we solve the following problem that possesses multiple hyperparameters:
\begin{align}\label{al:0913-2}
\begin{array}{rll}
\displaystyle{\min_{\bw,\blambda}}&\ &\|\bA_{\rm val}{\w}-\bm b_{\rm val}\|_2^2 \\ 
\mbox{s.t. }&  &\bw \in\displaystyle{\mathop{\rm argmin}_{\hat{\bw}}}\ \left(\|\bA_{\rm tr}\hat{\w}-\bm b_{\rm tr}\|_2^2 + {\exp(\lambda_1)\|\hat{\bw}\|_p^p}
+\hwCw
\right),
\end{array}
\end{align}
where
$\C(\bm \bar{\blambda}):={\rm Diag}(\exp(\lambda_i))_{i=2}^{n+1}$ being positive definite and $\bA_{\{\rm val,tr,te\}}$ and $\bm b_{\{\rm val,tr,te\}}$ are the ones used in the previous experiments in Section~\ref{sec:single}.
\subsubsection{Experimental conditions}
\tred{We set $p=0.5$ in problem\,\eqref{al:0913-2}}.
The experimental conditions are almost the same as those for problem~\eqref{al:0913}. The main differences are as follows:
We make comparison of Bayesian Optimization~\texttt{bayesopt}
\tred{with ``MaxObjectiveEvaluations=300''} and Algorithms~\ref{alg:smoothing} using the two algorithms for solving subproblems~\eqref{eqn:smoothprob}: The first one is the implicit function approach as in the previous experiments and the second one is \texttt{fmincon}, where we opt for the SQP method and set ``MaxIterations$=10^7$''.
Though we also implemented gridsearch seeking a solution over $30^n$ grids $\blambda\in \{10^{-4},10^{-4+\frac{8}{29}},\cdots,10^{4-\frac{8}{29}},10^{4}\}^n$,
the obtained results were actually quite poor because the number of grids, which was larger than $30^{10}$, was too huge to search. We thus omit those results with gridsearch.
Finally, according to the observation in the last experiment in Subsection\,\ref{experiment:mu}, we terminate Algorithm~\ref{alg:smoothing} when $\mu_k\le 0.01$.

\subsubsection{Numerical results for problem~\eqref{al:0913-2}}
Table~\ref{tbl2:comparison_full} summarizes the obtained results of applying the two types of Algorithms~\ref{alg:smoothing} and \texttt{bayesopt} to problem\,\eqref{al:0913-2}.
In the table, we denote by $\sharp \blambda$ the number of hyperparameters $\blambda$ in each problem.
Moreover, {\algone} stands for Algorithm~\ref{alg:smoothing} using the implicit function approach and {\algtwo} does the one using \texttt{fmincon}.
The hyphens ``--'' in the line of {\bf Facebook} for {\algtwo} 
indicate that \texttt{fmincon}, which is used in {\algtwo}, terminates with an infeasible solution of the smoothed problem\,\eqref{eqn:smoothprob}.

From the table, compared with the results for the single-hyperparameter bilevel problem\,\eqref{al:0913}, \texttt{bayesopt} does not work well. 
For four out of the five problems, it cannot terminate within the time-limit 600 seconds. The qualities of the output solutions of \texttt{bayesopt} upon termination are
also not good in values of ${\rm Err}_{\rm val}$
and ${\rm Err}_{\rm te}$.
\tred{For the sake of completeness, as well as the previous experiment, we examined
the first time when the best observed objective values, i.e., validation values of \texttt{bayesopt} were found.
Refer to Table\,\ref{tbl2:appendix} in Appendix\,\ref{sec:appC}.
Also see Figure\,\ref{Fig;0618-2} for graphs depicting the time-series of the best observed objective values concerning the data sets of {\bf Student} and {\bf CpuSmall}.
}

\tred{In contrast to \texttt{bayesopt},} the two Algorithms~\ref{alg:smoothing} with different subroutines show better performance.
There are notable differences between performances of {\algone} and {\algtwo}.
While {\algone} seems to fall into non-sparse solutions, but with good ${\rm Err}_{\{\rm te,val\}}$ for {\bf Insurance}, {\bf BodyFat}, and {\bf CpuSmall}, {\algtwo} finds good solutions balancing in sparsity, ${\rm Err}_{\rm te}$, and ${\rm Err}_{\rm val}$ for all the same data sets.
This may be because {\algone} computes a solution of problem\,\eqref{eqn:smoothprob} by remaining in the feasible set, namely, the solution set of the smoothed lower level problem.
This behavior may cause {\algone} to miss a chance of broadly seeking sparse solutions.} Meanwhile, {\algtwo} using the SQP method in \texttt{fmincon} tends to approach a solution of \eqref{eqn:smoothprob} from the outside of the feasible set,
which often leads to a good solution even in sparsity.

\begin{landscape}
  \begin{table}[h]
\centering
\caption{  
 Comparison of \texttt{bayesopt} and Algorithm	s~\ref{alg:smoothing} using the implicit function approach and \texttt{fmincon} as subroutines for solving subproblem\,\eqref{eqn:smoothprob} in Step~2 in terms of squared errors (validation error ${\rm Err}_{\rm val}$ and test error ${\rm Err}_{\rm te}$), CPU times (time(sec)), and sparsities (sparsity).
Here, ${\rm sparsity}:=\left|\{i\mid w_i=0\}\right|/n$ and hence 
a solution is sparser as ``{\rm sparsity}'' is closer to 1.
{\algone} stands for Algorithm~\ref{alg:smoothing} using the implicit function approach and
{\algtwo} does the one using \texttt{fmincon}.
The notation $\sharp \blambda$ stands for the number of hyperparameters in each problem. 
Moreover, the best values in ${\rm Err}_{\rm val}$, ${\rm Err}_{\rm te}$, and time(sec) are displayed in bold.
The hyphens ``--'' for {\bf Facebook} indicate that \texttt{fmincon}, which is used in {\algtwo}, terminates with an infeasible solution of the smoothed problem\,\eqref{eqn:smoothprob}.
The algorithms with ``600'' seconds stopped at the time-limit.
}
\label{tbl2:comparison_full}
	{
\scalebox{1}{
\begin{tabular}{lrrrrrrrrrrrrrrrrr}
\toprule
\multicolumn{2}{c}{Data}&
\multicolumn{4}{c}{{\algone}}&
\multicolumn{4}{c}{{\algtwo}}&
\multicolumn{4}{c}{\texttt{bayesopt} in MATLAB}\\
\cmidrule(lr){1-2}
\cmidrule(lr){3-6}
\cmidrule(lr){7-10}
\cmidrule(lr){11-14}
\multicolumn{1}{c}{name}&
\multicolumn{1}{c}{$\sharp\blambda$}&
\multicolumn{1}{c}{$\mbox{Err}_{\rm te}$}&
\multicolumn{1}{c}{$\mbox{Err}_{\rm val}$}&
\multicolumn{1}{c}{time (sec)}&
\multicolumn{1}{c}{sparsity}&
\multicolumn{1}{c}{$\mbox{Err}_{\rm te}$}&
\multicolumn{1}{c}{$\mbox{Err}_{\rm val}$}&
\multicolumn{1}{c}{time (sec)} &
\multicolumn{1}{c}{sparsity}&
\multicolumn{1}{c}{$\mbox{Err}_{\rm te}$}&
\multicolumn{1}{c}{$\mbox{Err}_{\rm val}$}&
\multicolumn{1}{c}{time (sec)}&
\multicolumn{1}{c}{sparsity}\\
\midrule
{\bf Facebook}				&	54 & 	{\bf 5.404} 	&	6.478 	&	{20.247}	&	0.038 	
					        &	--&	--&	-- &	-- 
&	7.629 	&	8.780 	&	600.000 	&	0.000 \\\hline
{\bf Insurance}					& 86 	&	{\bf 87.920}	&	95.694 	&	{50.714} 	&	0.000 
&	88.340 	&	{\bf 94.604} 	&	{\bf 4.473}	&	0.988
						& 98.000 	&	107.000 	&	600.000 	&	0.000\\\hline
{\bf Student}				& 273&	{\bf 1.132} 	&	{\bf 0.771} 	&	{\bf 1.451} 	&	0.368 
&	1.142 	&	0.786 	&	71.775 	&	0.670
&	21.002 	&	18.324 	&	600.000 &0.000\\\hline
{\bf BodyFat}					& 15 	&	{0.531} 	&	0.243 	&	{\bf 0.072} 	&	0.000 	
&	{\bf 0.286} 	&	{\bf 0.130} 	&	0.695 	&	0.933
&	46.675 	&	48.815 	&	455.230 	&	0.000\\\hline
{\bf CpuSmall}					& 13 	&	{132780}	&	131130 	&	{6.222} 	&	0.000
&	{\bf 131940}	&	{\bf 128540} 	&	{\bf 0.658} 	&	0.692
					 	&	156420 	&	151670 	&	600.000 	&	1.000 \\
\bottomrule
\end{tabular}}}
\end{table}
\end{landscape}

\section{Discussion on extension to other nonsmooth regularizers}\label{sec:discuss}
          In this section, we discuss the extension of the SB-KKT conditions from
          hyperparameter optimization of $\ell_p$-regularizers to those of other nonsmooth and nonconvex regularizers.
        Let $\Theta:[0,\infty)\to [0,\infty)$ be a function such that it is concave and continuously differentiable on $(0,\infty)$ and $\Theta(0)=0$.
            Many sparse regularizers are representable in terms of $\Theta$. Indeed, $\sum_{i=1}^n\Theta(|w_i|)$ reduces to
            the $\ell_p$- and $\log$-regularizers, SCAD, and MCP by selecting $\Theta$ appropriately as follows. Here, $a>1$, $\b>0$, and $\gamma>0$ are prefixed parameters and $x\ge 0$.
            The continuous differentiability of $\Theta$ for SCAD and MCP can be confirmed in view of the formula of $\Theta^{\prime}$:
            \begin{itemize}
            \item $\ell_p$-regularizer $(p\le 1)$ if $\Theta(x):=x^p$;
            \item $\log$-regularizer\,\citep{candes2008enhancing} if $\Theta(x):=\displaystyle \frac{1}{\log(1+\gamma)}\log(1+\gamma x)$;
            \item SCAD\,\citep{fan2001variable} if  
$$
\Theta(x):=\begin{cases}
\b x &\hspace{1em}\mbox{if }x \le \b\\
\displaystyle{-\frac{x^2-2abx +\b^2}{2(a-1)}}&\hspace{1em}\mbox{if }\b \le x \le a\b\\
\displaystyle{\frac{(a+1)b^2}{2}}&\hspace{1em}\mbox{otherwise}.
\end{cases}
$$
The first-order derivative is 
$$
\Theta^{\prime}(x)=\begin{cases}
\b &\hspace{1em}\mbox{if }x \le \b\\
\displaystyle{-\frac{x-a\b}{a-1}}&\hspace{1em}\mbox{if }\b \le x \le a\b\\
0&\hspace{1em}\mbox{otherwise};
\end{cases}
$$
\item MCP\,\citep{zhang2010nearly} if 
$$
\Theta(x):=\begin{cases}
\displaystyle{\b x -\frac{x^2}{2a}} &\hspace{1em}\mbox{if }x\le a\b\\
\displaystyle{\frac{a\b^2}{2}}&\hspace{1em}\mbox{otherwise}.
\end{cases}
$$
The first-order derivative is 
$$
\Theta^{\prime}(x)=\begin{cases}
\displaystyle{\b -\frac{x}{a}} &\hspace{1em}\mbox{if }x\le a\b\\
0&\hspace{1em}\mbox{otherwise}.
\end{cases}
$$
            \end{itemize}
            Consider the following extended formulation from problem\,\eqref{eqn:bilevel}:
        \begin{equation}
 \min_{\w^*_{\bm\lambda},\blambda}\ f(\w^*_{\bm\lambda})\hspace{0.5em}\mbox{s.t.}\hspace{0.5em}\displaystyle
  \w^*_{\bm\lambda} \in \argmin_{\bm w \in \R^n} \left(G(\w,\bar{\blambda})+\lambda_1\sum_{i=1}^n\Theta(|w_i|)\right),\ {\bm \lambda \geq \bm 0}.
\label{eqn:bilevel_ex}
        \end{equation}
       Any local optimum $\w$ of the lower-level problem in the above satisfies  
        \begin{eqnarray} 
          &\frac{\partial G(\w,\bar{\blambda})}{\partial w_i}+{\rm sgn}(w_i)\lambda_1\Theta^{\prime}(|w_i|)=0\ \mbox{  for $i\in \{1,2,\ldots,n\}\setminus I(\w)$},& \label{eq:scaled_first_ex1}\\
          &w_i=0\ \mbox{  for $i\in I(\w)$},&\label{eq:scaled_first_ex2}
        \end{eqnarray}
        which actually correspond with the scaled first order condition\,\eqref{eq:scaled_first}
        after transformations when $\Theta$ is chosen to be the $\ell_p$-regularizer.
We then obtain the following one-level problem that corresponds to \eqref{eqn:surrogate}
\begin{equation}
\displaystyle\min_{\w,\blambda}\ f(\w)\hspace{0.5em}\mbox{s.t.}\hspace{0.5em}\blambda\ge \0,\ \mbox{$\w$ satisfies \eqref{eq:scaled_first_ex1} and \eqref{eq:scaled_first_ex2}}.
\label{eqn:surrogate_ex}
\end{equation}
The following theorem states the SB-KKT conditions for \eqref{eqn:surrogate_ex}, which are the extended version of Theorem\,\ref{thm:optimality}.
Since its proof can be obtained via straightforward extension of that of Theorem\,\ref{thm:optimality} by noting the relation $\left(\Theta(|x|),\Theta^{\prime}(|x|),\Theta^{\prime\prime}(|x|)\right)=
\left(|x|^p,{\rm sgn}(x)p|x|^{p-1},p(p-1)|x|^{p-2}\right)$ for $x\neq 0$ when $\Theta(x)=x^p$ for $x\ge 0$, we omit it.
        \begin{theorem}\label{thm:optimality_ex}
Let $(\w^{\ast},\blambda^{\ast})\in \R^n\times \R^r$ be a local optimum of \eqref{eqn:surrogate_ex}.
Suppose that $\Theta$ is twice continuously differentiable at $|w^{\ast}_i|$ for $i\notin I(\w^{\ast})$.
Then,
$(\w^{\ast},\blambda^{\ast})$ together with some vectors
$\bzeta^{\ast}\in \R^n$ and $\bbeta^{\ast}\in\R^r$ satisfies the 
following conditions under an appropriate constraint qualification concerning the constraints 
$\frac{\partial G(\w,\bar{\blambda})}{\partial w_i}+\lambda_1{\rm sgn}(w_i)\Theta^{\prime}(|w_i|)=0\ \mbox{ ($i\notin I(\w^{\ast})$)}$, $w_i=0\ \mbox{ ($i\in I(\w^{\ast})$)}$
, and $\blambda\ge \bm 0$:
\begin{eqnarray}
    &\nabla f(\w^{\ast})+\sum_{i\notin I(\w^{\ast})}\H_i(\w^{\ast},\blambda^{\ast})\zeta^{\ast}_i=\0,\label{eqa:1214-1_ex}\\
  &\frac{\partial G(\w^{\ast},\bar{\blambda}^{\ast})}{\partial w_i}+\lambda_1^{\ast}{\rm sgn}(w_i^{\ast})\Theta^{\prime}(|w_i^{\ast}|)=0\ \mbox{$(i\notin I(\w^{\ast}))$},
          \label{eqa:1214-2_ex}\\
  & \sum_{i\notin I(\w^{\ast})}{\rm sgn}(w_i^{\ast})
\Theta^{\prime}(|w^{\ast}_i|)\zeta_i^{\ast}=\eta_1^{\ast},\label{eqa:1214-3_ex}\\
&\zeta_i^{\ast}=0\ (i\in I(\w^{\ast})),\label{eqa:1214-6_ex}\\
&\nabla R_i(\w^{\ast})^{\top}\bzeta^{\ast}=\eta_i^{\ast}\ \ (i=2,3,\ldots,r),\label{eqa:1224-4_ex}\\ 
&\0\le \blambda^{\ast}, ~\0 \le \bbeta^{\ast}, ~(\blambda^{\ast})^\top \bbeta^{\ast}=0,\label{eqa:1214-5_ex}
\end{eqnarray}
where 
$$
\H_i(\w,\blambda):=\nabla_{\w}\left(\frac{\partial{G(\w,\bar{\blambda})}}{\partial w_i}\right)+\lambda_1{\rm sgn}(w_i)\Theta^{\prime\prime}(|w_i|){\bm e}^i\in \R^n \ \ (i\notin I(\w^{\ast})).
$$         \end{theorem}
        When $\Theta(x)=x^p$ with $x\ge 0$ and $0<p\le 1$, conditions\,\eqref{eqa:1214-1_ex} and \eqref{eqa:1214-2_ex} premultiplied by ${\rm diag}(\w^{\ast})^2$ and ${\rm diag}(\w^{\ast})$, respectively, are equivalent to \eqref{eqa:1214-1} and \eqref{eqa:1214-2} under the presence of \eqref{eqa:1214-6_ex} and $w_i^{\ast}=0\ (i\in I(\w^{\ast}))$. 
        The above theorem is different from Theorem\,\ref{thm:optimality} in that $\Theta$ is additionally assumed to be $C^2$ at $|w^{\ast}_i|$ for $i\notin I(\w^{\ast})$.
        This is due to the existence of the term $\Theta^{\prime\prime}$ in $\H_i$.
        If $\Theta$ is chosen to correspond to $\ell_p$ or $\log$-regularizer, this assumption always holds. 
        In contrast, if $\Theta$ is selected to correspond to SCAD (resp., MCP), it is equivalent to $w_i^{\ast}\neq \b,a\b$ (resp., $w_i^{\ast}\neq ab$) for $i\notin I(\w^{\ast})$ and thus may fail to hold in general.
        Though it is expected to hold in many instances, we need to do a further research so as to remove or weaken it.

        It is easy to tailor the proposed smoothing method to problems having other regularizers such as SCAD and MCP. Convergence properties similar to the case of using the $\ell_p$-regularizer hold expectedly.
However, proofs for the global convergence to an SB-KKT point in the sense of Theorem\,\ref{thm:optimality_ex} may differ significantly from that in Section~\ref{sec:theory}, because our analysis for the $\ell_p$-regularizer actually relies on the specific forms of the smoothing function
$\varphi_\mu(\bm w)=\sum_{i=1}^n (w_i^2+\,\mu^2)^{\frac{p}{2}}$ and its first- and second-order derivatives.

\paragraph{Further extension:}
Besides the above, there are other directions for extending our results.
One direction is extension to structured sparse regularizers like the group Lasso model\,\citep{yuan2006model}.
Such a model often contains regularizers of the composite form $\sum_{i=1}^l\Theta(\theta_i(\w))$ with $\theta_i:\R^n\to \R\ (i=1,2,\ldots,l)$.
For instances of $\theta_i$, we can set $\theta_i(\w):=\w^{\top}{\bm a^i}$ with ${\bm a^i}\in \R^n$ or $\theta_i(\w):=\w^{\top}\bm K_i\w$ with $\bm K_i\in \R^{n\times n}$ being a symmetric positive definite matrix.

Another interesting direction is extension to problems with matrix variables. \cite{marjanovic2012l_q} considered regularized least square optimization for matrix completion.
For $\bm X\in \R^{r_1\times r_2}$, the regularization term that appears there takes the form of $\|\bm X\|^p_p:=\sum_{i=1}^{\min(r_1,r_2)}\sigma_i(\bm X)^p$ with $0<p\le 1$, where $\sigma_i(\bm X)\ (i=1,2,\ldots,\min(r_1,r_2))$ are the singular values of $X$. If $r_1=r_2$ and $X$ is a diagonal matrix, $\|\bm X\|^p_p$ reduces to the $\ell_p$-regularizer we have considered.  
In \citep{marjanovic2012l_q}, in order to find the best-qualified recovered matrix model, the authors iteratively solved problems involving $\|\bm X\|^p_p$ as a regularizer while varying hyperparameters. A bilevel approach may help to recover a matrix with higher quality faster.
\section{Conclusions}\label{sec:7}
We have proposed a bilevel optimization approach for selecting {the best hyperparameter (regularization parameter) of the $\ell_p$-regularizer}.
The bilevel optimization problem that appears in our approach 
has a nonsmooth and possibly nonconvex $\ell_p$-regularized problem as the lower-level problem. 
For this problem, we have developed the scaled bilevel KKT (SB-KKT) conditions and proposed a smoothing-type method.
Furthermore, we have made analysis on convergence of the proposed algorithm to an SB-KKT point.
Numerical experiments imply that it exhibited performance superior to Bayesian optimization and grid search especially in computational time.

The method/theoretical guarantee can be applicable to hyperparameter learning for classification. 
As a future work, we would like 
to make the algorithm more practical.
For this purpose, 
we may need to integrate some stochastic technique into the proposed algorithm. 
For example,
approximate KKT points computed by approximate gradient and Hessians can be used. In the stochastic setting,
we expect that the SB-KKT conditions will play a significant role in convergence analysis.
%

\subsection*{Acknowledgments}
We thank three anonymous reviewers and the editor for their valuable comments and suggestions.\\
This research was supported by JSPS KAKENHI Grant Numbers 20K19748 and 19H04069 and by JST ERATO Grant Number JPMJER1903.
\appendix
\renewcommand{\theequation}{A.\arabic{equation}}
\renewcommand{\thealgorithm}{A.\arabic{algorithm}}
\renewcommand{\thetheorem}{A.\arabic{theorem}}
\renewcommand{\thefigure}{\Alph{section}.\arabic{figure}} 
\renewcommand{\thetable}{\Alph{section}.\arabic{table}}
\def\thesection{\Alph{section}}
\setcounter{equation}{0}
\setcounter{theorem}{0}
\setcounter{algorithm}{0}
\setcounter{figure}{0}
\setcounter{table}{0}
\section{Omitted Proofs}
In this section, we provide proofs of some lemmas and propositions.
\subsection{Proof of Theorem~\ref{thm:optimality}}\label{appendix:kkt}
Firstly, notice that $(\w^{\ast},\blambda^{\ast})$ is also a local optimum of the following problem:
\begin{align}\label{eqn:surrogate2:app}
\begin{array}{rcc}
\displaystyle\min_{\w,\blambda}&  & f(\w)  \\
\mbox{s.t.} &     &\displaystyle \frac{\partial{G(\w,\bar{\blambda})}}{\partial w_i}+p\,{\rm sgn}(w_i)\lambda_1|w_i|^{p-1}=0\ (i\notin I(\w^{\ast}))\\
            &      &w_i=0\ (i\in I(\w^{\ast}))\\  
           &       &\blambda\ge \0.
\end{array}
\end{align}
Actually, 
this fact is easily confirmed by noting that 
$(\w^{\ast},\blambda^{\ast})$ is also feasible to \eqref{eqn:surrogate2:app} and the feasible region of 
\eqref{eqn:surrogate}
is larger than that of \eqref{eqn:surrogate2:app}.
Hence, under an appropriate constraint qualification such as
the linearly independent constraint qualification associated to \eqref{eqn:surrogate2:app}, the KKT conditions for \eqref{eqn:surrogate2:app} hold at $(\w^{\ast},\blambda^{\ast})$, i.e.,
there exist some vectors 
$\hat{\bzeta}^{\ast}:=(\hzeta_1^{\ast},\hzeta_2^{\ast},\ldots,\hzeta_n^{\ast})^{\top}\in \R^n$
and $\bbeta^{\ast}\in\R^r$ such that 
\begin{align}
&\displaystyle\frac{\partial f(\w^{\ast})}{\partial w_i}+\sum_{j\notin I(\w^{\ast})}
{\left(\frac{\partial^2G(\w,\bar{\blambda})}{\partial w_i\partial w_j}+p(p-1)\lambda_1|w_i|^{p-2}\right)}
\hzeta_j^{\ast}=0\hspace{0.5em}(i\notin I(\w^{\ast})),\label{eqn:1218-1}\\
&\displaystyle\frac{\partial f(\w^{\ast})}{\partial w_i}+\sum_{j\notin I(\w^{\ast})}\frac{\partial^2G(\w^{\ast},\bar{\blambda}^{\ast})}{\partial w_i\partial w_j}\hzeta^{\ast}_j+\hzeta^{\ast}_i=0\hspace{0.5em}(i\in I(\w^{\ast})),\label{eqn:1218-2} \\
&\nabla_{\blambda}f(\w^{\ast})-\bbeta^{\ast}+\sum_{i\notin I(\w^{\ast})}\frac{\partial}{\partial{\blambda}}\left(\frac{\partial{G({\w^{\ast}},{\bar{\blambda}^{\ast}})}}{\partial w_i}+p\,{\rm sgn}(w_i)\lambda_1|w_i|^{p-1}\right){\hat{\bzeta}^{\ast}_i}=\0,\label{al:0628-1} \\
&\displaystyle\frac{\partial{G(\w^{\ast},\bar{\blambda}^{\ast})}}{\partial w_i}+p\,{\rm sgn}(w^{\ast}_i)\lambda^{\ast}_1|w^{\ast}_i|^{p-1}=0\hspace{0.5em}(i\notin I(\w^{\ast})),\label{eqn:1218-5}\\
&w^{\ast}_i=0\ \ \ (i\in I(\w^{\ast})),\label{eqn:1218-6}\\
& \0\le \blambda^{\ast}, ~\0 \le \bbeta^{\ast}, ~(\blambda^{\ast})^\top \bbeta^{\ast}=0,\label{eqn:1218-7}
\end{align}
where
{$\hzeta_i^{\ast}\ (i\in I(\w^{\ast}))$, $\hzeta_i^{\ast}\ (i\notin I(\w^{\ast}))$, 
and $\bbeta^{\ast}$
are Lagrange multipliers corresponding to the constraints 
$w_i=0\ (i\in I(\w^{\ast})$ and 
$\frac{\partial{G(\w,\bar{\blambda})}}{\partial w_i}+p\,{\rm sgn}(w_i)\lambda_1|w_i|^{p-1}=0\ (i\notin I(\w^{\ast}))$, 
and $\blambda\ge \0$, respectively.}
To derive the first equality above, we made use of the fact 
\begin{equation*}
\frac{\partial{|w_i|^{p-1}}}{\partial w_i}= (p-1) {\rm sgn}(w_i)|w_i|^{p-2} 
\end{equation*}
at $w_i\neq 0$.
{Noting the relations}
$\nabla_{\blambda}f(\w)=\0$, $\partial G(\w,\bar{\blambda})/\partial \lambda_1=0$, 
$\partial^2 G(\w,\bar{\blambda})/\partial \lambda_i\partial w_i=\partial R_i(\w,\blambda)/\partial w_i\ (i=2,3,\ldots,r)$,
{we can rewrite condition\,\eqref{al:0628-1} as}
\begin{align}
&\displaystyle \sum_{i\notin I(\w^{\ast})}p\,{\rm sgn}(w^{\ast}_i)|w^{\ast}_i|^{p-1}\hzeta^{\ast}_i=\eta^{\ast}_1,\label{eqn:1218-3}\\
&\displaystyle\sum_{j\notin I(\w^{\ast})}\frac{\partial{R_i(\w)}}{\partial w_j}\hzeta^{\ast}_j=\eta^{\ast}_i\ \ (i=2,\ldots,r).\label{eqn:1218-4}
\end{align}

Next, define $\bzeta^{\ast}\in \R^n$ as the vector with $\zeta^{\ast}_i=0\ (i\in I(\w^{\ast}))$ and $\zeta^{\ast}_i=\hzeta^{\ast}_i\ (i\notin I(\w^{\ast}))$.
Let us show that $(\w^{\ast},\blambda^{\ast},\bzeta^{\ast},\bbeta^{\ast})$ satisfies the targeted conditions\,\eqref{eqa:1214-1}--\eqref{eqa:1214-5}.
For each $i\in \{1,2,\ldots,n\}$, we have 
\begin{align}
&(w_i^{\ast})^2\frac{\partial f(\w^{\ast})}{\partial w_i}+
(w_i^{\ast})^2\sum_{j=1}^n\frac{\partial^2 G(\w^{\ast},\bar{\blambda}^{\ast})}{\partial w_i\partial w_j}\zeta_{j}^{\ast}+\lambda_1^{\ast}p(p-1)|w_i^{\ast}|^p\zeta_i^{\ast}\notag\\
&=(w_i^{\ast})^2\frac{\partial f(\w^{\ast})}{\partial w_i}+
(w_i^{\ast})^2\sum_{j\notin I(\w^{\ast})}\frac{\partial^2 G(\w^{\ast},\bar{\blambda}^{\ast})}{\partial w_i\partial w_j}\zeta_{j}^{\ast}+\lambda_1^{\ast}p(p-1)|w_i^{\ast}|^p\zeta_i^{\ast}\notag\\
&=0,\notag
\end{align}
where the first equality follows from $\zeta_{j}^{\ast}=0\ (j\in I(\w^{\ast}))$ and the second one can be proved by cases; when $i\in I(\w^{\ast})$, the desired equality is obviously true because of \eqref{eqn:1218-6}; when $i\notin I(\w^{\ast})$, it is obtained from multiplying \eqref{eqn:1218-1} by $(w_i^{\ast})^2$ and using $\zeta_{i}^{\ast}=\hzeta_{i}^{\ast}$.
Therefore, we confirm \eqref{eqa:1214-1}.
Similarly, we can deduce \eqref{eqa:1214-2} and \eqref{eqa:1224-4} from \eqref{eqn:1218-5} and \eqref{eqn:1218-4} along with the definition of $\bzeta^{\ast}$, respectively.
The remaining conditions\,\eqref{eqa:1214-3}, \eqref{eqa:1214-6}, and \eqref{eqa:1214-5} are derived from 
\eqref{eqn:1218-3}, $\zeta_{i}^{\ast}=0\ (i\in I(\w^{\ast}))$, and \eqref{eqn:1218-7}, respectively.
Putting all the above results together, we confirm that $(\w^{\ast},\blambda^{\ast},\bzeta^{\ast},\bbeta^{\ast})$ satisfies \eqref{eqa:1214-1}--\eqref{eqa:1214-5}.
Consequently, we have the desired result.
\hfill$\blacksquare$

\subsection{Proof of Lemma\,\ref{lem:0708-2352}}\label{appendix:lemma}
We will give a proof of Lemma\,\ref{lem:0708-2352}. Firstly, we review the definition and several properties for a subgradient of a given function from \cite{rockafellar2009variational}. We finally give a proof of Lemma\,\ref{lem:0708-2352}.

Let us define regular and general subgradients for a given function according to \cite[8.3(a),(b)~Definition]{rockafellar2009variational}.
For simplicity, we confine ourselves to a continuous function $f:\R^n\to \R$.
\begin{definition}
For vectors $\bm v\in \R^n$ and $\bar{\bm x}\in \R^n$, 
\begin{enumerate}
\item we say that $\bm v$ is a regular subgradient of $f$ at $\bar{\bm x}$, written $\bm v\in \hat{\partial}_{\bm x}f(\bar{\bm x})$, if 
$f(\bm x)\ge f(\bar{\bm x})+\bm v^{\top}(\bm x-\bar{\bm x})+o(\|\bm x-\bar{\bm x}\|)$.
\item We say that $\bm v$ is a (general) subgradient of $f$ at $\bar{\bm x}$, written $\bm v\in \partial_{\bm x} f(\bar{\bm x})$, if there are  sequences $\{\bm x^{\nu}\}\subseteq \R^n$ 
converging to $\bar{\bm x}$ and $\{\bm v^{\nu}\}\subseteq \R^n$ 
converging to $\bm v$
such that $\bm v^{\nu}\in \hat{\partial}f(\bm x^{\nu})$ for each $\nu$. 
\end{enumerate}
We often simply write $\hat{\partial}_{\bm x}$ and $\partial_{\bm x}$ as $\hat{\partial}$ and $\partial$, respectively. 
\end{definition}
Obviously, it holds that $\hat{\partial} f(\bm x)\subseteq \partial f(\bm x)$. 

The following propositions are useful: 
\begin{proposition}
\cite[8.8(c)~Exercise]{rockafellar2009variational}
\label{prop:0707-1}
Let $f_i:\R^n\to \R\ (i=0,1)$ be continuous. 
Let $f:=f_0+f_1$. 
If $f_0$ is continuously differentiable around $\bar{\bm x}$, then 
$\hat{\partial} f(\bar{\bm x})=\nabla f_0(\bar{\bm x}) + \hat{\partial} f_1(\bar{\bm x})$ and 
$\partial f(\bar{\bm x})=\nabla f_0(\bar{\bm x})+\partial f_1(\bar{\bm x})$.  
\end{proposition} 
\begin{proposition}\cite[8.5 Proposition]{rockafellar2009variational}\label{prop:0708-2}
Let $f:\R^n\to \R$ be continuous.
Then, $\bm v\in \hat{\partial}f(\bm x)$ if and only if, 
on some neighborhood of $\bar{\bm x}$, 
there exists a differentiable function $g:\R^n\to \R$ such that 
$\nabla g(\bar{\bm x})=\bm v$, $g(\bm x)\le f(\bm x)$, and $g(\bar{\bm x})=f(\bar{\bm x})$. 
Moreover, $g$ can be taken to be continuously differentiable with $g(\bm x)<f(\bm x)$
for all $\bm x \neq \bar{\bm x}$ near $\bar{\bm x}$.
\end{proposition}
We next prove the following proposition associated with $\|\bm x\|_p^p\ (0<p\le 1)$.
\begin{proposition}\label{prop:0709-109}
For $\bm x\in \R^n$, let $I(\bm x):=\{i\mid x_i=0\}$ and $g(\bm x):=\lambda\|\bm x\|_p^p$ with $0<p\le 1$ and $\lambda\ge 0$.
Then, for $0<p<1$ and $\bar{\bm x}\in \R^n$, we have  
\begin{equation}
\partial g(\bar{\bm x}) = \left\{\bm v\mid 
v_i = \lambda p\,{\rm sgn}(\bar{x}_i)|\bar{x}_i|^{p-1}\ (i\notin I(\bar{\bm x})),\ v_i\in \R\ (i\in I(\bar{\bm x}))\right\}.
\label{eq:0708-2239}
\end{equation}
On the other hand, for $p=1$, we have 
\begin{equation}
\partial g(\bar{\bm x}) = \left\{\bm v\mid 
v_i = \,\lambda\,{\rm sgn}(\bar{x}_i)\ (i\notin I(\bar{\bm x})), 
v_i\in [-\lambda,\lambda]\ (i\in I(\bar{\bm x}))\right\}.
\label{eq:0708-2240}
\end{equation}
\end{proposition}
\begin{proof}
For convenience of expression, let $\hat{g}(\bm x):=\lambda\sum_{i\in I(\bar{\bm x})}|x_i|^p$.
Note that 
$
g(\bm x) = \hat{g}(\bm x) + \lambda\sum_{i\notin I(\bar{\bm x})}|x_i|^p
$
and $\lambda\sum_{i\notin I(\bar{\bm x})}|x_i|^p$ is continuously differentiable around $\bar{\bm x}$. Then, by Proposition\,\ref{prop:0707-1}, we have
\begin{equation}
\partial g(\bar{\bm x}) = \lambda\sum_{i\notin I(\bar{\bm x})}p\,{\rm sgn}(\bar{x}_i)|\bar{x}_i|^{p-1}{\bm e}^i
+\partial\hat{g}(\bar{\bm x}),
\label{eq:0708-2242}
\end{equation}
where ${\bm e}^i\in \R^n$ is the vector such that the $i$-th element 
is one and the others are zeros.
Supposing $I(\bar{\bm x})\neq \emptyset$, we next  describe 
$\partial_{\bm x}\hat{g}(\bar{\bm x})$ precisely. 
First, consider the case of $0<p<1$.
For any $\bm v\in \R^{n}$ with 
$v_i = 0\ (i\notin I(\bar{\bm x}))$, we can show that 
$\lambda\sum_{i\in I(\bar{\bm x})}|x_i|^p\ge 
\lambda\sum_{i\in I(\bar{\bm x})}v_ix_i$
holds on a sufficiently small neighborhood of $\bar{\bm x}$ since $\lambda\ge 0$.
Then, Proposition\,\ref{prop:0708-2} implies 
\begin{equation}
\hat{\partial}\hat{g}(\bar{\bm x})
\supseteq 
\left\{
{\bm v}\mid 
v_i = 0\ (i\notin I(\bar{\bm x}))
\right\}.
\label{eq:0707-1308}
\end{equation}
We next show the converse implication for the above.
To this end, choose a regular subgradient $\bm v\in 
\hat{\partial}\hat{g}(\bar{\bm x})
=
\left.\hat{\partial}\left(\lambda\sum_{i\in I(\bar{\bm x})}|x_i|^p\right)\right|_{\bm x = \bar{\bm x}}$ arbitrarily. Then,
according to Proposition\,\ref{prop:0708-2},
there exists some 
differentiable function $h$ such that $h(\bm x)\le \lambda\sum_{i\in I(\bar{\bm x})}|x_i|^p$ near $\bar{\bm x}$, $h(\bar{\bm x})=\lambda\sum_{i\in I(\bar{\bm x})}|\bar{x}_i|^p=0$, and $\nabla h(\bar{\bm x})=\bm v$.
Then, for arbitrarily chosen $j\notin I(\bar{\bm x})$, 
$h(\bar{\bm x}+s\bm e^j)\le \lambda\sum_{i\in I(\bar{\bm x})}|\bar{x}_i|^p= 0$ for any $s\in \R$ sufficiently small.
From this fact along with $h(\bar{\bm x})=0$, we see that 
$s=0$ is a local maximizer of $\max_{s\in \R}h(\bar{\bm x}+s\bm e^j)$, and thus 
$v_j=\partial h(\bar{\bm x})/\partial x_j = 
\partial h(\bar{\bm x}+s\bm e^j)/\partial s|_{s=0}=0$. Hence, 
since the index $j\in I(\bar{\bm x})$ was arbitrarily chosen, 
we obtain the converse implication for \eqref{eq:0707-1308}.
Using this fact and \eqref{eq:0707-1308},
we have 
\begin{equation}
\hat{\partial}\hat{g}(\bar{\bm x})
=
\left\{
{\bm v}\mid 
v_i = 0\ (i\notin I(\bar{\bm x}))
\right\}.
\label{eq:0707-2058}
\end{equation}
We next prove that 
\begin{equation}
\partial \hat{g}(\bar{\bm x})\subseteq \left\{
{\bm v}\mid 
v_i = 0\ (i\notin I(\bar{\bm x}))
\right\}.\label{eq:0708-2235}
\end{equation}
Choose $\bm v\in \partial{\hat{g}}(\bar{\bm x})$ arbitrarily.
Then, there exist sequences $\{\bm x^{\nu}\}$ and $\{\bm v^{\nu}\}$
such that $\lim_{\nu\to \infty}\bm x^{\nu} = \bar{\bm x}$, $\lim_{\nu\to \infty}\bm v^{\nu} = \bm v$, and $\bm v^{\nu}\in \hat{\partial} \hat{g}(\bm x^{\nu})$ for any $\nu$. 
For an arbitrary $j\notin I(\bar{\bm x})$, it is not difficult to verify
$\bm v^{\nu}_j=0$ for all $\nu$ sufficiently large. Therefore, we obtain $v_j =0$ for any $j\notin I(\bar{\bm x})$. Thus, we conclude \eqref{eq:0708-2235} which together with the facts of 
$
\hat{\partial}\hat{g}(\bar{\bm x})\subseteq \partial \hat
{g}(\bar{\bm x})
$ 
and \eqref{eq:0707-2058}
implies 
\begin{equation}
\partial \hat{g}(\bar{\bm x})=\left\{
{\bm v}\mid 
v_i = 0\ (i\notin I(\bar{\bm x}))
\right\}.\notag 
\end{equation}
Finally, from this equality and \eqref{eq:0708-2242}, we obtain the desired result\,\eqref{eq:0708-2239}.

For the case where $p=1$, 
it is easy to show the desired result\,\eqref{eq:0708-2240}
using the fact of 
$
\partial \hat{g}(\bar{\bm x})=\lambda\sum_{i\in I(\bar{\bm x})}
\left.\partial_{\bm x}\left|x_i\right|\hspace{0.1em}\right|_{\bm x=\bar{\bm x}}.
$
We omit the detailed proof.
\end{proof}
We are now ready to show Lemma\,\ref{lem:0708-2352}.\vspace{0.5em}\\
{\bf Proof of Lemma\,\ref{lem:0708-2352}:}
We first note that, since $G$ is continuously differentiable and $R_1(\w)=\|\w\|_p^p$ and $\lambda_1\ge 0$, we have 
\begin{align}
&\partial_{\bw}\left(G(\w,\bar{\blambda})+\lambda_1R_1(\w)\right)\notag \\
=&\nabla_{\bw}G(\w,\bar{\blambda}) + \partial_{\w}(\lambda_1R_1(\w))\notag \\
=&
\begin{cases}
\left\{\bm v\mid 
v_i=\frac{\partial G(\w,\bar{\blambda})}{\partial w_i} + 
\lambda_1 p\,{\rm sgn}(w_i)|w_i|^{p-1}\ (i\notin I(\w)),
\ v_i\in \R\ (i\in I(\w))
\right\} &(p<1)\vspace{0.5em}\\
\left\{\bm v\mid 
v_i=\frac{\partial G(\w,\bar{\blambda})}{\partial w_i}+\lambda_1 {\rm sgn}(w_i)\ (i\notin I(\w)),
v_i\in \frac{\partial G(\w,\bar{\blambda})}{\partial w_i}+[-\lambda_1,\lambda_1]\ (i\in I(\w))
\right\} &(p=1),
\end{cases}
\label{al:0718}
\end{align}
where the first equality follows from Proposition\,\ref{prop:0707-1}
and the second equality comes from Proposition\,\ref{prop:0709-109}.

{Now, let us show the first claim.
Suppose $\bm 0\in \partial_{\w}(G(\w,\bar{\blambda})+\lambda_1R_1(\w))$.
Then, by \eqref{al:0718}, we have 
\begin{align}
&w_i = 0\ (i\in I(\w)), \label{al:0709-031} \\ 
&\displaystyle \frac{\partial{G(\w,\bar{\blambda})}}{\partial w_i}+p\,{\rm sgn}(w_i)\lambda_1|w_i|^{p-1}=0\ (i\notin I(\w)),\label{al:0709-032}
\end{align}
which readily imply $\bW \nabla_{\w}G(\w,\bar{\blambda})+p\lambda_1|\bw|^p= \bm 0$.
Hence, we obtain the first claim. 

We next show the latter claim for the case of $p<1$. 
Suppose that 
$\bW \nabla_{\w}G(\w,\bar{\blambda})+p\lambda_1|\bw|^p= \bm 0$. 
Then, we see that \eqref{al:0709-031} and \eqref{al:0709-032} hold. In view of this fact together with 
\eqref{al:0718} for $p<1$, we obtain $\bm 0\in \partial_{\w}(G(\w,\bar{\blambda})+\lambda_1R_1(\w))$. Thus, we conclude the latter claim.}
\\
\hfill$\blacksquare$

{\subsection{Proof of Proposition~\ref{lem:1217-1}}
\label{appendix:Lemma2}
Denote $\w^k=(w^k_1,w^k_2,\ldots,w^k_n)^{\top}$ for each $k$.
We first show \eqref{al:1214-1921}. 
Note that it follows from \eqref{al:1214-1} that 
\begin{equation}
w_i^k(\nabla\bphi_{\mu_{k-1}}(\w^k))_i=p(w_i^k)^2((w^k_i)^2+\mu_{k-1}^2)^{\frac{p}{2}-1}
\label{eq:1214-2030}
\end{equation}
for each $i\in \{1,2,\ldots,n\}$.
Then, for the index $i\notin I(\bw^{\ast})$, 
we have $w^{\ast}_i\neq 0$ and thus get 
\begin{align}
\lim_{k\to \infty}w_i^k(\nabla\bphi_{\mu_{k-1}}(\w^k))_i&=p(w_i^{\ast})^2|w_i^{\ast}|^{p-2} \notag\\
                    &=p|w_i^{\ast}|^{p}. \label{al:1214-2047}
\end{align}
We next choose $i\in I(\bw^{\ast})$ arbitrarily and
divide the index set $K:=\{1,2,\ldots,\}$ into the following two sets:
$$
U_1^i:=\{k\in K\mid w^k_i\neq 0\},\
U_2^i:=\{k\in K\mid w^k_i=0\}.
$$
Then, 
for $k\in U_1^i$, equation\,\eqref{eq:1214-2030} together with $p/2 - 1< 0$ and
$w^k_i\neq 0$ yields that 
\begin{align}
w_i^k(\nabla\bphi_{\mu_{k-1}}(\w^k))_i&\le p|w_i^k|^2|w_i^k|^{2(\frac{p}{2}-1)}\notag\\
                                                   &=p|w_i^k|^p. \label{al:1214-2036} 
\end{align} 
Since $w_i^k(\nabla\bphi_{\mu_{k-1}}(\w^k))_i\ge 0$ holds for each $k\in U_1^i$ in view of the right-hand of \eqref{eq:1214-2030} and $\lim_{k\to \infty}\mu_{k-1}=0$,
letting $k\in U_1^i\to\infty$ in \eqref{al:1214-2036} implies 
\begin{equation}
\lim_{k\in U^i_1\to\infty}w_i^k(\nabla\bphi_{\mu_{k-1}}(\w^k))_i=p|w^{\ast}_i|^{p}=0.\label{eq:1214-2240-0}
\end{equation}
Similarly, for all $k \in U_2^i$, we have $w_i^k(\nabla\bphi_{\mu_{k-1}}(\w^k))_i=0$  because of $w^k_i=0$ $(k\in U^i_2)$ and \eqref{al:1214-1}.
This fact together with \eqref{eq:1214-2240-0} yields 
\begin{equation}
\lim_{k\to\infty}w_i^k(\nabla\bphi_{\mu_{k-1}}(\w^k))_i=p|w^{\ast}_i|^{p}.\label{eq:1214-2240}
\end{equation}
Combining this with \eqref{al:1214-2047}, we conclude \eqref{al:1214-1921}.

We next show \eqref{al:1214-1922}.
In view of \eqref{al:1214-2}, we have 
\begin{align}
&(w^k_i)^2(\nabla^2\bphi_{\mu_{k-1}}(\w^k))_{ii}\notag \\
&=p(w^k_i)^2((w^k_i)^2+\mu_{k-1}^2)^{\frac{p}{2}-1}\notag\\
                     &\hspace{3em}+p(p-2)(w^k_i)^4((w^k_i)^2+\mu_{k-1}^2)^{\frac{p}{2}-2}\notag\\
&=\left(1+
\frac{(p-2)(w^k_i)^2}{(w^k_i)^2+\mu_{k-1}^2}\right)
\left(w^k_i\left(\nabla\bphi(\w^k)\right)_i\right),\label{al:1214-2208}
\end{align}
for any $i=1,2,\ldots,n$, where the last equality is due to \eqref{al:1214-1}.
For the case of $i\notin I(\bw^{\ast})$, we obtain 
\begin{equation*}
\lim_{k\to\infty}\frac{(w^k_i)^2}{(w^k_i)^2+\mu_{k-1}^2}=1, 
\end{equation*}
which together with \eqref{al:1214-2047} and \eqref{al:1214-2208} implies 
\begin{equation}
\lim_{k\to\infty}(w^k_i)^2(\nabla^2\bphi_{\mu_{k-1}}(\w^k))_{ii}=p(p-1)|w_i^{\ast}|^p.\label{eq:1214-2230}
\end{equation}
In turn, let us focus on the case of $i\in I(\bw^{\ast})$.
Then, the sequence $\{(w^k_i)^2/{\left((w^k_i)^2+\mu_{k-1}^2\right)}\}$ is bounded since $|(w^k_i)^2/{\left((w^k_i)^2+\mu_{k-1}^2\right)}|<1$ follows from $\mu_{k-1}>0$ for all $k$.
Hence, using \eqref{eq:1214-2240}, we derive from \eqref{al:1214-2208} that
\begin{equation*}
\lim_{k\to\infty}(w^k_i)^2(\nabla^2\bphi_{\mu_{k-1}}(\w^k))_{ii}=0=p(p-1)|w_i^{\ast}|^p,
\end{equation*}
where the last equality is due to $w_i^{\ast}=0$ for $i\in I(\w^{\ast})$.
By this equation together with \eqref{eq:1214-2230}, we conclude \eqref{al:1214-1922}.
The proof is complete.\\
\hfill$\blacksquare$
\subsection{Proof of Lemma~\ref{prop:0312}}\label{appendix:speed}
Choose $i\in I(\w^{\ast})$ arbitrarily. 
{We show the claim for the case where} $w^k_i \neq 0$ for all $k\in K$.
{It is not difficult to extend the argument to the general case where $w^k_i=0$ occurs for infinitely many $k$.}
Also, we may assume $\lambda_1^k>0$ for all $k\in K$ because of Assumption~A1. 
For simplicity, denote 
$$
F_i(\w^{k},\bar{\blambda}^{k}):=\frac{\partial G(\w^{k},\bar{\blambda}^{k})}{\partial w_i}
$$
for each $k\in K$.
From 
\eqref{al:1214-1} and the $i$-th element of
condition\,\eqref{eq:1215-2} with 
$(\w,\blambda,{\bm \varepsilon}_4)=(\w^k,\blambda^k,{\bm \varepsilon}_{4}^{k-1})$, we have, for each $k\in K$, 
\begin{equation}
F_i(\w^{k},\bar{\blambda}^{k}) +p\lambda_1^kw_i^k((w_i^k)^2+\mu_{k-1}^2)^{\frac{p}{2}-1}=(\bm \varepsilon^{k-1}_4)_i,\label{eq:0312-0}
\end{equation}
which together with the assumption $w^k_i\neq 0$ and $\lambda^k_1\neq 0\ (k\in K)$ implies
\begin{equation}
F_i(\w^{k},\bar{\blambda}^{k})-(\bm \varepsilon^{k-1}_4)_i\neq 0\ (k\in K).\label{eq:0402-1}
\end{equation}
Recall that ${\bm \varepsilon}^{k-1}_4\to {\bm 0}$ as $k\to \infty$.
Noting this fact and \eqref{eq:0312-0}, we get 
\begin{equation}
\mu_{k-1}^2=\frac{\left|F_i(\w^{k},\bar{\blambda}^{k})-({\bm \varepsilon}^{k-1}_4)_i\right|^{\temp}}{\tilde{p}\tilde{\lambda}_1^k|w_i^k|^{\temp}} - (w^k_i)^2, \label{eq:0312}
\end{equation}
where 
$$\tilde{p}:=p^{\temp},\ \tilde{\lambda}_1^k:=({\lambda}_1^k)^{\temp}.$$
Then, it follows that 
\begin{align}
\frac{|w_i^k|^{\mtemp}}{\mu_{k-1}^2}&=\frac{
\tilde{p}\tilde{\lambda}_1^k
}{
\left|F_i(\w^{k},\bar{\blambda}^{k}){-{\bm (\varepsilon}^{k-1}_4)_i}\right|^{\temp}
-\tilde{p}\tilde{\lambda}_1^k|w^k_i|^{2+\temp}}.\label{al:0313-1}
\end{align}

To show the desired result, it suffices to prove that  
$\left\{{|w_i^k|^{\mtemp}}/{\mu_{k-1}^2}\right\}_{k\in K}$ is bounded from above.
To this end, we first consider the case of $p=1$. 
By substituting $p=1$ for \eqref{eq:0312-0}, we get
\begin{align}
F_i(\w^{k},\bar{\blambda}^{k}){-({\bm \varepsilon}^{k-1}_4)_i}+\lambda_1^k\frac{w_i^k}{\sqrt{(w_i^k)^2+\mu_{k-1}^2}}=0.\label{al:0402-2324}
\end{align}
Moreover, by substituting $p=1$ for \eqref{al:0313-1}, 
we have 
\begin{align}
\frac{|w_i^k|^{2}}{\mu_{k-1}^2}&=\frac{(\lambda_1^k)^{-2}}{
\left|F_i(\w^{k},\bar{\blambda}^{k}){-({\bm \varepsilon}^{k-1}_4)_i}\right|^{-2}
-(\lambda_1^k)^{-2}}\notag \\
&=\displaystyle{\frac{\left|F_i(\w^{k},\bar{\blambda}^{k}){-({\bm \varepsilon}^{k-1}_4})_i\right|^2}{({\lambda}_1^{k})^2-\left|F_i(\w^{k},\bar{\blambda}^{k}){-({\bm \varepsilon}^{k-1}_4)_i}\right|^2}}.
\label{al:0313-2}
\end{align}
From equation\,\eqref{al:0402-2324}, it is not difficult to see that $|F_i(\w^{k},\bar{\blambda}^{k}){-({\bm \varepsilon}^{k-1}_4)_i}|^2\le |\lambda_1^k|^2$. In this inequality, let $k\in K\to \infty$. Then, Assumption~A3 
together with $F_i(\w^{\ast},\bar{\blambda}^{\ast})=
\partial G(\w^{\ast},\bar{\blambda}^{\ast})/\partial w_i$ yields 
\begin{equation}
(\lambda_1^{\ast})^2-|F_i(\w^{\ast},\bar{\blambda}^{\ast})|^2>0.\label{eq:0312-3}
\end{equation}
Letting $k\in K\to \infty$ in equation\,\eqref{al:0313-2} and noting \eqref{eq:0312-3}, we readily derive that
\begin{equation}
\lim_{k\in K\to\infty}\frac{|w_i^k|^{\mtemp}}{\mu_{k-1}^2}=
\displaystyle{\frac{\left|F_i(\w^{\ast},\bar{\blambda}^{\ast})\right|^2}{({\lambda}_1^{\ast})^2-\left|F_i(\w^{\ast},\bar{\blambda}^{\ast})\right|^2}}<\infty.\label{eq:0312-148}
\end{equation}
We next consider the case of $p<1$.
By using \eqref{al:0313-1} again, it holds that 
\begin{align}
\lim_{k\in K\to\infty}\frac{|w_i^k|^{\mtemp}}{\mu_{k-1}^2}
&=\frac{
\tilde{p}\tilde{\lambda}_1^{\ast}
}{
\left|F_i(\w^{\ast},\bar{\blambda}^{\ast})\right|^{\temp}
-\tilde{p}\tilde{\lambda}_1^{\ast}|w^{\ast}_i|^{2+\temp}}\notag\\
&=\frac{
\tilde{p}\tilde{\lambda}_1^{\ast}
}{\left|F_i(\w^{\ast},\bar{\blambda}^{\ast})\right|^{\temp}}\notag\\
&<\infty,\label{al:0312-214}
\end{align}
where $\tilde{\lambda}_1^{\ast}:=({\lambda}_1^{\ast})^{\temp}>0$ and the second equality follows from $2+\frac{2}{p-2}>0$ and  
$w^{\ast}_i=0$ because of $i\in I(\w^{\ast})$. 
Particularly, note that 
the last strict inequality is true due to $2/(p-2)<0$ even if $\left|F_i(\w^{\ast},\bar{\blambda}^{\ast})\right|=0$.
Finally, by \eqref{eq:0312-148} and \eqref{al:0312-214}, we conclude the desired result.\\
\hfill$\blacksquare$

\subsection{Proof of Proposition~\ref{prop:0607}}\label{appendix:boundedness}
We prepare the following lemma.
\begin{LEMM}\label{lem:0607}
Suppose that Assumption~A4 holds and let
$(\w^{\ast},\blambda^{\ast})$ be an arbitrary accumulation point of
the sequence  $\{
(\w^k,\blambda^k)\}$.
Recall that we write
$\nabla_{\tilde{\w}}h(\w):=\left(\frac{\partial h(\w)}{\partial w_{i_1}},\ldots,\frac{\partial h(\w)}{\partial w_{i_p}}\right)^{\top}\in \R^{p}$
for a function $h:\R^n\to \R$ and the index set $\{i_1,i_2,\ldots,i_p\}:=\{1,2,\ldots,n\}\setminus I(\w^{\ast})$.
Moreover, denote $\tilde{\w}:=(w_i)_{i\notin I(w^{\ast})}$ and 
{\begin{align}
\nabla_{(\tilde{\w},\blambda)}\Phi_i(\w,\blambda)&:=\begin{bmatrix}
\nabla_{\tilde{\w}}\Phi_i(\w,\blambda)\\
\nabla_{\blambda}\Phi_i(\w,\blambda)
\end{bmatrix}\in \R^{n-|I(\w^{\ast})|+r}
\ \ (i\notin I(\w^{\ast})),\label{eq:0713-1}\\
\nabla_{(\tilde{\w},\blambda)}\lambda_i&:=\begin{bmatrix}
\nabla_{\tilde{\w}}\lambda_i\\
\nabla_{\blambda}\lambda_i
\end{bmatrix}\in \R^{n-|I(\w^{\ast})|+r}
\ \ (i\in I(\blambda^{\ast})). \label{al:0714-2328}
\end{align}
Then, the vectors 
$$
\left\{
\left\{
\nabla_{(\tilde{\w},\blambda)}\Phi_i(\w^{\ast},\blambda^{\ast})\right\}_{i\notin I(\w^{\ast})},
\left\{
\nabla_{(\tilde{\w},\blambda)} \lambda_i|_{\blambda=\blambda^{\ast}}
\right\}_{i\in I(\blambda^{\ast})}
\right\}
$$}
are linearly independent.
\end{LEMM}
\begin{proof}
Notice that $\nabla_{(\tilde{\w},\blambda)}\lambda_i$
is 
the vector such that the   
$(n-|I(\w^{\ast})|+i)$-th entry is 1 and the others are 0s.
Under Assumption~A4, we see that the matrix
$$
{\bm M}:=\left[
\left(\nabla \Phi_i(\w^{\ast},\blambda^{\ast})\right)_{i\notin I(\w^{\ast})}, 
(\nabla_{(\w,\blambda)} w_i|_{\w=\w^{\ast}})_{i\in I(\w^{\ast})},
(\nabla_{(\w,\blambda)} \lambda_i|_{\blambda=\blambda^{\ast}})_{i\in I(\blambda^{\ast})}
\right]\in \R^{(n+r)\times (n+|I(\blambda^{\ast})|)}
$$
is of full-column rank.
Since  
the matrix
{\begin{align}
{\bm N}&:=
\begin{bmatrix}
{\rm zeros}(|I(\w^{\ast})|,n-|I(\w^{\ast})|)
&
{\bm E_{|I(\w^{\ast})|}}
& {\rm zeros}(|I(\w^{\ast})|,|I(\blambda^{\ast})|)
\\
\left(
\nabla_{(\tilde{\w},\blambda)}\Phi_i(\w^{\ast},\blambda^{\ast})
\right)_{i\notin I(\w^{\ast})}&
{\rm zeros}(n-|I(\w^{\ast})|+r,|I(\w^{\ast})|)&
(\nabla_{(\tilde{\w},\blambda)}\lambda_i|_{\blambda=\blambda^{\ast}})_{i\in I(\blambda^{\ast})}
\end{bmatrix}\notag \\
&\in \R^{(n+r)\times (n+|I(\blambda^{\ast})|)},\notag
\end{align}
where 
${\bm E}_s$ denotes the $s\times s$ identity matrix and
${\rm zeros}(s,t)$ stands for the zero matrix in $\R^{s\times t}$,}
is obtained by applying appropriate elementary column and row operations to $\bm M$,
we find that 
$\bm N$ is of full-column rank. Hence, 
the desired result is obtained.
\end{proof}

\noindent{{\bf Proof of Proposition~\ref{prop:0607}:}} For simplicity, let
$$
{\bxi}^k:=((\bzeta^k)^{\top},(\bbeta^k)^{\top})^{\top},\ \hat{\bzeta}^k:=\frac{\bzeta^k}{\|\bxi^k\|},\ \hat{\bbeta}^k:=\frac{{\bbeta^k}}{\|\bxi^k\|}
$$
for each $k$.
Suppose to the contrary that $\{{\bxi}^k\}$ is unbounded. 
Choosing an arbitrary accumulation point $(\w^{\ast},\blambda^{\ast})$
of the sequence $\{(\w^k,\blambda^k)\}$, 
without loss of generality,
we can assume that
$(\w^k,\blambda^k)\to (\w^{\ast},\blambda^{\ast})$ and $\|\bxi^k\|\to \infty$ as $k\to \infty$, if necessary, by taking a subsequence.
Let us denote
an arbitrary accumulation point of $\{\bxi^k/\|\bxi^k\|\}$ by 
$\hat{\bxi}^{\ast}:=((\hat{\bzeta}^{\ast})^{\top},\hat{\bbeta}^{\ast})^{\top}$,
where
$\hat{\bzeta}^{\ast}$ and $\hat{\bbeta}^{\ast}$ are accumulation points of $\{\hat{\bzeta}^k\}$ and $\{\hat{\bbeta}^{k}\}$, respectively.
Again, without loss of generality, we can suppose 
$
\lim_{k\to \infty}\hat{\bxi}^k=\hat{\bxi}^{\ast}.
$
Notice that $\|\hat{\bxi}^{\ast}\|=1$. 
By dividing both sides of \eqref{eq:1215-1}, \eqref{eqn:1217-1}, \eqref{eqn:1217-2},
and \eqref{eq:1215-3} with 
$\w = \w^k,\blambda = \blambda^k,\bzeta = \bzeta^k,\bbeta = \bbeta^k$ 
and $(\bm\varepsilon_1,\varepsilon_2,{\bm \varepsilon}_3,{\bm\varepsilon}_4,\varepsilon_5)
=(\bm\varepsilon_1^{k-1},\varepsilon_2^{k-1},{\bm \varepsilon}_3^{k-1},{\bm\varepsilon}_4^{k-1},\varepsilon_5^{k-1})
$
by $\|\bxi^k\|$, we have, for each $k$,
\begin{align}
&\frac{\left(\nabla f(\w^k)\right)_i}{\|\bxi^k\|}
+\left(\nabla_{\w\w}^2G(\w^{k},\bar{\blambda}^{k})\hat{\bzeta}^k\right)_{i}
+\lambda_1^k(\nabla^2\bphi_{\mu_k}(\w^k))_{ii}\hat{\zeta}_i^k=\frac{{({\bm \varepsilon}^{k-1}_1)_i}}{\|\bxi^k\|}
\ \ (i=1,2,\ldots,n), \label{eq:0605-1}\\
&\nabla\bphi_{\mu_k}(\w^k)^{\top}{\hat{\bzeta}^k}-{\hat{\eta}_1^k}=\frac{\varepsilon^{k-1}_2}{\|\bxi^k\|},\label{eq:0605-2}\\
&{\nabla R_i(\w^k)^{\top}\hat{\bzeta}^k-\hat{\eta}^k_i=\frac{(\bm \varepsilon_3^{k-1})_i}{\|\bxi^k\|}\ \ (i=2,3,\ldots,r)},\label{eq:0714-1708}\\
&{\lambda_i^k\hat{\eta}^k_i\le \frac{\varepsilon_5^{k-1}}{\|\bxi^k\|},\ \lambda_i^k\ge 0,\ \hat{\eta}^k_i\ge 0}\ 
(i=1,2,\ldots,r),
\label{eq:0605-3}
\end{align}
where the last conditions are deduced by componentwise decomposition of \eqref{eq:1215-3}.
Note that 
${\bm \varepsilon^{k-1}_1}/\|\bxi^k\|$,
${\varepsilon_2^{k-1}}/\|\bxi^k\|$, 
${\bm \varepsilon_3^{k-1}}/{\|\bxi^k\|}$,
and 
$
{\varepsilon_5^{k-1}}/{\|\bxi^k\|}
$
converge to 0 as $k\to \infty$.
By driving $k\to \infty$ in \eqref{eq:0605-3} 
for $i=1$ and using $\lim_{k\to \infty}\lambda^k_1=\lambda_1^{\ast}>0$ from Assumption~A1, 
we have 
\begin{equation}
\hat{\eta}^{\ast}_1=0. \label{eq:0606-2} 
\end{equation}
{In a similar manner, we can get 
\begin{equation}
{\hat{\eta}^{\ast}_i = 0\ \ (i\notin I(\blambda^{\ast}))},\label{eq:0715-057}
\end{equation}
where $I(\blambda^{\ast})=\{i\in \{1,2,\ldots,r\}\mid \lambda^{\ast}_i=0\}$ as is defined in Assumption~A4. }
Expressions \eqref{eq:0606-2} and \eqref{eq:0715-057} together with $\|\hat{\bxi}^{\ast}\|=1$, i.e., 
$\|\hat{\bxi}^{\ast}\|^2=\|\hat{\bzeta}^{\ast}\|^2+\sum_{i=1}^r|\hat{\eta}^{\ast}_i|^2=1$ imply
\begin{equation}
\|\hat{\bzeta}^{\ast}\|^2+\sum_{i\in I(\blambda^{\ast})}|\hat{\eta}^{\ast}_i|^2 = 1.\label{eq:0606-1}
\end{equation}

Next, let $k\to \infty$ in \eqref{eq:0605-1}. By the boundedness of $\{\nabla^2_{\w\w}G(\w^k,\bar{\blambda}^k){\hat{\bzeta}}^k\}$ and $\lim_{k\to \infty}\nabla f(\w^k)/\|\bxi^k\|=\bm 0$, 
we find that 
$\left\{\lambda_1^k\left(\nabla^2\bphi_{\mu_k}(\w^k)\right)_{ii}{\zeta}^k_{i}/\|\bxi^k\|\right\}$ is bounded for each $i$.
Using this fact,
$\lim_{k\to \infty}\lambda^k_1=\lambda_1^{\ast}>0$, and 
$\lim_{k\to \infty}|(\nabla^2\bphi_{\mu_k}(\w^k))_{ii}|\to \infty$
for $i\in I(\w^{\ast})$ by Proposition\,\ref{lem:0312}
yield 
\begin{equation}
\hat{\zeta}^{\ast}_i = 0\ \ (i\in I(\w^{\ast})).\label{eq:0605-4}
\end{equation}
We next show that 
\begin{equation}
\sum_{i\notin I(\w^{\ast})}{\rm sgn}(w_i^{\ast})|w_i^{\ast}|^{p-1}\hat{\zeta}^{\ast}_i=0. \label{eq:0607-1}
\end{equation}
For proving \eqref{eq:0607-1}, it suffices to show
\begin{equation}
\lim_{k\to \infty}\nabla \bphi_{\mu_{k-1}}(\w^k)^{\top}\hat{\bzeta^k}=\sum_{i\notin I(\w^{\ast})}{p}\,{\rm sgn}(w_i^{\ast})|w_i^{\ast}|^{p-1}\hat{\zeta}^{\ast}_i. \label{eq:0607-2}
\end{equation}
Indeed, we can derive \eqref{eq:0607-1} from \eqref{eq:0607-2} by taking the limit of \eqref{eq:0605-2}, \eqref{eq:0605-4}, and \eqref{eq:0606-2} into account.
Choose $i\in I(\w^{\ast})$ arbitrarily.
By Lemma\,\ref{prop:0312}, there exists some $\gamma>0$ such that 
\begin{equation}
\mu_{k-1}^2\ge \gamma|w^k_i|^{\frac{2}{2-p}}\label{eq:0312-1818-0607}
\end{equation}
for all $k$ sufficiently large.  In what follows, we consider sufficiently large $k$ so that inequality\,\eqref{eq:0312-1818-0607} holds.
Then, by $0< p\le 1$, we get 
\begin{equation*}
\frac{\mu_{k-1}^{2-p}}{\gamma^{\frac{2-p}{2}}}\ge |w^k_i|,\label{eq:0312-1818-2-0607}
\end{equation*}
which implies
\begin{align}
\frac{1}{p}(\nabla \bphi_{\mu_{k-1}}(\w^k))_i\hat{\zeta}_i^k
=&\left|w^k_i((w^k_i)^2+\mu_{k-1}^2)^{\frac{p}{2}-1}\hat{\zeta}^k_i\right|\notag \\
\le& \left|w^k_i\mu_{k-1}^{2(\frac{p}{2}-1)}\hat{\zeta}^k_i\right|\notag \\
\le& \frac{\mu_{k-1}^{2-p}}{\gamma^{\frac{2-p}{2}}}\mu_{k-1}^{2(\frac{p}{2}-1)}\left|\hat{\zeta}^k_i\right|\notag\\
=&\gamma^{\frac{p}{2}-1}\left|\hat{\zeta}^k_i\right|.\label{al:1216:2340-0607}
\end{align}
From relation\,\eqref{al:1216:2340-0607} and expression\,\eqref{eq:0605-4}
we obtain 
$\lim_{k\to \infty}(\nabla \bphi_{\mu_{k-1}}(\w^k))_i\hat{\zeta}_i^k = 0$.
Since $i\in I(\w^{\ast})$ was arbitrarily chosen, 
{it holds that 
\begin{equation}
\lim_{k\to\infty}\sum_{i\in I(\w^{\ast})}(\nabla \bphi_{\mu_{k-1}}(\w^k))_i\hat{\zeta}^k_i = 0.
\label{eq:0618-1}
\end{equation}
It then follows that 
\begin{align*}
\lim_{k\to\infty}\nabla \bphi_{\mu_{k-1}}(\w^k)^{\top}\hat{\bzeta^k}&=\lim_{k\to\infty}\left(\sum_{i\in I(\w^{\ast})}\left(\nabla \bphi_{\mu_{k-1}}(\w^k)\right)_i\hat{\zeta}^k_i
+\sum_{i\notin I(\w^{\ast})}\left(\nabla \bphi_{\mu_{k-1}}(\w^k)\right)_i\hat{\zeta}^k_i\right)     \\ 
      &=\lim_{k\to\infty}\sum_{i\notin I(\w^{\ast})}\left(\nabla \bphi_{\mu_{k-1}}(\w^k)\right)_i\hat{\zeta}^k_i\\
      &=\sum_{i\notin I(\w^{\ast})}{p}\,{\rm sgn}(w_i^{\ast})|w_i^{\ast}|^{p-1}\hat{\zeta}^{\ast}_i,                                                                             
\end{align*}
where 
the second equality follows from \eqref{eq:0618-1} and the last equality is due to the relation
\begin{equation}
\lim_{k\to\infty}(\nabla\bphi_{\mu_{k-1}}(\w^k))_i=p\,{\rm sgn}(w_i^{\ast})|w_i^{\ast}|^{p-1}\ \ (i\notin I(\w^{\ast})),\label{eq:0618-2}
\end{equation}
which can be derived from \eqref{al:1214-1}.}
{Therefore,} we conclude the desired expression \eqref{eq:0607-2} and thus \eqref{eq:0607-1}.
{In addition to \eqref{eq:0618-2},}
for $i\notin I(\w^{\ast})$, 
we obtain from \eqref{al:1214-2} that 
\begin{align*}
\lim_{k\to\infty}(\nabla^2\bphi_{\mu_{k-1}}(\w^k))_{ii} = p(p-1)|w_i^{\ast}|^{p-2}.
\end{align*}
Then, 
forcing $k\to \infty$ in \eqref{eq:0605-1} yields
\begin{align}
\left. \frac{\partial \left(\nabla_{{\w}}G(\w,\bar{\blambda})^{\top}\hat{\bzeta}^{\ast}\right)}
{\partial w_i}
\right|_{(\w,\blambda)=(\w^{\ast},\blambda^{\ast})}+\lambda_1^{\ast}p(p-1)|w_i^{\ast}|^{p-2}\hat{\zeta}_i^{\ast}=0 \ \ (i\notin I(\w^{\ast})),\notag 
\end{align}
which can be transformed by using \eqref{eq:0605-4} into
\begin{align}
&\left.
\frac{\partial \left(\sum_{j\notin I(\w^{\ast})}\frac{\partial G(\w,\bar{\blambda})}{\partial w_j}\hat{\zeta}^{\ast}_j\right)}{\partial w_i}
\right|_{(\w,\blambda)=(\w^{\ast},\blambda^{\ast})}
+\lambda_1^{\ast}p
\left.\frac{\partial \left(
\sum_{j\notin I(\w^{\ast})}{\rm sgn}(w_j)|w_j|^{p-1}
\hat{\zeta}_j^{\ast}
\right)}{\partial w_i}\right|_{\w=\w^{\ast}}\vspace{2em}\notag \\
&\hspace{0em}=0 \ \ (i\notin I(\w^{\ast})).\label{al:0607-1}
\end{align}
Put $\tilde{\w}:=(w_i)_{i\notin I(\w^{\ast})}$.
{Letting $k\to \infty$ in \eqref{eq:0714-1708}, we get
$\nabla R_i(\w^{\ast})^{\top}\hat{\bzeta}^{\ast}-\hat{\eta}^{\ast}_i=0\ \ (i=2,\ldots,r)$, 
which together with \eqref{eq:0605-4} implies 
\begin{equation}
\sum_{j\notin I(\w^{\ast})}\frac{\partial R_i(\w^{\ast})}{\partial w_j}\hat{\zeta}_j^{\ast}
-\hat{\eta}^{\ast}_i=0\ \ (i=2,\ldots,r).
\label{eq:0714-1740}
\end{equation}}
{Now, let $\bm{\Psi}^{\ast}:=(\Psi^{\ast}_i)_{i\notin I(\w^{\ast})}^{\top}\in \R^{n-|I(\w^{\ast})|}$ with 
\begin{equation}
\Psi_i^{\ast}:=
\left.
\frac{\partial \left(\sum_{j\notin I(\w^{\ast})}\frac{\partial G(\w,\bar{\blambda})}{\partial w_j}\hat{\zeta}^{\ast}_j\right)}{\partial w_i}
\right|_{(\w,\blambda)=(\w^{\ast},\blambda^{\ast})}
+\lambda_1^{\ast}p
\left.\frac{\partial \left(
\sum_{j\notin I(\w^{\ast})}{\rm sgn}(w_j)|w_j|^{p-1}
\hat{\zeta}_j^{\ast}
\right)}{\partial w_i}\right|_{\w=\w^{\ast}}
\label{al:0607-1-2}
\end{equation}
and 
${\bm e}^j\in \R^{r}$ be the vector such that the $j$-th element is 1 and others are 0s. 
In addition, $\Phi_i$ $(i\notin I(\w^{\ast}))$, $\nabla_{(\tilde{\w},\blambda)} \Phi_i$ $(i\notin I(\w^{\ast}))$, and $\nabla_{(\tilde{\w},\blambda)}\lambda_i\ (i\in I(\blambda^{\ast}))$
are the functions defined in Assumption~A4, \eqref{eq:0713-1}, and \eqref{al:0714-2328} in Lemma\,\ref{lem:0607},
respectively.}
{Then, it follows that 
\begin{align}
&\sum_{j\notin I(\w^{\ast})}\nabla_{(\tilde{\w},\blambda)} \Phi_j(\w^{\ast},\blambda^{\ast})\hat{\zeta}_j^{\ast}
-
\sum_{j\in I(\blambda^{\ast})} \nabla_{(\tilde{\w},\blambda)} \lambda_j|_{\blambda=\blambda^{\ast}}\hat{\eta}^{\ast}_j\notag \\
&=
\sum_{j\notin I(\w^{\ast})}\begin{bmatrix}
\nabla_{\tilde{\w}}\Phi_j(\w^{\ast},\blambda^{\ast})\\
\nabla_{\blambda}\Phi_j(\w^{\ast},\blambda^{\ast})
\end{bmatrix}\hat{\zeta}_j^{\ast}
-
\sum_{j\in I(\blambda^{\ast})}
\begin{bmatrix}
{\rm zeros}(n-|I(\w^{\ast})|,1)
\\
\hat{\eta}_j^{\ast}
\bm e^j
\end{bmatrix}
\notag \\
&=
\begin{bmatrix}
\left(
\left.
\frac{\partial 
\left(\sum_{j\notin I(\blambda^{\ast})}
\Phi_j(\w,\blambda) \hat{\zeta}_j^{\ast}
\right)
}{\partial w_i}
\right|_{(\w,\blambda)=(\w^{\ast},\blambda^{\ast})}
\right)_{i\notin I(\w^{\ast})}^{\top}\\
-\hat{\bbeta}^{\ast}+
\sum_{j\notin I(\w^{\ast})}
\nabla_{\blambda}\Phi_j(\w^{\ast},\blambda^{\ast})\hat{\zeta}_j^{\ast}
\end{bmatrix}\notag \\
&=
\begin{bmatrix}
{\bm \Psi}^{\ast}\\
\sum_{j\notin I(\w^{\ast})}\hat{\zeta}_j^{\ast}\left(p\,{\rm sgn}(w_j^{\ast})|w_j^{\ast}|^{p-1}\right)\\
\sum_{j\notin I(\w^{\ast})}
\frac{\partial R_2(\w^{\ast})}{\partial w_j}
\hat{\zeta}_j^{\ast}-\hat{\eta}_2^{\ast}
\\
\vdots\\
\sum_{j\notin I(\w^{\ast})}
\frac{\partial R_r(\w^{\ast})}{\partial w_j}
\hat{\zeta}_j^{\ast}-\hat{\eta}^{\ast}_r
\end{bmatrix}\notag \\                                      
&=\bm 0,\label{eq:0615-2}
\end{align}
where ${\rm zeros}(n-|I(\w^{\ast})|,1)$ denotes the zero matrix in $\R^{n-|I(\w^{\ast})|}$,
the second equality follows from \eqref{eq:0715-057},
the third one is from 
\eqref{eq:0606-2}, definition\,\eqref{al:0607-1-2} of $\bm{\Psi}^{\ast}$, and easy calculation, and the last one is derived from \eqref{eq:0607-1}, \eqref{eq:0714-1740}, and \eqref{al:0607-1-2}.
Expression\,\eqref{eq:0615-2} together with Lemma~\ref{lem:0607}
entails $\hat{\zeta}^{\ast}_i=0\ ({i\notin I(\w^{\ast})})$ and $\hat{\eta}^{\ast}_i=0\ (i\in I(\blambda^{\ast}))$.
Hence, by \eqref{eq:0605-4}, we obtain $\|\hat{\bzeta}^{\ast}\|^2+\sum_{i\in I(\blambda^{\ast})}|\hat{\eta}^{\ast}_i|^2=\bm 0$.
However, it contradicts \eqref{eq:0606-1}.}
Therefore, the sequence $\{(\bzeta^k,\bbeta^k)\}$ is bounded.
\\
\hfill$\blacksquare$}

\renewcommand{\theequation}{B.\arabic{equation}}
\renewcommand{\thealgorithm}{B.\arabic{algorithm}}
\renewcommand{\thetheorem}{B.\arabic{theorem}}
\def\thesection{\Alph{section}}
\setcounter{equation}{0}
\setcounter{theorem}{0}
\setcounter{algorithm}{0}

{\section{Description of the algorithm for solving the smoothed problem\,\eqref{eqn:smoothprob}.}\label{appendix:B}
\subsection{Implicit function based method}\label{appendix:B1}
In this section, we describe the algorithm that is used for solving the following problem arising by smoothing problems\,\eqref{al:0913} and \eqref{al:0913-2} in the numerical experiments in Section~\ref{sec:numeric}:
\begin{align}\label{al:0913:smoothed}
\begin{array}{cll}
\displaystyle{\min_{(\bw,\blambda)\in \R^n\times \R^{n+1}}}&\ &f_{\rm val}(\w):=\|\bA_{\rm val}{\w}-\bm b_{\rm val}\|_2^2 \\ 
\mbox{s.t. }&  &\bw \in\displaystyle{\mathop{\rm argmin}_{\hat{\bw}}}\ \left\{\phi_{\mu}(\hat{\w},\blambda):=\|\A_{\rm tr}\hat\w-\bm b_{\rm tr}\|^2+e^{\lambda_1} \sum_{i=1}^n(\hat{w}_i^2+\mu^2)^{\frac{p}{2}}+\nu \hwCw\right\},
\end{array}
\end{align}
where 
$\nu\in \{0,1\}$ and $\C(\bm \bar{\blambda}):={\rm Diag}(\exp(\lambda_i))_{i=2}^{n+1}$.
The above problems with $\nu=0$ and $1$ correspond to problems\,\eqref{al:0913} and \eqref{al:0913-2}, respectively.
Our goal is to compute a KKT triplet $(\w,\blambda,\bbeta)\in \R^n\times \R^{n+1}\times \R^n$ of the above problem with the constraint replaced by the equality constraint $\nabla_{\w}\phi_{\mu}(\w,\blambda)=\0$.
Namely, we compute $(\w,\blambda,\bbeta)$ which satisfies
\begin{align}
\Theta(\w,\bbeta):=
\begin{bmatrix}\nabla f_{\rm val}(\w)\\ \0
\end{bmatrix}  +
\begin{bmatrix}
\nabla^2_{\w\w}\phi_{\mu}(\w,\blambda)\\
\nabla^2_{\w\blambda}\phi_{\mu}(\w,\blambda)
\end{bmatrix}\bbeta=\0,\ \nabla_{\w}\phi_{\mu}(\w,\blambda)=\0,
\label{al:1906-0709}
\end{align}
where $\nabla^2_{\w\blambda}\phi_{\mu}(\w,\blambda)=\nabla_{\blambda}\left(\nabla_{\w}\phi_{\mu}(\w,\blambda)\right)\in \R^{{(n+1)}\times n}$.

Given $\widetilde{\blambda}$ and $\mu$, let $\widetilde{\w}$ be a stationary point of the smoothed lower-level problem  
$\min_{{\w}}\phi_{\mu}({\w},\blambda)$.
According to the standard implicit function theorem,
if $\nabla^2_{\w\w}\phi_{\mu}(\widetilde{\w},\widetilde{\blambda})$ is of full rank,
there exist some open neighborhood $U_{\widetilde{\bm{\lambda}}}$ of $\widetilde{\bm{\lambda}}$ and a twice continuously differentiable implicit function $\bm{w}(\cdot):U_{\widetilde{\bm{\lambda}}}\to \R^n$ such that 
\begin{align*}
\widetilde{\bm{w}}=\bm{w}(\widetilde{\bm{\lambda}}),\ \nabla_{\w}\phi_{\mu}(\bm{w}(\bm{\lambda}),\bm{\lambda}) = \bm{0}\ \ (\bm{\lambda}\in U_{\widetilde{\bm{\lambda}}}).
\end{align*}
In $U_{\widetilde{\bm{\lambda}}}$, we may regard problem\,\eqref{al:0913:smoothed}
with the constraint replaced by
$\nabla_{\w}\phi_{\mu}(\w,\blambda)=\0$
as
\begin{equation}
\min_{\blambda\in U_{\widetilde{\bm{\lambda}}}}\ \left\{F(\blambda):=\|\bm A_{{\mathrm {val}}}\bm{w}({\bm{\lambda}})-\bm{b}_{{\mathrm {val}}}\|_2^2\right\}.
\label{eq:0406-1}
\end{equation}
By the implicit function theorem again, we then have 
$$
\nabla \bm{w}(\widetilde{\blambda})=-\nabla_{\w\blambda}^2\phi_{\mu}(\w(\widetilde{\blambda}),\widetilde{\blambda})\left(\nabla^2_{\bm{w}\w}\phi_{\mu}(\bm{w}(\widetilde{\blambda}),\widetilde{\blambda}) \right)^{-1},
$$
and hereby the gradient of the objective of problem~\eqref{eq:0406-1} at $\widetilde{\blambda}$ is expressed as follows: 
\begin{align*}
\nabla F(\widetilde{\blambda})
&=\nabla_{\blambda} \|\A_{{\mathrm {val}}}\bm{w}({\widetilde{\blambda}})-\bm{b}_{{\mathrm {val}}}\|_2^2\notag\\
&=  2\nabla \bm{w}(\widetilde{\blambda})
\A_{{\mathrm {val}}}^{\top} \left(\A_{{\mathrm {val}}}\bm{w}({\widetilde{\blambda}})-\bm{b}_{{\mathrm {val}}}\right)
\notag\\
&= -2
\nabla_{\w\blambda}^2\phi_{\mu}(\w(\widetilde{\blambda}),\widetilde{\blambda})\left(\nabla^2_{\bm{w}\w}\phi_{\mu}(\bm{w}(\widetilde{\blambda}),\widetilde{\blambda}) \right)^{-1}\A_{{\mathrm {val}}}^{\top} \left(\A_{{\mathrm {val}}}\bm{w}({\widetilde{\blambda}})-\bm{b}_{{\mathrm {val}}}\right).\label{eq:0406-1-2}   
\end{align*}
By computing the above gradient at each iterate, we can preform the quasi-Newton method\,\citep{nocedal2006numerical} for problem\,\eqref{eq:0406-1} to have a solution $\blambda^*$ with $\nabla F(\blambda^*)=\0$. Once $\blambda^*$ is gained together with $\w(\blambda^*)$,
we substitute them into the first equation $\Theta(\w,\bbeta)=\0$ in \eqref{al:1906-0709} and solve the resultant linear equation
$\Theta(\w(\blambda^*),\bbeta)=\0$ for $\bbeta$ to have a solution, say $\bbeta^*$.
The triplet $(\w(\blambda^*),\blambda^*,\bbeta^*)$ is then nothing but the desired KKT triplet.

The overall algorithm is described as in Algorithm\,\ref{alg_subproblem0}.
\begin{algorithm}[tb]
\caption{Implicit function based quasi-Newton method for the smoothed subproblem}                                  
\label{alg_subproblem0}  
\begin{algorithmic}[1]                  
\REQUIRE $\bm{\lambda}^0\in\mathbb{R}^{\dim}$, $\epsilon \geq 0$, $\alpha,\beta\in (0,1)$, ${\bm B}_0\in S^{\dim}_{++}$ ($S^{\dim}_{++}$: The set of $({\dim})\times({\dim})$ real symmetric positive definite matrices);
Set $k\gets0$.
\WHILE{$\nabla F(\blambda^k)\ge \epsilon$}
\STATE 
Find $\bm{w}^k$ satisfying $\nabla_{\w}\phi_{\mu}(\w^k,\blambda^k)=0$. 
\STATE 
Set ${\bm d}_{\blambda}\gets -{\bm B}_k^{-1} \nabla F(\blambda^k)$.
\STATE Find the smallest integer $\ell_k\ge 0$ satisfying 
\begin{equation}
F(\blambda^k+\beta^{\ell_k}{\bm d}_{\blambda})\le  F(\blambda^k) +\alpha \beta^{\ell_k}\nabla F(\blambda^k)^{\top}{\bm d}_{\blambda}.
\label{eq:0407-0124}
\end{equation}
Set $t_k\gets \beta^{\ell_k}$.
\STATE $\bm{\lambda}^{k+1} \gets \bm{\lambda}^{k}+t_k {\bm d}_{\blambda}$.
\STATE Set ${\bm B}_{k+1}\in S^{\dim}_{++}$.
\STATE $k\gets k+1$
\ENDWHILE
\STATE  Set $(\bar{\w},\bar{\blambda})\gets (\w^k,\blambda^k)$.
\STATE  Solve $\Theta(\bar{\w},\bar{\blambda},\bbeta)=\0$ for $\bbeta$ to obtain a Lagrange multiplier $\bar{\bbeta}$.
\OUTPUT $(\bar{\w},\bar{\blambda},\bar{\bbeta})$
\end{algorithmic}
\end{algorithm}
For the algorithm to work, 
the full-rankness of $\nabla^2_{\w\w}\phi_{\mu}(\widetilde{\w},\widetilde{\blambda})$ is necessary to ensure the existence of the implicit function
$\w(\cdot)$. This is expected to hold in many instances, although it cannot be guaranteed generally. 
We must solve the lower-level problem in Line~2 every time $\ell_k$ is updated while performing linesearch\,\eqref{eq:0407-0124}, 
and thus how we solve the smoothed lower-level problem affects the overall efficiency of Algorithm\,\ref{alg_subproblem0}. In the subsequent section, we will present a certain Newton-type method for solving the smoothed lower-level problem.

{Next, we make a remark on the linesearch procedure in Algorithm\,\ref{alg_subproblem0}. 
As mentioned previously, we need to solve the smoothed lower-level problem $\min_{{\w}}\phi_{\mu}({\w},\blambda^k+\beta^{\ell_k}{\bm d}_{\blambda})$
so as to evaluate $F(\blambda^k+\beta^{\ell_k}{\bm d}_{\blambda})$ every time $\ell_k$ is incremented. 
Actually, to compute $F(\blambda^k+\beta^{\ell_k}{\bm d}_{\blambda})$, we need to know
the value of $\w(\blambda^k+\beta^{\ell_k}{\bm d}_{\blambda})$ by solving the equation
$\nabla_{\w}\phi_{\mu}({\w},\blambda^k+\beta^{\ell_k}{\bm d}_{\blambda})=\bm 0$. 
However, the smoothed lower-level problem $\min_{{\w}}\phi_{\mu}({\w},\blambda^k+\beta^{\ell_k}{\bm d}_{\blambda})$ is nonconvex when $p<1$ and thus the set of
solutions of $\nabla_{\w}\phi_{\mu}({\w},\blambda^k+\beta^{\ell_k}{\bm d}_{\blambda})=\bm 0$ is not singleton in general\footnote{When $p=1$,
$\min_{{\w}}\phi_{\mu}({\w},\blambda^k+\beta^{\ell_k}{\bm d}_{\blambda})$ is strongly convex minimization in virtue of the term
$\sum_{i=1}^n(w_i^2+\mu^2)^{\frac{p}{2}}$ with $\mu>0$, and thus its solution set is singleton.}.
This fact yields that} applying a numerical method to this equation may not return $\w(\blambda^k+\beta^{\ell_k}{\bm d}_{\blambda})$.
Nevertheless, in practice, we expect $\w(\blambda^k+\beta^{\ell_k}{\bm d}_{\blambda})$ to be computed successfully
by applying, e.g., a Newton-type method with $\w(\blambda^k)$ as a starting point to the equation, because
$\w(\blambda^k+\beta^{\ell_k}d_{\blambda})$ actually gets closer to $\w(\blambda^k)$ as $\ell_k$ is increased in the linesearch procedure.}

The convergence analysis of Algorithm\,\ref{alg_subproblem0} can be mostly done in a manner similar to that of the standard quasi-Newton method. 
Indeed, we can show that any accumulation point of $\{(\w^k,\blambda_k)\}$
is a KKT point of the smoothed subproblem under the following two sets of assumptions: 
\begin{assumption} Let $\{\bm{\lambda}^{k}\}$ be a sequence produced by Algorithm \ref{alg_subproblem0}. Then, the following properties hold:
\begin{enumerate}
\item The sequence $\{\bm{\lambda}^k\}$ is bounded.
\item There exist some $\alpha_1, \alpha_2\ (0<\alpha_1\leq\alpha_2)$ such that 
\begin{align}
\alpha_1 {\bm E}\preceq \B_k \preceq \alpha_2 {\bm E} 
\end{align}
for all $k$, where ${\bm E}$ is the identity matrix with the same size with $\B_k$, and for symmetric matrices ${\bm X},{\bm Y}$, ${\bm X}\preceq \bm Y$ stands for ${\bm Y}-{\bm X}$ is positive semidefinite.
\end{enumerate}
\label{assumption_1}

\end{assumption}
The above assumptions are often made in convergence analysis of the quasi-Newton method, whereas the following assumption is specific to our setting.

\begin{assumption}\label{assumption_2}
$\nabla_{\bm{w}\bm{w}}^2\phi_{\mu}({\bm{w}^k},{\bm{\lambda}^k})$ is of full rank for each $k$, and
so is $\nabla_{\bm{w}\bm{w}}^2\phi_{\mu}({\bm{w}^{\ast}},{\bm{\lambda}^{\ast}})$ even at an arbitrary accumulation point $(\bm{w}^{\ast},\bm{\lambda}^{\ast})$.
\end{assumption}
Assumption~\ref{assumption_2} ensures that the implicit function $\w(\cdot)$ exists at each iterate and even at an arbitrary accumulation point.

The following theorem holds under Assumptions \ref{assumption_1} and \ref{assumption_2}.
As the proof is similar to that for the quasi-Newton method, we omit it here.
\begin{theorem}
Suppose that Assumptions \ref{assumption_1} and \ref{assumption_2} hold. 
Then, any accumulation point of $\{\bm{\lambda}^{k}\}$ satisfies
$\nabla F({\bm{\lambda}}) =\bm{0}$.
\end{theorem}

\subsection{Newton-type method for solving the smoothed lower-level problem}\label{appendix:B2}
In this section, we describe the modified Newton-type algorithm used for solving the smoothed lower-level problem $\min_{\w}\phi_{\mu}(\w,\blambda)$ in problem\,\eqref{al:0913:smoothed}. 
For brevity, the algorithm is presented in the form pertaining to the following problem:
\begin{equation}
\min_{\w}\ \psi_{\mu}(\w):=\frac{1}{2}\|\bC\w-\bff\|^2 + \eta\sum_{i=1}^n(w_i^2+\mu^2)^{\frac{p}{2}},
\label{eq:0414-1}
\end{equation}  
where $\eta\in \R$ is positive, $\bC\in \R^{m\times n}$, and $\bff\in \R^m$.
Note that by setting $\bC$ and $\bff$ appropriately, the function $\psi_{\mu}$ above reduces to $\phi_{\mu}$.

We begin with the update-formula of the standard Newton method for problem~\eqref{eq:0414-1} at the $r$-th iterate $\w^r\in \R^n$:
\begin{align*}
&\w^{r+1}\leftarrow \w^r - \bm B(\w^r)^{-1}\nabla\psi_{\mu}(\w^r),\ \mbox{where}\\
&\bm B(\w):=
\bC^{\top}\bC+
p\eta{\rm Diag}\left(
(w_i^2+\mu^2)^{\frac{p}{2}-1}
+
\underbrace{\frac{p-2}{2}w_i^2(w_i^2+\mu^2)^{\frac{p}{2}-2}}_{\mbox{negative}}
\right)_{i=1}^n.
\end{align*}
However, the matrix $\B(\w^r)$ is not necessarily nonsingular because of the above negative part, and thus 
the Newton method may not work.\footnote{
In fact, when $p=1$, $\B(\w)$ is nonsingular even in the presence of the negative part, because
\begin{align}
p\eta{\rm Diag}\left(
(w_i^2+\mu^2)^{\frac{p}{2}-1}
+
\frac{p-2}{2}w_i^2(w_i^2+\mu^2)^{\frac{p}{2}-2}
\right)_{i=1}^n
&=p\eta{\rm Diag}\left(
\left(\mu^2+\frac{p}{2}w_i^2\right)(w_i^2+\mu^2)^{\frac{p}{2}-1}
\right)_{i=1}^n,
\end{align}
which turns out to be positive definite.
} 
As a remedy, in the spirit of the modified Newton method, we modify $\B(\w^r)$ to the following matrix $\widetilde{\B}(\w^r)$ by deleting the negative part:
$$
{\widetilde{\B}(\w):=\bC^{\top}\bC + p\eta {\rm Diag}\left({(w_i^2+\mu^2)^{\frac{p}{2}-1}}\right)_{i=1}^n}.
$$
Now, the presented algorithm is described formally as in Algorithm~\ref{alg_subproblem}.
In fact, the algorithm is identical to the one that is proposed by \citet[Section~2]{lai2011unconstrained}, which gives the following theorem:
\begin{theorem}{\citep[Theorem~2.1]{lai2011unconstrained}}
Let $\{\w^r\}$ be a sequence generated by Algorithm\,\ref{alg_subproblem} with $\epsilon=0$.
It is bounded and its arbitrary accumulation point satisfies $\nabla \psi_{\mu}(\w)=\bm 0$.
\end{theorem}
It is worthwhile to note that Algorithm\,\ref{alg_subproblem} does not request a linesearch procedure for the global convergence, which is often costly. 

\begin{algorithm}
\caption{Modified Newton-type method for $\min_{\w}\psi_{\mu}(\w)$}
\label{alg_subproblem}
\begin{algorithmic}[1]
\REQUIRE $\w^0\in \R^n$, $r\leftarrow0$, $\epsilon>0$   
\WHILE{$\|\nabla\psi_{\mu}(\w^r)\|>\epsilon$} 
\STATE $\w^{r+1} \leftarrow \w^r - \widetilde{\bm B}(\w^r)^{-1}\nabla \psi_{\mu}(\w^r)$,
\STATE $r \leftarrow r+1$.
\ENDWHILE
\STATE Set $\bar{\w}\gets \w^r$
\OUTPUT $\bar{\w}$
\end{algorithmic}
\end{algorithm}

\section{Supplementary tables and figures of \texttt{bayesopt} for the numerical experiments}\label{sec:appC}
This section provides the supplementary Tables~\ref{tbl1:appendix} and \ref{tbl2:appendix} 
that show the first time when \texttt{bayesopt} 
found the best observed objective value.
These results were recorded in a single run of \texttt{bayesopt} for each problem, thus differ from the averaged results shown in Tables\,\ref{tbl:comparison_full} and \ref{tbl2:comparison_full}. 
In addition, it also gives Figures\,\ref{Fig;0618-1} and \ref{Fig;0618-2} that depict how the best observed objective value of \texttt{bayesopt} varies over time. In order to monitor the change of values in a long period, 
we extended the time limit of \texttt{bayesopt} to 1200 seconds from 600 seconds that was employed for making Tables~\ref{tbl1:appendix} and \ref{tbl2:appendix}.  
These figures were obtained by solving the problems organized from the data sets of {\bf CpuSmall} and {\bf Student}. 
\begin{table}[H]
\centering
\caption{
The first time of \texttt{bayesopt} for finding a solution of problem\,\eqref{al:0913}, which attains the final best observed objective value, i.e., validation value
(Those results of \texttt{bayesopt} were recorded in a single run, thus differ from the averaged results over 5 runs shown in Table\,\ref{tbl:comparison_full}.
For the sake of comparison, the results of Algorithm~\ref{alg:smoothing} are also shown, which are the same as those in Table\,\ref{tbl:comparison_full}.
``{\ftime}'' stands for the first time in seconds where 
the best objective value is observed.  
The best values in {\ftime} and $\mbox{Err}_{\rm val}$ are displayed in bold.)
}
\label{tbl1:appendix}
	{
\scalebox{1}{
\begin{tabular}{crrrrr}
\toprule
\multicolumn{2}{c}{Data}&
\multicolumn{2}{c}{\texttt{bayesopt}}&
\multicolumn{2}{c}{{\Alg}}\\
\cmidrule(lr){1-2}
\cmidrule(lr){3-4}
\cmidrule(lr){5-6}
\multicolumn{1}{c}{name}&
\multicolumn{1}{c}{$p$}&
\multicolumn{1}{c}{$\mbox{Err}_{\rm val}$}&
\multicolumn{1}{c}{\ftime}&
\multicolumn{1}{c}{$\mbox{Err}_{\rm val}$}&
\multicolumn{1}{c}{time (sec)}\\
\midrule
{\bf Facebook} &1 	&	 6.476& 42.256&  {\bf 6.474}&  {\bf 17.399} \\
& 0.8&{\bf 6.504}&122.158&  6.512&  {\bf 22.242}\\
&0.5&{\bf 6.536}
&66.589&6.550 & {\bf 16.820} \\ \hline
{\bf Insurance} & 1&{\bf 95.764}&49.538 &{\bf 95.764}  &	{\bf 33.077} \\
& 0.8	&   95.737&   63.580    &	   {\bf 95.676} 	&	{\bf 32.465} 	\\
	        & 0.5 	&95.604		&{\bf 13.960}    &   {\bf 95.562} 	&	44.904  \\ \hline

{\bf Student} 	&  1    &   {\bf 0.777} &12.405      &   0.778 &	{\bf 10.586}    \\
	        & 0.8 	&{\bf 0.724}	&339.980	&  {\bf 0.724} &	{\bf 2.348} 	  \\

        & 0.5	&0.731	        &147.118	&    {\bf 0.724} &	 {\bf 3.618}\\ \hline
{\bf BodyFat} 	&   1	&{\bf 0.209}   & 4.839	& {\bf 0.209} 	&	{\bf 0.068} \\
                &  0.8	&  0.180 &3.899	& {\bf 0.179} 	&	{\bf 0.203} \\
&  0.5 &  {\bf 0.212}	&1.160 & 0.267 &{\bf 0.395} \\ \hline
{\bf CpuSmall} &1 	&    131124&	{\bf 1.202}&  {\bf 130981} 	&	11.299 \\
& 0.8 	&    131187&1.853	& {\bf 130982}&	{\bf 0.741} 	\\
	    & 0.5	&  131234&1.712	& {\bf 131058} 	&	{\bf 0.672} 	\\ 
\bottomrule
\end{tabular}}}
\end{table}

\begin{figure}[htbp]
 \caption{\redcolor
 Best observed objective value (validation value) vs running time in seconds (\texttt{bayesopt} for problem\,\eqref{al:0913} with a single $\ell_{0.8}$ hyperparameter)}
 \label{Fig;0618-1}
 \centering
\includegraphics[scale = 0.8]{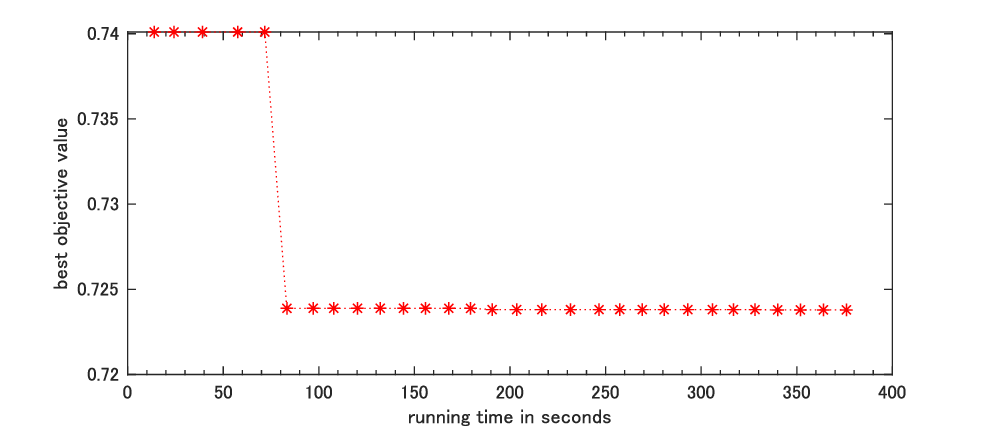}
 \subcaption{{\bf Student} (The number of hyperparameters is 1;
the proposed bilevel algorithm found a solution with 0.724 in 2 seconds.)}
\hspace{0em}\includegraphics[scale = 0.8]{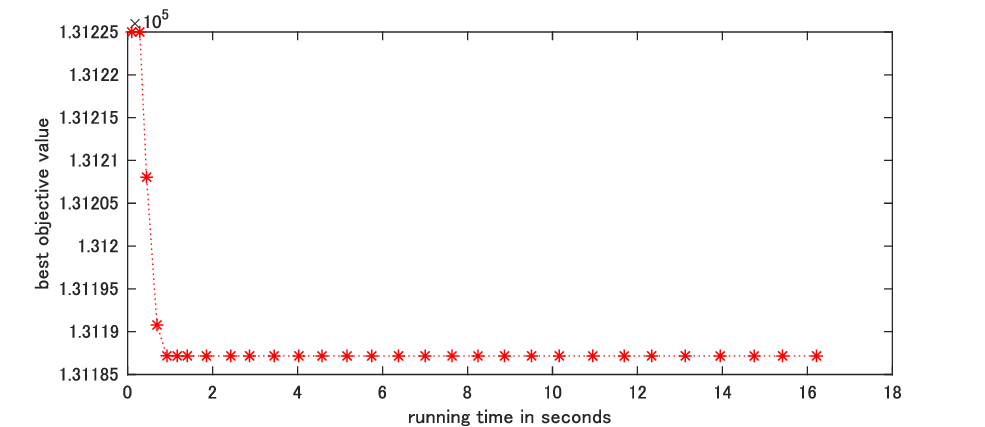}
\subcaption{{\bf CpuSmall} (The number of hyperparameters is 1; the proposed bilevel algorithm found a solution with $1.3098\times 10^5$ in 0.7 seconds.)}
\end{figure}
\begin{table}[H]
\centering
\caption{
\redcolor The first time of \texttt{bayesopt} for finding a solution of problem\,\eqref{al:0913-2}, which attains the final best observed objective value, i.e., validation value
(Those results of \texttt{bayesopt} were recorded in a single run, thus differ from the averaged results over 5 runs shown in Table\,\ref{tbl2:comparison_full}.
For the sake of comparison, the results of Algorithm~\ref{alg:smoothing} ({\algone}, {\algtwo}) are also shown, which are the same as those in Table\,\ref{tbl2:comparison_full}.
``{\ftime}'' stands for the first time in seconds where 
the best objective value is observed.  
The best values in {\ftime} and $\mbox{Err}_{\rm val}$ are displayed in bold.)
}
\label{tbl2:appendix}
	{
\scalebox{1}{
\begin{tabular}{crrrrrrr}
\toprule
\multicolumn{2}{c}{Data}&
\multicolumn{2}{c}{\texttt{bayesopt}}&
\multicolumn{2}{c}{{\algone}}&
\multicolumn{2}{c}{{\algtwo}}\\
\cmidrule(lr){1-2}
\cmidrule(lr){3-4}
\cmidrule(lr){5-6}
\cmidrule(lr){7-8}
\multicolumn{1}{c}{name}&
\multicolumn{1}{c}{$\sharp\blambda$}&
\multicolumn{1}{c}{$\mbox{Err}_{\rm val}$}&
\multicolumn{1}{c}{\ftime}&
\multicolumn{1}{c}{$\mbox{Err}_{\rm val}$}&
\multicolumn{1}{c}{time (sec)} &
\multicolumn{1}{c}{$\mbox{Err}_{\rm val}$}&
\multicolumn{1}{c}{time (sec)}\\
\midrule
{\bf Facebook}				&	54
&	8.780 	&	{\bf 1.518}  
&	{\bf 6.478} 	&	20.247	 	
&	--&	--
\\\hline
{\bf Insurance}					& 86
&	107.000 	&	{\bf 4.288}
&	95.694 	&	{50.714}  	
&	{\bf 94.604} 	&	{4.473}
\\\hline
{\bf Student}				& 273	
&	18.324 	&	26.756 
&	{\bf 0.771} 	&	{\bf 1.451} 
&	0.786 	&	71.775 	\\\hline
{\bf BodyFat}
& 15  
&	46.815 	&	0.337
&	0.243 	&	{\bf 0.072} 	
&	{\bf 0.130} 	&	0.695 \\\hline
{\bf CpuSmall}					
& 13
&	151394 	&	531.837  
&	131130 	&	6.222 
&	{\bf 128540} 	&	{\bf 0.658} \\
\bottomrule
\end{tabular}}}
\end{table}

 \begin{figure}[htbp]
 \caption{\redcolor Best observed objective value (validation value) vs running time in seconds (\texttt{bayesopt} for problem\,\eqref{al:0913-2} with multiple hyperparameters)}
 \label{Fig;0618-2}
 \centering
 \includegraphics[scale = 0.8]{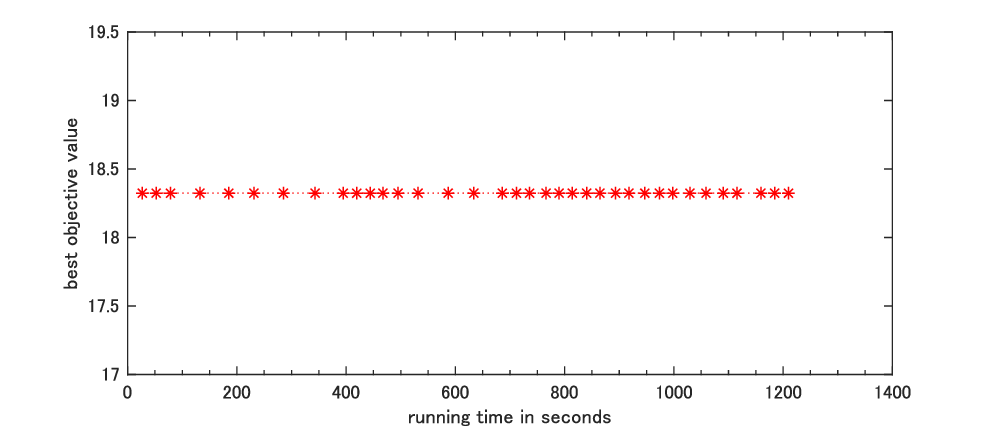}
 \subcaption{{\bf Student} (The number of hyperparameters is 273;
the proposed bilevel algorithms
found solutions with 0.771 in 1.45 seconds (Alg.1-A)
and 0.786 in 71.78 seconds (Alg.1-B).)}
\includegraphics[clip,scale = 0.45]{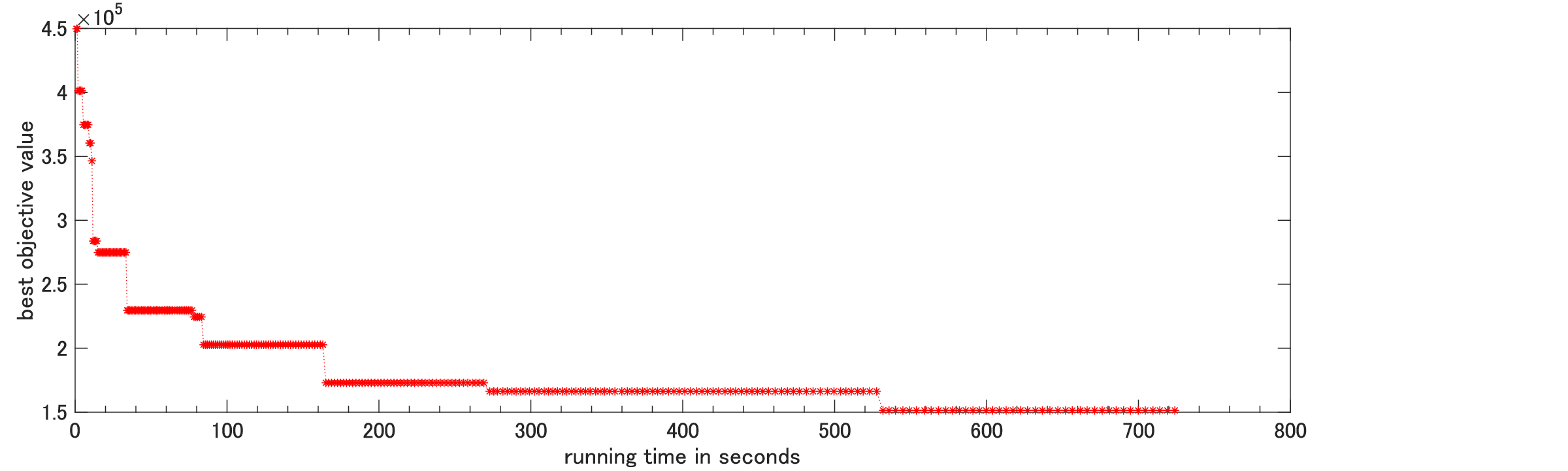}\hspace{0cm}
\subcaption{{\bf CpuSmall} (The number of hyperparameters is 13; the proposed bilevel algorithms found solutions with $1.31\times 10^5$ in 6.22 seconds (Alg.1-A)
and $1.29\times 10^5$ in 0.66 seconds (Alg.1-B).)
}\end{figure}}

\bibliography{bilevelparam}
\end{document}